\documentclass[a4paper,5p,number,sort&compress,times,]{elsarticle}
\usepackage{amsmath,amsfonts,amssymb,amsthm}
\usepackage{graphicx,psfrag,subfigure}
\usepackage[usenames,dvipsnames]{color}
\usepackage{epsfig,rotating,lscape}

\newtheorem{thm}{Theorem}
\newtheorem{prop}[thm]{Proposition}
\newtheorem{lem}[thm]{Lemma}

\newtheorem{numres}[thm]{Numerical Result}
\newtheorem{conj}[thm]{Conjecture}
\newdefinition{remark}{Remark}
\newdefinition{example}{Example}

\newcommand{\Argsinh}{\mathop{\rm argsinh}\nolimits}
\newcommand{\Order}{\mathop{\rm O}\nolimits}
\newcommand{\order}{\mathop{\rm o}\nolimits}

\newcommand{\rmd}{\,{\rm d}}
\newcommand{\rme}{{\rm e}}
\newcommand{\rmi}{\mathrm{i}}
\newcommand{\am}{\mathop{\rm am}\nolimits}
\newcommand{\sn}{\mathop{\rm sn}\nolimits}
\newcommand{\cn}{\mathop{\rm cn}\nolimits}

\newcommand{\Length}{\mathop{\rm Length}\nolimits}
\newcommand{\LengthSpectrum}{\mathop{\mathcal{LS}}\nolimits}

\newcommand{\Nset}{\mathbb{N}}
\newcommand{\Qset}{\mathbb{Q}}
\newcommand{\Rset}{\mathbb{R}}
\newcommand{\Tset}{\mathbb{T}}
\newcommand{\Cset}{\mathbb{C}}
\newcommand{\Zset}{\mathbb{Z}}

\newif\iffigures
\figurestrue

\begin{document}

\begin{frontmatter}

\title{Exponentially small asymptotic formulas for
the length spectrum in some billiard tables\tnoteref{t1}}

\author[ma4]{P.~Mart{\'{\i}}n}
\ead{martin@ma4.upc.edu}

\author[ma1]{R.~Ram{\'\i}rez-Ros\corref{cor}}
\ead{rafael.ramirez@upc.edu}

\author[ma1]{A.~Tamarit-Sariol}
\ead{anna.tamarit@upc.edu}

\cortext[cor]{Corresponding author}

\address[ma1]{Departament de Matem\`{a}tica Aplicada I,
 Universitat Polit\`ecnica de Catalunya,
 Diagonal 647, 08028 Barcelona, Spain}

\address[ma4]{Departament de Matem\`{a}tica Aplicada IV,
 Universitat Polit\`ecnica de Catalunya,
 Campus Diagonal Nord, Edifici C3. C. Jordi Girona, 1-3, 08034 Barcelona, Spain}

\tnotetext[t1]{The authors were supported in part by CUR-DIUE Grant 2014SGR504 (Catalonia) and
MINECO-FEDER Grant MTM2012-31714 (Spain).
We acknowledge the use of the UPC Applied Math cluster system for
research computing.}

\begin{abstract}
Let $q \ge 3$ be a period.
There are at least two $(1,q)$-periodic trajectories
inside any smooth strictly convex billiard table,
and all of them have the same length when the table is an ellipse or a circle.
We quantify the chaotic dynamics of axisymmetric billiard tables
close to their borders by studying the asymptotic behavior of
the differences of the lengths of their axisymmetric $(1,q)$-periodic
trajectories as $q \to +\infty$.
Based on numerical experiments,
we conjecture that, if the billiard table is a generic axisymmetric
analytic strictly convex curve,
then these differences behave asymptotically like an exponentially small
factor $q^{-3}\rme^{-r q}$ times either a constant or an oscillating function,
and the exponent $r$ is half of the radius of convergence
of the Borel transform of the well-known asymptotic series for the lengths of
the $(1,q)$-periodic trajectories.
Our experiments are restricted to some perturbed ellipses and circles,
which allows us to compare the numerical results
with some analytical predictions obtained by Melnikov methods and
also to detect some non-generic behaviors due to the presence of extra symmetries.
Our computations require a multiple-precision arithmetic
and have been programmed in PARI/GP.
\end{abstract}

\begin{keyword}
Billiards; Length spectrum; Exponentially small phenomena; Numeric experiments;
Melnikov method
\end{keyword}

\end{frontmatter}

\section{Introduction}

Billiards as a dynamical system go back to Birkhoff~\cite{Birkhoff1966}.
Let $Q$ be a closed smooth strictly convex curve in the Euclidean plane.
The Birkhoff billiard models the motion of a particle inside the region enclosed by~$Q$.
The particle moves with unit velocity and without friction following
a straight line; it reflects elastically when it hits $Q$.
Therefore, billiard trajectories consist of polygonal lines
inscribed in $Q$ whose consecutive sides obey to the rule
``the angle of reflection is equal to the angle of incidence.''
Such trajectories are sometimes called broken geodesics.
See~\cite{KatokHasselblatt1995,KozlovTreschev1991,Tabachnikov1995_Billiards}
for a general description.

A $(p,q)$-periodic billiard trajectory forms a closed polygon of $q$ sides
that makes $p$ turns before closing.
Birkhoff~\cite{Birkhoff1966} proved that there are at least two different
Birkhoff $(p,q)$-periodic billiard trajectories inside $Q$
for any relatively prime integers $p$ and $q$ such that $1 \le p < q$.

The \emph{length spectrum} of $Q$ is the subset of $\Rset_+$
defined as
\[
\LengthSpectrum(Q) =
l \Nset \cup \bigcup_{(p,q)} \Lambda^{(p,q)}\Nset,
\]
where $l = \Length(Q)$ and
$\Lambda^{(p,q)} \subset \Rset_+$ is the set of the lengths
of all $(p,q)$-periodic billiard trajectories inside $Q$.
The \emph{maximal difference} among lengths of $(p,q)$-periodic trajectories
is the non-negative quantity
\[
\Delta^{(p,q)} =
\sup \Lambda^{(p,q)} - \inf \Lambda^{(p,q)}.
\]

Many geometric and dynamical properties are encoded in the length spectrum
$\LengthSpectrum(Q)$ and the differences $\Delta^{(p,q)}$.

An old geometric question is:
\emph{Does the set $\LengthSpectrum(Q)$ allow one to reconstruct
the convex curve $Q$?}
The length spectrum and the Laplacian spectrum with Dirichlet boundary
conditions are closely related~\cite{AnderssonMelrose1977}.
Therefore, the question above can be colorfully restated as~\cite{Kac1966}:
\emph{Can one hear the shape of a drum?}
We refer to the book~\cite{Siburg2004} for some results on
this question.

The difference $\Delta^{(p,q)}$ is important from a dynamical point of view,
because it is an upper bound of Mather's $\Delta W_{p/q}$.
In its turn, $\Delta W_{p/q}$ is equal to the flux through
the $(p,q)$-resonance of the corresponding billiard
map~\cite{MacKayMeissPercival1984,Mather1986,Meiss1992,MatherForni1994}.
Thus, the variation of $\Delta^{(p,q)}$ in terms of the
rotation number $p/q \in (0,1)$ gives information about
the size of the different chaotic zones of the billiard map.
See Section~\ref{sec:BilliardMaps} for a more complete description
of these ideas.

Here, our main goal is to gain some insight into the billiard dynamics
close to the boundary of the billiard table.
We focus on the $(1,q)$-periodic billiards trajectories;
that is, we set $p = 1$.
We want to determine the asymptotic behavior of
\[
\Delta^{(1,q)} =
\sup \Lambda^{(1,q)} - \inf \Lambda^{(1,q)}
\]
as $q \to +\infty$.

Let $L^{(1,q)} \in \Lambda^{(1,q)}$ be the length of a
$(1,q)$-periodic billiard trajectory inside $Q$.
It does not matter which one.
Marvizi and Melrose~\cite{MarviziMelrose1982} proved that
if $Q$ is smooth and strictly convex,
then there exist some asymptotic coefficients $l_j = l_j(Q)$ such that
\begin{equation}\label{eq:asymptoticLength}
L^{(1,q)} \asymp \sum_{j\ge 0} l_j q^{-2j},
\qquad q \to \infty.
\end{equation}
For instance, $l_0 = l = \Length(Q)$ and
$l_1 = - \frac{1}{24} \left( \int_Q \kappa^{2/3} \rmd s \right)^3$,
where $\kappa$ and $\rmd s$ are the curvature and the length element
of $Q$, respectively.
The symbol $\asymp$ means that the series in the right hand side
is asymptotic to $L^{(1,q)}$.
The asymptotic coefficients $l_j$ can be explicitly written as
integrals over $Q$ of suitable algebraic expressions of $\kappa$
and its derivatives.
The first five coefficients can be found in~\cite{Sorrentino2013}.
The asymptotic series~\eqref{eq:asymptoticLength}
does not depend on the choice $L^{(1,q)} \in \Lambda^{(1,q)}$,
so
\[
\lim_{q \to +\infty} q^k \Delta^{(1,q)} = 0,\qquad \forall k>0.
\]
That is, the differences $\Delta^{(1,q)}$ are \emph{beyond all orders} in $q$.
In fact,
they satisfy the following \emph{exponentially small upper bound}
in the analytic case~\cite{Martinetal2014}.
If $Q$ is a closed analytic strictly convex curve,
then there exist constants $K, q_0, \alpha > 0$ such that
\[
\Delta^{(1,q)} \le K \rme^{-2\pi \alpha q},\qquad \forall q \ge q_0.
\]
The exponent $\alpha$ is related to the width of the complex strip
where a certain $1$-periodic angular coordinate is analytic.
If a billiard map (or any analytic exact twist map)
has a rotational invariant circle of Diophantine rotation number $\omega$,
there exist other exponentially small upper bounds for $\Delta^{(p,q)}$
(or for the residues of $(p,q)$-periodic orbits) when $p/q \to \omega$.
See~\cite{Greene1979,MacKay1992,DelshamsLlave2000}.

Similar singular behaviors have been observed in problems about
the splitting of separatrices of analytic
maps~\cite{FontichSimo1990,Gelfreichetal1991,DelshamsRRR1998,DelshamsRRR1999,
           Gelfreich1999,GelfreichLazutkin2001,GelfreichSauzin2001,RamirezRos2005,
           GelfreichSimo2008,Martinetal2011a,Martinetal2011b,BaldomaMartin2012}.
All these splittings are not exponentially small in a discrete big parameter
$q \in \Nset$, but in a continuous small parameter $h > 0$.
Namely, $h$ is the characteristic exponent of the hyperbolic fixed point
whose separatrices split.
Thus, we may think that $h = 1/q$ for comparison purposes.
The splitting size in many analytic maps
satisfies the exponentially small asymptotic formula
\begin{equation}\label{eq:SplittingSize}
\mbox{``splitting size''} \asymp A(1/h) h^{-m} \rme^{-r/h},\qquad
h \to 0^+,
\end{equation}
for some exponent $r > 0$, some power $m \in \Rset$, and some
function $A(1/h)$ that is either constant or oscillating.
The exponent $r$ and the function $A(1/h)$ are determined
by looking at the complex singularities closest to the real axis of
the homoclinic solution of a limit Hamiltonian flow related to the map.
Such methodology has been rigorously established for the
standard map~\cite{Gelfreich1999}, the H\'enon map~\cite{GelfreichSauzin2001},
and some perturbed McMillan
maps~\cite{DelshamsRRR1998,Martinetal2011a,Martinetal2011b}.
It has also been numerically checked in certain billiard maps~\cite{RamirezRos2005}
and several polynomial maps~\cite{GelfreichSimo2008},
but there are other maps where it fails~\cite{BaldomaMartin2012}.
Let us briefly recall some claims about polynomial standard maps
contained in~\cite{Gelfreichetal1991,GelfreichSimo2008}.
First, $r = 2\pi \delta$, where $\delta$ is the distance
of these singularities to the real axis.
Besides,
\begin{equation}\label{eq:Aq_Standard}
A(1/h) = \mu a/2 + a \sum_{j=1}^J \cos(2\pi \beta_j/h + \varphi_j),
\end{equation}
for some $\mu \in \{0,1\}$, some amplitude $a \neq 0$, and
some phases $\varphi_j \in \Rset$, when these singularities are
\[
\pm \delta \rmi \mbox{ (if and only if $\mu = 1$)},
\pm \beta_1 \pm \delta \rmi,\ldots, \pm \beta_J \pm \delta \rmi.
\]
For instance, the limit Hamiltonian flow for the standard map
is a pendulum, so $\pm \pi \rmi/2$ are the closest singularities
to the real axis  and the ``splitting size'' is the so-called
\emph{Lazutkin constant} $\omega_0 \simeq 1118.827706$ times
$h^{-2} \rme^{-\pi^2/h}$, see~\cite{Gelfreich1999}.

It is also known that, usually, $r = \rho/2$,
where $\rho$ is the radius of convergence of the Borel transform of
the divergent asymptotic series that approaches the
separatrices~\cite{DelshamsRRR1999,GelfreichSauzin2001,RamirezRos2005,GelfreichSimo2008}.

By looking at our billiard problem from the perspective of those results
(and others not mentioned here for the sake of brevity),
it is natural to make the following conjecture.
This conjecture is strongly supported by our numerical experiments.

\begin{conj}\label{conj:General}
If $Q$ is a closed analytic strictly convex curve,
but it is neither a circle nor an ellipse,
the asymptotic series~(\ref{eq:asymptoticLength})
diverges for all period $q \in \Nset$, but it is Gevrey-1,
so its Borel transform
\begin{equation}\label{eq:BorelTransform}
\sum_{j\ge 0} \hat{l}_j z^{2j-1},\qquad
\hat l_j = \frac{l_j}{(2j-1)!},
\end{equation}
has a radius of convergence $\rho \in (0,+\infty)$.
Set $r = \rho/2$.

If $Q$ is a generic axisymmetric algebraic curve, then
\begin{equation}\label{eq:DeltaAsymptotic}
\Delta^{(1,q)} \asymp |A(q)| q^{-3} \rme^{-r q},\qquad
q \to +\infty,
\end{equation}
for some function $A(q)$ that is either constant:
$A(q) = a/2 \neq 0$, or oscillatory:
\begin{equation}\label{eq:Aq_Billiard}
A(q) = \mu a/2 + a \sum_{j=1}^J \cos(2\pi \beta_j q),
\end{equation}
with $\mu \in \{0,1\}$, $a \neq 0$, $J \ge 1$,
and $0 < \beta_1 < \cdots < \beta_J$.
The cases $A(q) = a/2$ and $A(q) = a \cos(2\pi \beta)$
take place in open sets of the space of axisymmetric algebraic curves.
All the other cases are phenomena of co-dimension one.

If $Q$ is a generic bi-axisymmetric algebraic curve,
$\Delta^{(1,q)}$ has the previous asymptotic behavior when $q$ is even
and $q \to +\infty$, but $\Delta^{(1,q)} = \Order(q^{-2} \rme^{-2 r q})$
when $q$ is odd and $q \to +\infty$.
\end{conj}

We stress that the oscillating function~(\ref{eq:Aq_Standard})
has some phases,
but there are no phases in the oscillating function~(\ref{eq:Aq_Billiard}).
This phenomenon is not new.
The asymptotic formulas for the exponentially small splittings
of generalized standard maps with trigonometric polynomials
do not have phases either~\cite{GelfreichSimo2008}.

A curve is \emph{axisymmetric} when it is symmetric with
respect to a line,
and \emph{bi-axisymmetric} when it is symmetric with respect
to two perpendicular lines.
A planar curve is \emph{algebraic} when its points
are the zeros of some polynomial in two variables.
We require strict convexity,
since it is already an essential hypothesis in the smooth setup.
We only consider algebraic curves by comparison with the
above results about polynomial standard maps.
Our algebraic curves have no singular points,
because we ask them to be closed and analytic.

If $Q$ is a circle of radius $r_0$,
all its $(p,q)$-periodic billiard trajectories
have length $2 r_0 q \sin(\pi p/q)$,
so $\Delta^{(p,q)} = 0$ for all $p/q \in (0,1)$,
and the asymptotic series~(\ref{eq:asymptoticLength}) becomes
\[
L^{(1,q)} =
2 r_0 q \sin(\pi/q) =
2 r_0 \sum_{j \ge 0}
\frac{(-1)^j \pi^{2j+1}}{(2j+1)!} q^{-2j},
\]
which converges for all $q$.
In particular, $\rho = +\infty$.
Ellipses have analogous properties.
This has to do with the fact that elliptic and circular billiards
are integrable.
A conjecture attributed to Birkhoff claims that the only integrable
smooth convex billiard tables are ellipses and circles~\cite{Poritsky1950}.
Following the discussion on the Mather $\beta$-function
contained in~\cite{Sorrentino2013},
this old conjecture is reformulated as:
\emph{The series in~(\ref{eq:asymptoticLength}) converges
for some period $q \in \Nset$ if and only if $Q$ is an ellipse or a circle.}

In this paper,
we present several numerical experiments and some analytical results
that support Conjecture~\ref{conj:General}.
For the sake of simplicity,
all numerical experiments are carried out using the model tables
\begin{equation}\label{eq:ModelTables}
Q = \left\{ (x,y) \in\Rset^2 : x^2 + y^2/b^2 + \epsilon y^n = 1 \right\}.
\end{equation}
Here, $b \in (0,1]$ is the semi-minor axis,
$\epsilon\in \Rset$ is the perturbative parameter,
and $n\in\Nset$, with $3 \le  n\le 8$, is the degree of the perturbation.
We will refer to $Q$ as a perturbed ellipse when $0 < b < 1$
and as a perturbed circle when $b=1$.
Next, we explain the four main reasons for this choice of billiard tables.

As a first reason, we know that all the billiard
tables~(\ref{eq:ModelTables}) are \emph{nonintegrable} for $n \ge 3$
and $0 < \epsilon \ll 1$, and so the dynamics inside them should be
far from trivial. The question of which perturbed ellipses give rise
to integrable billiards is addressed in~\cite{DelshamsRRR1996}.
Theorem~4.1 of that paper imply that the
tables~(\ref{eq:ModelTables}) are nonintegrable if $0 < b < 1$, $n
\ge 4$ is even, and $\epsilon$ is small enough. This result can be
extended, after some technicalities, to odd degrees. Furthermore,
all integrable deformations of ellipses of small eccentricities
---this includes, of course, circles---
are ellipses~\cite{AvilaDeSimoiKaloshin2014},
so the tables~(\ref{eq:ModelTables}) are nonintegrable if
$b = 1$, $n \ge 3$, and $\epsilon$ is small enough.

The second reason is that we want to use some \emph{Melnikov methods} that
are well suited for the study of billiards inside perturbed ellipses
and perturbed circles~\cite{RamirezRos2006,PintodeCarvalhoRamirezRos2013}.
We recall that $\Delta^{(1,q)} = 0$ for any $q \ge 3$ in elliptic billiards.
Thus, since the difference $\Delta^{(1,q)} = \Delta^{(1,q)}(\epsilon)$
is analytic in $\epsilon$ and vanishes at $\epsilon = 0$, we know that
\[
\Delta^{(1,q)} = \Delta^{(1,q)}(\epsilon) =
\epsilon \Delta^{(1,q)}_1 + \Order(\epsilon^2),
\]
for some coefficient $\Delta^{(1,q)}_1 \in \Rset$ that can be computed
explicitly.
To be precise, it turns out that if $0 < b < 1$ then
\begin{equation}\label{eq:Delta1Melnikov}
\Delta^{(1,q)}_1 \asymp M_n q^{m_n} \rme^{-c q},\qquad q \to +\infty,
\end{equation}
for some Melnikov exponent $c > 0$ not depending on $n$,
some Melnikov power $m_n \in \Zset$, and some Melnikov constant $M_n \neq 0$.
These three Melnikov quantities can be explicitly computed,
but we have carried out the computations only for the cubic ($n=3$)
and quartic ($n=4$) perturbations for the sake of brevity.
Besides, $\lim_{b \to 1} c = +\infty$.
The Melnikov method provides no information
when $n$ is odd and $q$ even; $\Delta^{(1,q)} = 0$ in such case.
See Proposition~\ref{prop:melnikov_boundary} for details.

Which is the relation between the asymptotic
formula~(\ref{eq:DeltaAsymptotic}) and the first order
Melnikov computation~(\ref{eq:Delta1Melnikov})?
The answer is that $r \neq c$ and $m_n \neq -3$,
so \emph{the Melnikov method does not accurately predict
the singular behavior of $\Delta^{(1,q)}$}.
Nevertheless, $\lim_{\epsilon \to 0} r =c$, so some information can be
retrieved from the Melnikov method, at least for perturbed ellipses.

The case of perturbed circles is harder.
See Section~\ref{sec:PerturbCircles}.

\emph{Symmetries} are another reason for the choice of
tables~(\ref{eq:ModelTables}). On the one hand, symmetries greatly
simplify the computation of periodic trajectories. To be precise, we
just compute the signed difference $D_q$ between two particular
axisymmetric $(1,q)$-periodic trajectories, instead of
$\Delta^{(1,q)}$ or $\Delta W_{1/q}$. Clearly, $|D_q| \le
\Delta^{(1,q)}$. Often, $|D_q| = \Delta^{(1,q)} = \Delta W_{1/q}$.
See Proposition~\ref{prop:melnikov_boundary}. On the other hand,
bi-axisymmetric curves are a very particular class of axisymmetric
curves, so our model tables may display other asymptotic behaviors
when $n$ is even. We will check that this expectation is fulfilled.
Concretely,
\[
\Delta^{(1,q)} \asymp |B(q)| q^{-2} \rme^{-2 r q}, \qquad
q \to +\infty,
\]
for some constant or oscillating function $B(q)$
when $n$ is even and $q$ is odd.
This asymptotic behavior has several differences with respect to
the generic one conjectured in~(\ref{eq:DeltaAsymptotic}).
Both the exponent in $\rme^{-r q}$ and (if any) the frequencies
$0 < \beta_1 < \cdots < \beta_J$ are doubled,
the power in $q^{-3}$ is increased by one, etcetera.
We think that this new asymptotic behavior is generic
among bi-axisymmetric algebraic curves when the period $q$ is odd.

The last reason for the choice of such simple billiard tables is to
reduce the \emph{computational effort} as much as possible. In
particular, we limit the degree of the perturbation to the range $3
\le n \le 8$ for this reason. Recall that each set $\Lambda^{(1,q)}$
is contained in an exponentially small (in $q$) interval, so the
computation of $\Delta^{(1,q)}$ (or $D_q$) gives rise to very strong
cancellations. This forces us to use a multiple-precision arithmetic
to compute them. We have performed some computations with more than
twelve thousand digits, based on the open source PARI/GP
system~\cite{PARIGP}. Similar computations in the setting of
splitting of separatrices of analytic maps can be found
in~\cite{DelshamsRRR1999,RamirezRos2005,GelfreichSimo2008}.

Finally, we recall that the exponent $r$ is found by looking at the complex
singularities of the homoclinic solution of a limit Hamiltonian flow
in many cases of splitting of separatrices.
\emph{Does such kind of limit problem exist in our billiard setting?}
Unfortunately, we do not have a completely satisfactory answer yet,
but we propose a candidate in Section~\ref{sec:LimitProblem}.
It is empirically derived by using the Taylor expansions
of the billiard dynamics close to the border given by Lazutkin
in~\cite{Lazutkin1973}.
Let $\kappa(s)$ be the curvature of $Q$ in some arc-length parameter
$s \in \Rset/l \Zset$.
Let $\xi \in \Rset/\Zset$ be a new angular variable defined by
\begin{equation}\label{eq:xiDefinition}
C \frac{\rmd \xi}{\rmd s} = \kappa^{2/3}(s), \qquad
C=\int_Q \kappa^{2/3}\rmd s.
\end{equation}
Let $\delta$ be the distance of the set of singularities and zeros
of the curvature~$\kappa(\xi)$ to the real axis.
We thought that $r = 2\pi \delta$, but our experiments disprove it.
We have only obtained that $r \le 2\pi \delta$,
the equality being an infrequent situation.
But there are some good news about our candidate.
First, the Melnikov exponent is $c = 2\pi\delta$, when $Q = \{ (x,y)
\in \Rset^2 : x^2 + y^2/b^2 = 1 \}$, with $0 < b < 1$. See
Proposition~\ref{prop:deltaEllipse}. Second, we have also seen that
if $b = 1$ and $n \ge 3$ is fixed, then there exist some constants
$\chi_n,\eta_n \in \Rset$, $\chi_n \le \eta_n$, such that
\[
r = \frac{|\log \epsilon|}{n} + \chi_n + \order(1),\qquad
2\pi \delta = \frac{|\log \epsilon|}{n} + \eta_n + \order(1),
\]
as $\epsilon \to 0^+$.
The second formula is proved in Proposition~\ref{prop:deltaPerturbedCircle},
the first one is numerically checked in Section~\ref{sec:PerturbCircles}.
Therefore, our candidate captures exactly the logarithmic growth of
the exponent $r$ for perturbed circles.
Third, our experiments suggest that $r = 2\pi\delta$ when
$b=1$, $n \in \{5,7\}$, and $\epsilon \in (0,1/10)$.

The paper has the following structure.
Section~\ref{sec:BilliardMaps} contains the dynamical interpretation
of Mather's $\Delta W_{p/q}$. We discuss our candidate for limit
problem in Section~\ref{sec:LimitProblem}. The axisymmetric tables
and their axisymmetric periodic billiard trajectories are presented
in more detail in Section~\ref{sec:ModelTables}. The main results
about perturbed ellipses and perturbed circles are described in
Sections~\ref{sec:PerturbEllipses} and~\ref{sec:PerturbCircles},
respectively. All proofs have been relegated to the appendices.

\section{Twist maps, actions, Mather's $\Delta W$, and billiards}
\label{sec:BilliardMaps}

We recall some results about exact twist maps and billiards.
We refer to the
books~\cite{KozlovTreschev1991,Tabachnikov1995_Billiards,KatokHasselblatt1995}
and the surveys~\cite{Meiss1992,MatherForni1994}
for a more detailed exposition.

Let $\Tset = \Rset/\Zset$ and  $I=(y_-,y_+) \subset \Rset$
for some $-\infty \le y_- < y_+ \le +\infty$.
Let $\omega = \rmd x \wedge \rmd y$ be the canonical area form on
$\Tset \times I$.
Note that $\omega = -\rmd \lambda$, where $\lambda = y \rmd x$.
A smooth diffeomorphism $f: \Tset \times I \to \Tset \times I$
is an \emph{exact twist map} when it preserves $\omega$, has zero flux,
and satisfies the twist condition $\frac{\partial x_1}{\partial y} > 0$,
where $F: \Rset \times I \to \Rset \times I$,
$F(x,y) = (x_1,y_1)$, is a (fixed) lift of $f$.

We also assume that $f$ can be extended as
rigid rotations of angles $\varrho_\pm$ to the boundaries
$C_\pm = \Tset \times \{ y_\pm \}$.
We know that $\varrho_- < \varrho_+$ from the twist condition.
Let $p$ and $q$ be two relatively prime integers
such that $p/q \in (\varrho_-,\varrho_+)$ and $q\ge 1$.
A point $(x,y) \in \Rset \times I$ is $(p,q)$\emph{-periodic}
when $F^q(x,y)=(x+p,y)$.
The corresponding point $(x,y) \in \Tset \times I$
is a $q$-periodic point of $f$ that is translated
$p$ units in the base by the lift.
A $(p,q)$-periodic orbit is \emph{Birkhoff} when it is ordered
around the cylinder in the same way that the orbits of the
rigid rotation of angle $p/q$.
The Poincar\'e-Birkhoff Theorem states that
there exist at least two different Birkhoff $(p,q)$-periodic
orbits~\cite{KatokHasselblatt1995,Meiss1992}.

Let $E = \{ (x,x_1) \in \Rset^2 : \varrho_- < x_1 - x < \varrho_+ \}$.
Then there exists a function $h:E \to \Rset$ such that
$h(x+1,x_1+1) = h(x,x_1)$ and
\[
y_1 \rmd x_1 - y \rmd x = \rmd h(x,x_1).
\]
This function is called \emph{Lagrangian} or \emph{generating function}.
It is determined modulo an additive constant.
Twist maps satisfy the following classical Lagrangian formulation.
Their orbits are in one-to-one correspondence with
the (formal) stationary configurations of the \emph{action functional}
\[
\Rset^\Zset \ni \mathbf{x} = (x_j)_{j\in\Zset} \mapsto
W[\mathbf{x}] = \sum_{j \in \Zset} h(x_j,x_{j+1}).
\]
Note that, although the series for $W[\mathbf{x}]$ may be divergent,
$\frac{\partial W}{\partial x_j}$ only involves two terms
of the series, and so $\nabla W$ is well defined.

If $O = \{ (x_j,y_j) \}_{j \in \Zset}$ is a $(p,q)$-periodic orbit of $f$,
then
\[
h(x_{j+q},x_{j+q+1}) = h(x_j + p,x_{j+1} + p) = h(x_j,x_{j+1}),
\]
so there are only $q$ different terms in the action functional $W$,
which encode the $(p,q)$-periodic dynamics.
In particular, any $(p,q)$-periodic orbit $O = \{ (x_j,y_j) \}_{j \in \Zset}$
is in correspondence with a stationary configuration
$\mathbf{x} = (x_{0},\ldots, x_{q-1})\in\Rset^{q-1}$
of the $(p,q)$-\emph{periodic action}
\begin{equation}\label{eq:PeriodicAction}
W^{(p,q)}[\mathbf{x}] =
h(x_0,x_1) + h(x_1,x_2) + \cdots + h(x_{q-1},x_0 + p).
\end{equation}
We say that $W^{(p,q)}[O] = W^{(p,q)}[\mathbf{x}]$
is the $(p,q)$-periodic action of the $(p,q)$-periodic orbit $O$.
The Birkhoff $(p,q)$-periodic orbits provided by
the Poincar\'e-Birkhoff Theorem correspond to the minimizing and minimax
stationary configurations of $W^{(p,q)}$.

Mather defined the quantity~$\Delta W_{p/q} \ge 0$
as the action of the minimax periodic orbit minus
the action of the minimizing one~\cite{Mather1986}.
Mather's $\Delta W_{p/q}$ has a nice dynamical interpretation.
It is equal to the \emph{flux} through any homotopically non trivial
curve without self-intersections passing through all the points
of both the minimizing and the minimax $(p,q)$-periodic orbits.
This is the \emph{MacKay-Meiss-Percival action
principle}~\cite{MacKayMeissPercival1984}.
Thus, $\Delta W_{p/q}$ gives a rough estimation of the size of
the $(p,q)$-resonance of the twist map.
We also recall that the hyperbolic (respectively, elliptic) periodic orbits
in a given resonance are generically minimizing (respectively, minimax).

Next, we adapt these ideas to billiard maps.

Let $Q$ be a smooth strictly convex curve in the Euclidean plane.
For simplicity, we assume that $l = \Length(Q) = 1$.
Let $\gamma: \Tset \to Q$, $s \mapsto \gamma(s)$,
be an arc-length counterclockwise parametrization of $Q$.
The bounce position of the particle inside $Q$ is determined
by the arc-length parameter $s$.
The direction of motion is measured by the angle of incidence $r \in (0,\pi)$.
Let
\begin{equation}\label{eq:BilliardMap}
f : \Tset\times(0,\pi) \to \Tset\times(0,\pi),\qquad
f(s,r)=(s_1,r_1),
\end{equation}
be the corresponding \emph{billiard map}.
Figure~\ref{fig:BilliardMap} illustrates this map.
The coordinates $(s,r)$ are called \emph{Birkhoff coordinates}.

\begin{figure}
\iffigures
\centering
\includegraphics[height=1.8in]{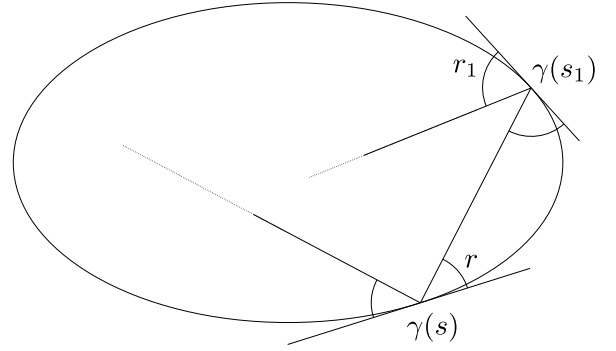}
\else
\vspace{1.8in}
\fi
\caption{The billiard map $f(s,r) = (s_1,r_1)$.}
\label{fig:BilliardMap}
\end{figure}

Let $\ell(s,s_1) = | \gamma(s) - \gamma(s_1)|$ be the Euclidean distance
between two impact points on $Q$.
It is easy to prove that
\[
\frac{\partial \ell}{\partial s}(s,s_1) = - \cos r,\qquad
\frac{\partial \ell}{\partial s_1}(s,s_1) = \cos r_1.
\]
If we consider the coordinates $(x,y) = (s,-\cos r) \in \Tset \times (-1,1)$,
then the billiard map $f$ becomes a smooth exact twist map with
Lagrangian $h(x,x_1) = -\ell(s,s_1)$ and boundary rotation numbers
$\varrho_- = 0$ and $\varrho_+ = 1$.
That is, the action of a periodic billiard trajectory is,
up to the sign, its length.
In particular,
Mather's $\Delta W_{p/q}$ is the length of the $(p,q)$-periodic
billiard trajectory that minimizes the action (and so, maximizes the length)
minus the length of the minimax one.
Generically,
\[
\Delta W_{p/q} = H^{(p,q)} - E^{(p,q)},
\]
where $H^{(p,q)}$ and $E^{(p,q)}$ are the lengths of the
hyperbolic and elliptic $(p,q)$-periodic billiard trajectories inside $Q$.
For instance, $H^{(1,2)} = 4a$, $E^{(1,2)} = 4b$, and
$\Delta W_{1/2} = 4(a-b)$ for the billiard inside
the ellipse $x^2/a^2 + y^2/b^2 = 1$ with $0 < b < a$.

Note that any $(p,q)$-periodic billiard trajectory gives rise to a
$(q-p,q)$-periodic one by inverting the direction of motion.
This means that $\Delta W_{p/q} = \Delta W_{(q-p)/q}$
for all $p/q \in (0,1/2)$.

We have listed the biggest Mather's $\Delta W_{p/q}$
for the billiard inside the perturbed circle
$x^2 + y^2 + y^4/10 = 1$.
See Table~\ref{table:MatherDeltaW}.
The rest of Mather's $\Delta W_{p/q}$ are smaller that $10^{-4}$.
The values in the table suggest that the $(1,2)$-resonance
and both $(p,4)$-resonances should be the most important ones.
This prediction is confirmed in Figure~\ref{fig:PhasePortrait},
where we display the biggest resonances of the billiard map
inside $x^2 + y^2 + y^4/10 = 1$.

\begin{table}
\begin{center}
\begin{tabular}{c r r r}
\hline
\hline
$(p,q)$ &  $H^{(p,q)}$ & $E^{(p,q)}$ & $\Delta W_{p/q}$ \\
\hline
$(1,2)$ & $4.000000$ & $3.828482$ & $0.171577$ \\
$(1,4)$ \mbox{ and } $(3,4)$ &  $5.594652$ &  $5.536901$ & $0.057751$ \\
$(1,3)$ \mbox{ and } $(2,3)$ &  $5.115169$ &  $5.112940$ & $0.002229$ \\
$(3,8)$ \mbox{ and } $(5,8)$ & $14.773311$ & $14.772302$ & $0.001009$ \\
$(5,12)$\mbox{ and }$(7,12)$ & $23.151909$ & $23.150969$ & $0.000940$ \\
$(3,10)$\mbox{ and }$(7,10)$ & $15.925337$ & $15.924445$ & $0.000892$ \\
$(1,6)$ \mbox{ and } $(5,6)$ &  $5.904338$ &  $5.903527$ & $0.000811$ \\
$(2,5)$ \mbox{ and } $(3,5)$ &  $9.366997$ &  $9.366503$ & $0.000494$ \\
$(1,8)$ \mbox{ and } $(7,8)$ &  $6.024507$ &  $6.024232$ & $0.000275$ \\
$(3,7)$ \mbox{ and } $(4,7)$ & $13.455442$ & $13.455236$ & $0.000206$ \\
$(1,5)$ \mbox{ and } $(4,5)$ &  $5.785133$ &  $5.785011$ & $0.000122$ \\
\hline
\hline
\end{tabular}
\end{center}
\caption{The biggest Mather's $\Delta W_{p/q}$
for the billiard inside $x^2 + y^2 + y^4/10 = 1$.}
\label{table:MatherDeltaW}
\end{table}

Mather's $\Delta W_{p/q}$ allow us to single out the most important resonances,
but they do not give an exact measure of the size of resonances.
To begin with, there is not a unique way to define such size.
A choice is the area $A_{p/q}$ of the Birkhoff instability region
that contains the $(p,q)$-resonance.
A \emph{Birkhoff instability region} is a region of the phase space delimited
by two rotational invariant curves (RICs) without any other RIC in its interior.
If we have a twist map with a $(p,q)$-resonant RIC,
then $\Delta W_{p/q} = \Order(\epsilon)$ and
$A_{p/q} = \Order(\epsilon^{1/2})$ under generic perturbations of order
$\Order(\epsilon)$.
See~\cite{Olvera2001}.
This shows up a clear difference between these two quantities.
For instance, the billiard map inside the circle $x^2 + y^2 = 1$
has a $(1,2)$-resonant RIC, which is destroyed under the
perturbation $x^2 + y^2/(1-\epsilon)^2 = 1$.
However, this perturbed billiard table is integrable (it is an ellipse),
so both quantities can be analytically computed:
$\Delta W_{1/2} = 4 \epsilon$ and $A_{1/2} = 8 \epsilon^{1/2}$.
We omit the details.

\begin{figure}
\iffigures
\psfrag{0}{$0$}
\psfrag{l/4}{$l/4$}
\psfrag{3l/4}{$3l/4$}
\psfrag{l}{$l$}
\psfrag{l/2}{$l/2$}
\psfrag{0}{$0$}
\psfrag{pi}{$\pi$}
\psfrag{p2}{$\frac{\pi}{2}$}
%\psfrag{-1}{$-1$}
%\psfrag{1}{$1$}
\includegraphics[height=2.4in]{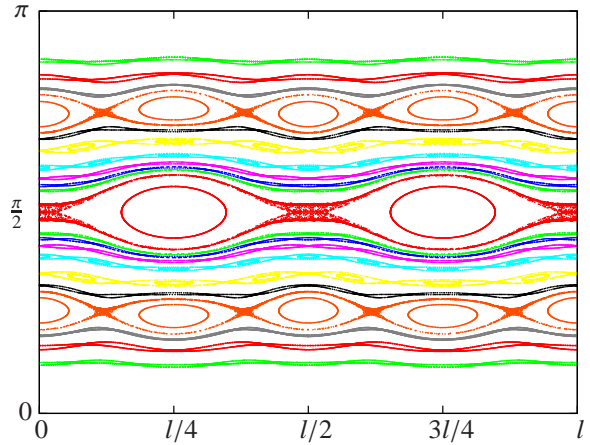}
\else \vspace{2.4in} \fi \caption{The biggest $(p,q)$-resonances of
the billiard map $f(s,r) = (s_1,r_1)$ inside the perturbed circle $Q
= \{(x,y) \in \Rset^2 : x^2 + y^2 + y^4/10 = 1\}$. We recall that $l
= \Length(Q)$. All $(p,q)$-resonances with odd period $q$ have $2q$
elliptic islands due to the bi-axisymmetric character of the curve.
From bottom to top: $(1,8)$, $(1,6)$, $(1,5)$, $(1,4)$, $(3,10)$,
$(1,3)$, $(3,8)$, $(2,5)$, $(5,12)$, $(3,7)$, $(1,2)$, and their
$(q-p,q)$ symmetric counterparts.} \label{fig:PhasePortrait}
\end{figure}

\section{A candidate for limit problem}
\label{sec:LimitProblem}

To begin with, we recall how to obtain the limit problem for
the splitting of separatrices of the generalized standard map
$f(x,y) = (x_1,y_1)$ given by
\begin{equation}\label{eq:StandardMap}
x_1 = x + y_1,\qquad y_1 = y + \epsilon p(x).
\end{equation}
For simplicity,
we assume that $p(x)$ is a polynomial, $p(0) = 0$, and $p'(0) = 1$,
so the origin is a hyperbolic fixed point of $f$ with eigenvalues
$\lambda = \rme^h$ and $\lambda^{-1} = \rme^{-h}$,
where $\epsilon = 4 \sinh^2(h/2)$.
There is numerical evidence that the splitting size in this kind of polynomial
standard maps satisfies the asymptotic formula~(\ref{eq:SplittingSize})
for some exponent $r > 0$, some power $m \in \Rset$, and some constant
or oscillating function $A(1/h)$.
We determine the exponent following~\cite{GelfreichLazutkin2001}.

First, we transform the original map into the map
\[
x_1 = x + \mu z_1,\qquad z_1 = z + \mu p(z)
\]
by means of the scaling $z = y/\mu$, where $\mu = \sqrt{\epsilon}$.
Note that $\mu \asymp h$ as $\epsilon \to 0^+$.
The dynamics of this map for small $\mu$ resembles the dynamics
of the $\mu$-time flow of the Hamiltonian
$H_0(x,z) = z^2/2 - \int p(x) \rmd x$.
Besides, the origin is a hyperbolic equilibrium point of
the Hamiltonian system
\[
x' = \partial_z H_0(x,z) = z,\qquad
y' = - \partial_x H_0(x,z) = p(x).
\]
If the singular level set $\{(x,z) \in \Rset^2 : H_0(x,z) \equiv H_0(0,0) \}$
contains a separatrix to the origin,
then we compute the flow on it and we get a homoclinic solution
$(x_0(\xi),z_0(\xi))$ that can be seen as the limit of the
map on its separatrices when $\epsilon \to 0^+$.
Such homoclinic solution is determined, up to a constant time shift,
by imposing
\[
x''_0(\xi) = p(x_0(\xi)),\qquad
\lim_{\xi \to \pm \infty} x_0(\xi) = 0.
\]
It turns out that there exists $\delta > 0$ such that
$x_0(\xi)$ is analytic in the open complex strip
$\mathcal{I}_\delta = \{ \xi \in \Cset : |\Im \xi| < \delta\}$
and has singularities on the boundary of $\mathcal{I}_\delta$.
Then $r = 2\pi \delta$.
This claim is contained in~\cite{GelfreichLazutkin2001},
although a complete proof is still pending.
However, Fontich and Sim\'o proved the following exponentially small
upper bound in~\cite{FontichSimo1990}.
If $\alpha \in (0,\delta)$,
then there exist some constants $K,h_0 > 0$ such that
\[
\mbox{``splitting size''} \le K \rme^{-2\pi \alpha/h},\qquad
\forall 0 < h \le h_0.
\]
We want to emphasize an essential, but sometimes forgotten, hypothesis
of the Fontich-Sim\'o Theorem.
Let
\[
\sigma_0(\xi) = (x_0(\xi),y_0(\xi))
\]
be the original homoclinic solution.
The generalized standard map~(\ref{eq:StandardMap}) should have
an analytic extension to a complex neighborhood in $\Cset^2$
of $\sigma_0(\overline{\mathcal{I}_\alpha})$.
If $p(x)$ is a polynomial,
then $f$ can be extended to the whole $\Cset^2$ and
this hypothesis is automatically fulfilled.
On the contrary, it remains to be checked when $p:\Rset \to \Rset$
is just a real analytic function.

Next, we adapt these ideas to our billiard problem.

Let $Q$ be an analytic strictly convex curve in the Euclidean plane.
Set $l = \Length(Q)$.
Let $\kappa(s)$ be the curvature of $Q$ in some arc-length parameter
$s \in \Rset /l\Zset$.
Note that $\kappa(s) > 0$ for all $s \in \Rset/l\Zset$.
Let $\rho(s) = 1/\kappa(s)$ be the radius of curvature.
We are interested in the dynamics of the billiard
map~(\ref{eq:BilliardMap}) when the angle of incidence $r$
tends to zero.
More precisely,
we consider that $r = \Order(1/q)$ and $q \to +\infty$.

Lazutkin~\cite{Lazutkin1973} gave the Taylor expansion
\[
s_1 = s + 2\varrho(s)r + \Order(r^2 ), \qquad
r_1 = r - 2\varrho'(s)r^2/3 + \Order (r^3 )
\]
for the dynamics of the billiard map~(\ref{eq:BilliardMap})
around $r = 0$.
Once fixed a period $q \gg 1$,
we take $\mu = 1/q \ll 1$ as the small parameter.
Then we transform the previous expansion into
\[
s_1 = s + \mu \varrho(s) v^{1/2} + \Order(\mu^2),\quad
v_1 = v - \frac{2}{3} \mu \varrho'(s) v^{3/2} + \Order(\mu^2),
\]
by means of the change of variables $\sqrt{v} = 2r/\mu$.
The billiard dynamics for small $\mu$ resembles the dynamics
of the $\mu$-flow of the Hamiltonian
$H_0(s,v) = \frac{2}{3} \varrho(s) v^{3/2}$.
That is, the $\mu$-flow of the Hamiltonian system
\[
s' = \varrho(s) v^{1/2},\qquad
v' = -\frac{2}{3}\varrho'(s) v^{3/2}.
\]
We compute the flow on the level set
$\mathcal{H}_C := \{H_0(s,v) \equiv \frac{2}{3}C^3 \}$,
for some constant $C > 0$.
If $(s,v) \in \mathcal{H}_C$,
then the first equation of the Hamiltonian system reads as
\[
\frac{\rmd s}{\rmd \xi} = s' =
\varrho(s) v^{1/2} =
C \varrho^{2/3}(s),
\]
or, equivalently, as
\begin{equation}\label{eq:ChangeXiS}
C \frac{\rmd \xi}{\rmd s} = \kappa^{2/3}(s).
\end{equation}
We only need the following observations to determine $C$.
We are looking at the $(1,q)$-periodic trajectories inside $Q$.
We have approximated the billiard dynamics by the $\mu$-time of the
Hamiltonian flow with $\mu = 1/q$.
Any $(1,q)$-periodic trajectory gives one turn after $q$ iterates
of the billiard map,
so the variable $\xi$ should be increased by one
if $s$ is increased by $l = \Length(Q)$.
Therefore,
\begin{equation}\label{eq:CDefinition}
C = C \int_0^l \frac{\rmd \xi}{\rmd s} \rmd s =
\int_0^l \kappa^{2/3}(s) \rmd s =
\int_Q \kappa^{2/3} \rmd s.
\end{equation}
Relation~(\ref{eq:xiDefinition}) is obtained by
joining equations~(\ref{eq:ChangeXiS}) and~(\ref{eq:CDefinition}).
Let $s = s_0(\xi)$ be the inverse of the solution $\xi = \xi_0(s)$
of the differential equation~(\ref{eq:xiDefinition}) determined,
for the sake of definiteness, by the initial condition $\xi_0(0) = 0$.
By abusing the notation,
let $\kappa(\xi) = \kappa(s_0(\xi))$ be the curvature in the
new angular variable $\xi \in \Rset/\Zset$.
Then $\kappa(\xi)$ is a $1$-periodic real analytic function
which does not vanish on the reals.
Let us assume that there exists $\delta > 0$
such that $\kappa(\xi)$ is analytic and does not vanish
on the open complex strip $\mathcal{I}_\delta$
and has singularities and/or zeros on the boundary of $\mathcal{I}_\delta$.
Note that we are avoiding not only singularities but also zeros of
the curvature $\kappa(\xi)$.
On the one hand, the results found by Marvizi and Melrose
only hold for smooth strictly convex curves, so the zeros
of the curvature are a source of potential problems.
On the other hand, several positive and negative fractional powers
of the curvature appear in the previous computations (see also below),
and such powers are not analytic at the zeros of the curvature.

Following the numerical evidences in the splitting problems
of the polynomial standard maps,
we thought that $r = 2\pi \delta$, but our experiments disprove it.
We have obtained that $r \le 2\pi \delta$,
the equality being an infrequent situation.

An explanation of such discrepancy is the following one.
Set $\sigma_0(\xi) = (s_0(\xi),r_0(\xi))$,
$r_0(\xi) = \mu \sqrt{v_0(\xi)}/2 = C \kappa^{1/3}(\xi)/2q$.
We know that the billiard map~(\ref{eq:BilliardMap}) can be analytically
extended to $(\Rset/l\Zset) \times [0,\pi)$; see~\cite[Proposition 5]{Martinetal2014}.
However, we do not know whether it can be analytically extended to a
complex neighborhood in $(\Cset/l\Zset) \times \Cset$ of
$\sigma_0(\mathcal{I}_\alpha)$ as $\alpha \to \delta^-$ and $q \to +\infty$ or not.
Hence, the inequality $r \le 2\pi \delta$ does not look so bad
in the light of the previous discussion about the Fontich-Sim\'o Theorem.
In fact, it is commonly accepted that the magnitude involved
in the exponent of the exponentially small formulas for splitting problems
is not the minimum distance to the real line of the set of singularities of the
time parametrization of the separatrix but the minimum distance to
the real line of the set of singularities \emph{of the perturbation
of the system when evaluated on the time parametrization of the
separatrix}.
See~\cite{GuardiaSeara2012,BaldomaMartin2012} for some examples.
It seems reasonable to think that
one has to compute the singularities of the Lagrangian evaluated on the
solution of~\eqref{eq:xiDefinition}, which, in its turn,
reduces to the study of the singularities of $\gamma(s_0(\xi))$.
This is a work in progress.

\section{Model tables}\label{sec:ModelTables}

We restrict our study to the perturbed ellipses and perturbed circles
given implicitly in~(\ref{eq:ModelTables}).
To be precise, the algebraic curve $x^2+y^2/b^2+\epsilon y^n = 1$
has several real connected components when $n$ is odd.
Henceforth, we only consider the one that tends to the
ellipse (or circle) $x^2 + y^2/b^2 = 1$
as $\epsilon$ tends to zero.

Let $\epsilon_n = \epsilon_n(b)$ be the maximal positive parameter such that
\begin{equation}\label{eq:MaximalInterval}
\mbox{$Q$ is analytic and strictly convex for all
$\epsilon \in I_n := (0,\epsilon_n)$}.
\end{equation}
On the one hand, $I_n=(0,+\infty)$ when $n$ is even.
In such cases, we will reach the value $\epsilon=1$
in some numerical computations.
On the other hand, if $n$ is odd,
the algebraic curve defined by $x^2+y^2/b^2+\epsilon y^n=1$
has a singular point on the~$y$-axis when
\begin{equation}\label{eq:SingularEpsilon}
\epsilon =
\bar{\epsilon}_n =
\bar{\epsilon}_n(b) :=
2(n-2)^{n/2-1}n^{-n/2}b^{-n}.
\end{equation}
Thus, $Q$ is no longer analytic when $\epsilon = \bar{\epsilon}_n$.
Our computations suggest that $\epsilon_n = \bar{\epsilon}_n$ so
we restrict our experiments to the range $0<\epsilon<\bar{\epsilon}_n$.
We note that $\bar{\epsilon}_3(b) \approx 0.3849/ b^3$,
$\bar{\epsilon}_5(b) \approx 0.1859/b^5$, and
$\bar{\epsilon}_7(b) \approx 0.1232/ b^7$.
We also restrict our experiments to the degrees $3\le n\le 8$.

The symmetries of our model tables simplify the search of
some periodic trajectories.
If $n$ is even,
$Q$ is symmetric with respect to both axis of coordinates,
so $Q$ is bi-axisymmetric.
If $n$ is odd,
$Q$ is symmetric with respect to the $y$-axis only,
so $Q$ is axisymmetric but not bi-axisymmetric.
We say that a billiard trajectory is \emph{axisymmetric} when its
corresponding polygon is symmetric
with respect to some axis of coordinates.
We only compute axisymmetric periodic trajectories, APTs for short.

First, let us focus on the case odd $n$.
The axisymmetric trajectories inside $Q$ are characterized as
the ones with an impact point on
or with a segment perpendicular to the $y$-axis.
The APTs are characterized as the ones satisfying
twice the former condition.
Thus, there are four kinds of APTs inside $Q$.
Besides, only two of these kinds are possible depending on the
(parity of the) period~$q$.

The classification for even $n$ is richer because
the symmetry with respect to the~$x$-axis plays the same role.
See Table~\ref{table:ClassificationAPTs}.

\begin{table}[!t]
\centering
\begin{tabular}{ccc}
\hline\hline
$n$ & $q$ & Examples of APTs with minimal periods \\
\hline
even & $2k+1$ &
\begin{tabular}{c}
\iffigures
\epsfig{file=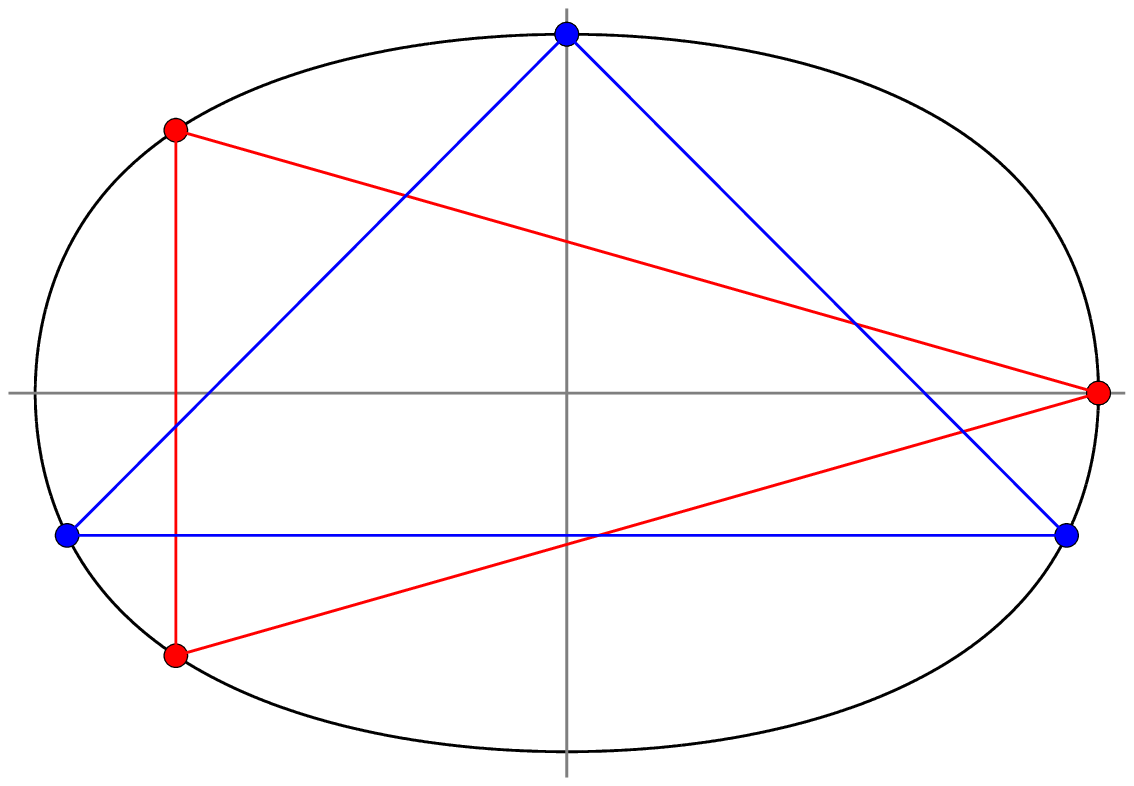,width=2in,clip=}
\else
\fbox{\rule{0in}{1.4in}\hspace{5cm}}
\fi
\end{tabular} \\
even & $4k+2$ &
\begin{tabular}{c}
\iffigures
\epsfig{file=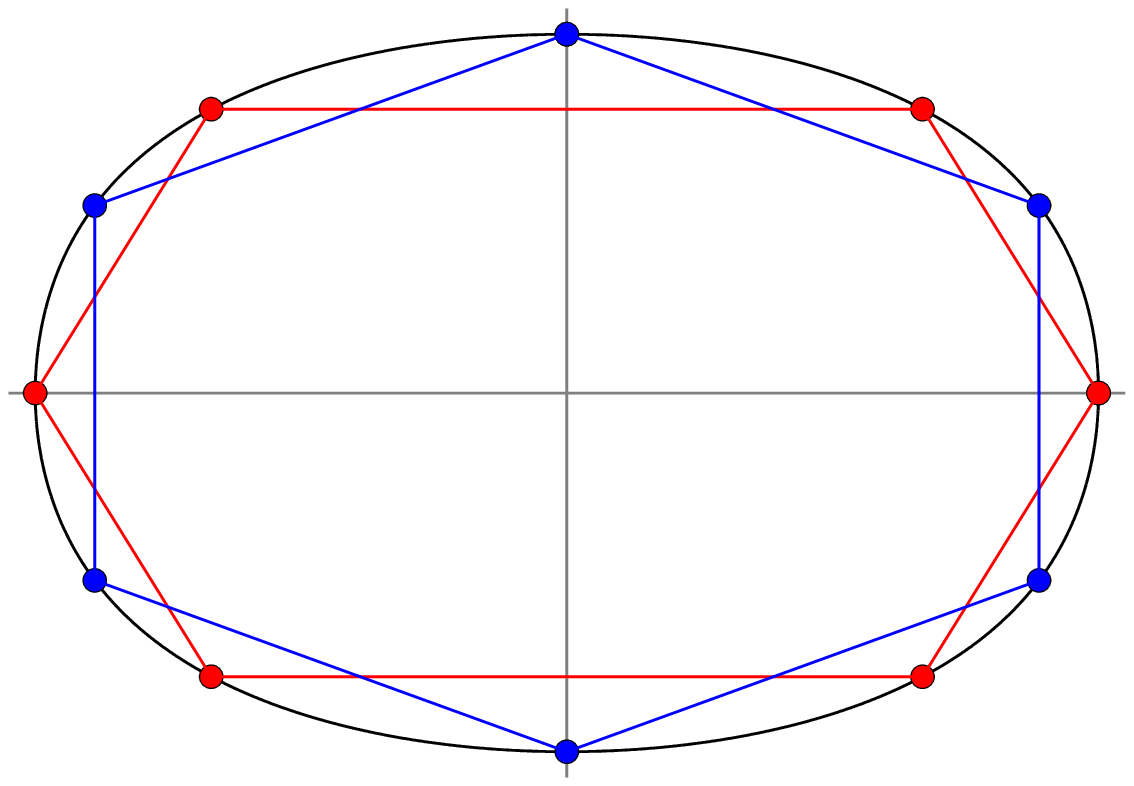,width=2in,clip=}
\else
\fbox{\rule{0in}{1.4in}\hspace{5cm}}
\fi
\end{tabular} \\
even & $4k$  &
\begin{tabular}{c}
\iffigures
\epsfig{file=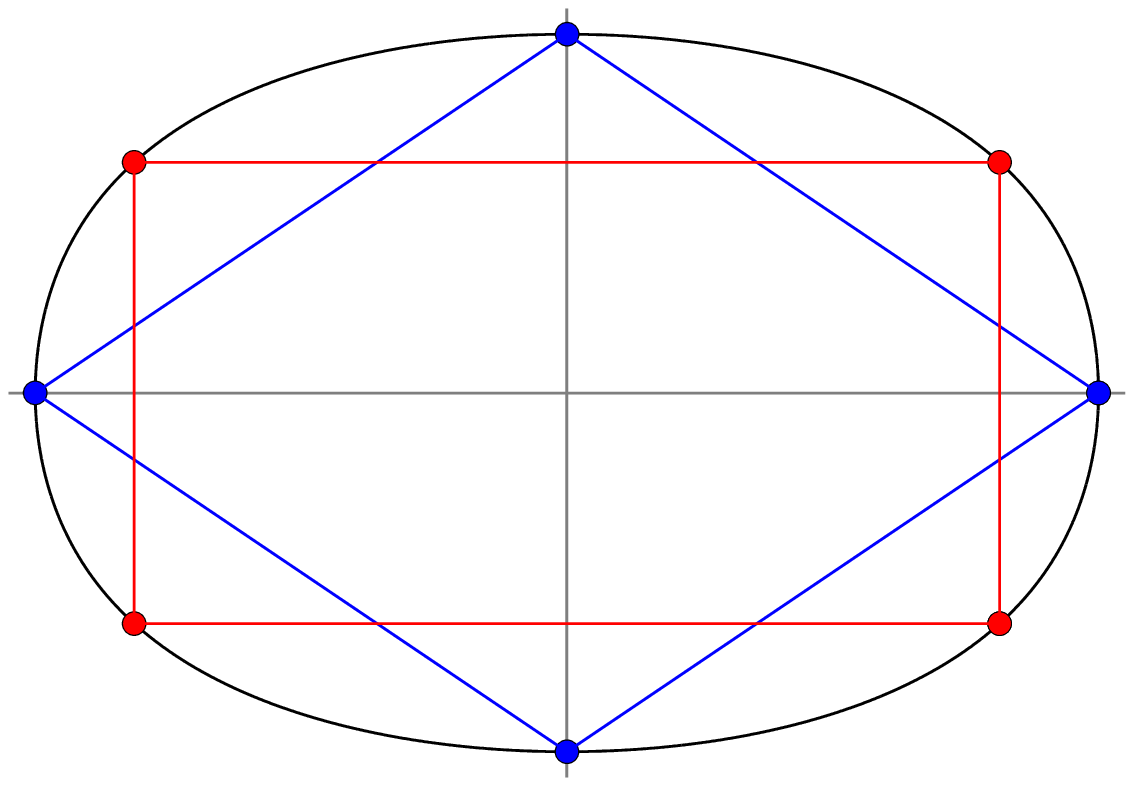,width=2in,clip=}
\else
\fbox{\rule{0in}{1.4in}\hspace{5cm}}
\fi
\end{tabular} \\
odd & $2k+1$ &
\begin{tabular}{c}
\iffigures
\epsfig{file=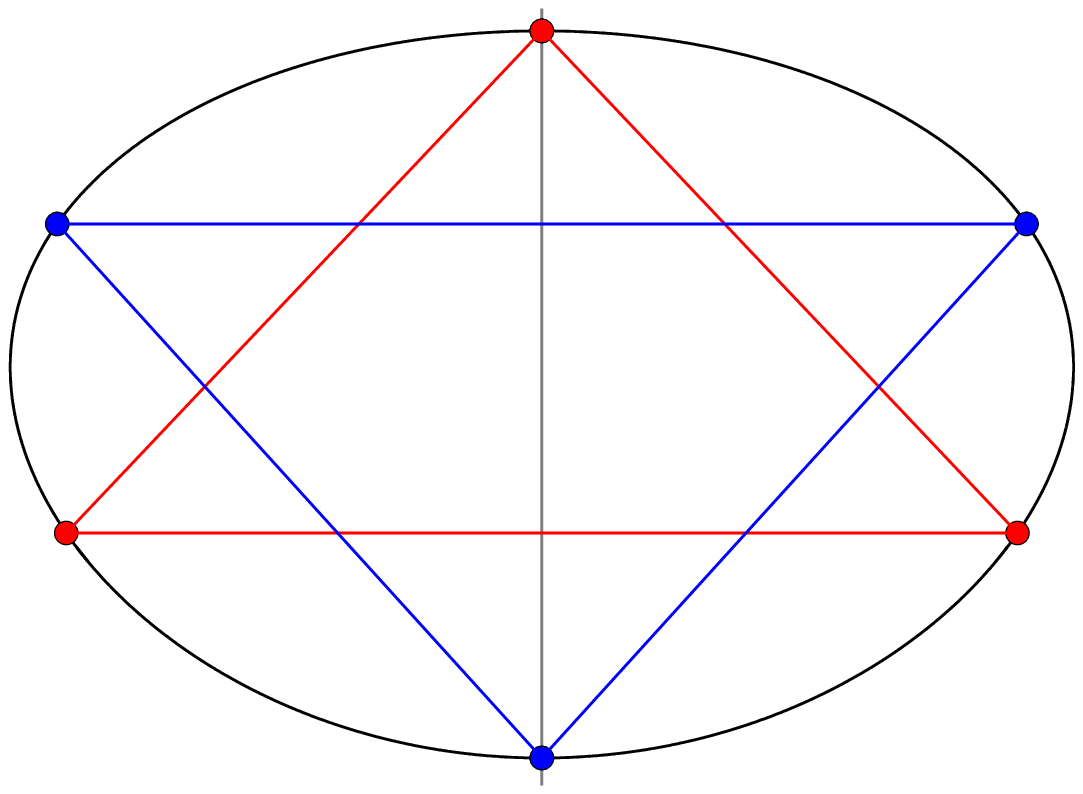,width=2in,clip=}
\else
\fbox{\rule{0in}{1.4in}\hspace{5cm}}
\fi
\end{tabular} \\
odd & $2k$ &
\begin{tabular}{c}
\iffigures
\epsfig{file=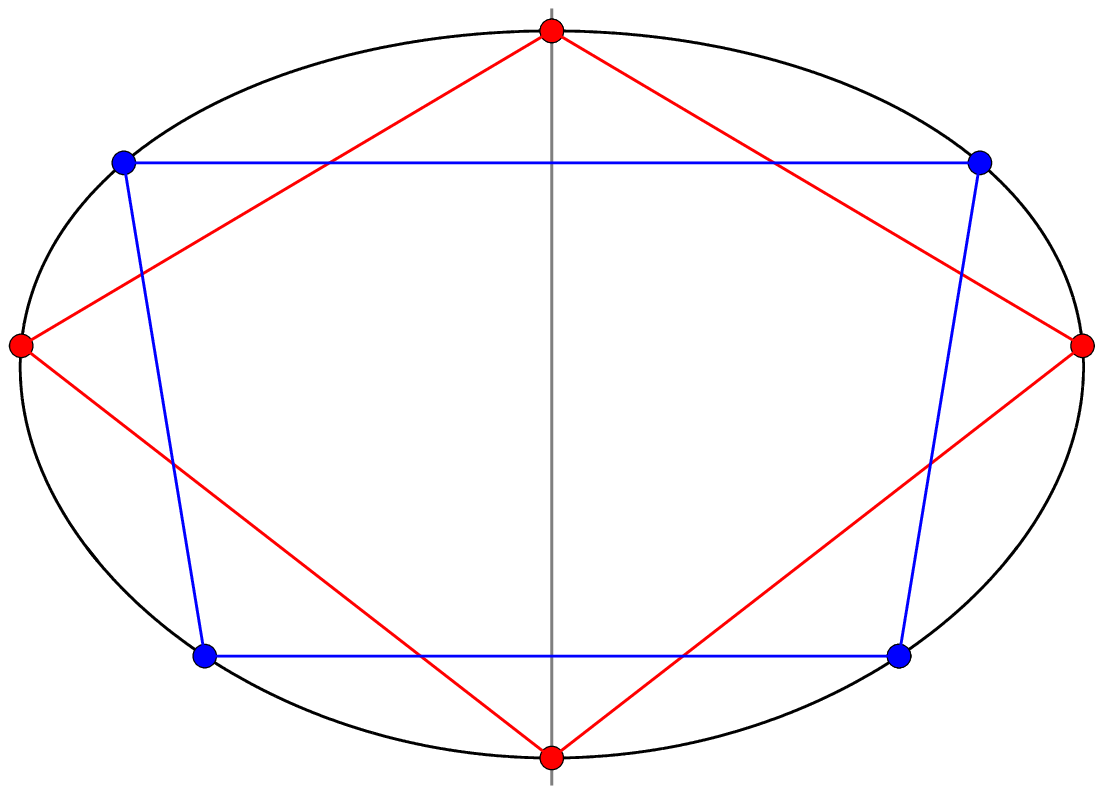,width=2in,clip=}
\else
\fbox{\rule{0in}{1.4in}\hspace{5cm}}
\fi
\end{tabular} \\
\hline\hline
\end{tabular}
\caption{Classification of $(1,q)$-APTs inside bi-axisymmetric
and axisymmetric billiard tables $Q$.
In each case, the difference~$D_q$ is the length of the $(1,q)$-APT
in red minus the length of the $(1,q)$-APT in blue.
The gray lines denote the axis of symmetry.}
\label{table:ClassificationAPTs}
\end{table}

We wanted to study the differences $\Delta^{(1,q)}$ and
the Mather's $\Delta W_{1/q}$,
but instead we will compute the signed differences $D_q$ between
the lengths of the $(1,q)$-APTs.
Clearly, $|D_q|\le \Delta^{(1,q)}$.
In some cases, all periodic trajectories are axisymmetric,
and so $\Delta^{(1,q)} = \Delta W_{1/q} = |D_q|$.
See Proposition~\ref{prop:melnikov_boundary}.

We will fix the semi-minor axis~$b$ and
the degree $n$ in our numerical experiments.
That is, we will study the dependence of $D_q = D_q(\epsilon)$ on
the perturbative parameter~$\epsilon$ and the period~$q$.
The quantity $D_q(\epsilon)$ is analytic at $\epsilon=0$ because
all $(1,q)$-APTs are so.
On the contrary, the period~$q$ is a singular parameter of this problem
because $D_q$ is exponentially small in $q$.
Thus, we will deal with:
\begin{itemize}
\item The \emph{regular} case, where we study the asymptotic behavior of
$D_q(\epsilon)$ when $\epsilon\to 0$ and $q\ge 3$ is fixed; and
\item
The \emph{singular} case, where we study the asymptotic behavior of
$D_q(\epsilon)$ when $q\to+\infty$ and $\epsilon\in\Rset$ is fixed.
\end{itemize}

We will see that the classical Melnikov method is suitable to
study the regular case but it is not so to study the singular one.
Besides, the Melnikov method gives more information on perturbed ellipses
than on perturbed circles.
The singular case is only studied numerically.

\section{Perturbed ellipses}\label{sec:PerturbEllipses}

In this section we restrict ourselves to the case $0<b<1$.
We begin with the regular case, so
the semi-minor axis $b$, the degree $n\ge 3$,
and the period $q\ge 3$ are fixed,
whereas $\epsilon \to 0^+$.
Since the quantity~$\Delta^{(1,q)}=\Delta^{(1,q)}(\epsilon)$ is analytic
and vanishes at $\epsilon=0$, then
\begin{equation}\label{eq:DeltaExpansion}
\Delta^{(1,q)} = \epsilon \Delta_1^{(1,q)}+\Order(\epsilon^2),
\end{equation}
for some coefficient $\Delta_1^{(1,q)}\in\Rset$.
This coefficient can be computed by using a standard Melnikov method.
In fact, the model tables~\eqref{eq:ModelTables} have been chosen
in such a way that the asymptotic behavior of $\Delta_1^{(1,q)}$
can be determined.
The analytical results for $\Delta_1^{(1,q)}$
in the cubic and quartic perturbations are stated below,
but we need to introduce some notation first.

Given~$m\in[0,1)$,
the \emph{complete elliptic integral of the first kind} is
\begin{equation}
\label{eq:DefinitionCompleteEllipticIntegral} K = K(m) =
\int_{0}^{\pi/2}(1-m \sin^2 \theta)^{-1/2} \rmd\theta.
\end{equation}
We also write $K' = K'(m) = K(1-m)$.

\begin{prop}\label{prop:melnikov_boundary}
If $b\in(0,1)$ and $q\ge 3$, the following properties hold.
\begin{enumerate}
\item
$\Delta_1^{(1,q)} = 0$, for odd $n$ and even $q$.
\item
There exist some constants $c, M_3, M_4, K_4 > 0$,
depending only on $b$, such that
\begin{equation}\label{eq:Delta1_asymp_boundary}
\Delta_1^{(1,q)} \asymp
\begin{cases}
M_3  \rme^{-cq},  & {\text{ for $n=3$ and odd $q$,}}\\
K_4 q\rme^{-2cq}, & {\text{ for $n=4$ and odd $q$,}}\\
M_4 q\rme^{-cq},  & {\text{ for $n=4$ and even $q$,}}
\end{cases}
\end{equation}
when $q \to +\infty$.
Besides, $K_4 = 2M_4$, and
\begin{equation}\label{eq:c_Melnikov}
c = \frac{\pi K(b^2)}{2 K(1-b^2)} = \frac{\pi K'(1-b^2)}{2 K(1-b^2)}.
\end{equation}
\item
If $n=3$ and $q$ is odd or if $n=4$,
then there exists $\tilde{\epsilon}_n = \tilde{\epsilon}_n(b, q) \in I_n$
such that all $(1,q)$-periodic billiard trajectories
inside~\eqref{eq:ModelTables} are axisymmetric when
$\epsilon \in (0,\tilde{\epsilon}_n)$.
In particular, $\Delta^{(1,q)} = \Delta W_{1/q} = |D_q|$ for all
$\epsilon \in (0,\tilde{\epsilon}_n)$.
\end{enumerate}
\end{prop}

See \ref{app:proofMelnikov} for the proof.
The explicit values of $M_4$ and $M_3$ can be found in~\eqref{eq:M3_M4}.
Related computations can be found
in~\cite{PintodeCarvalhoRamirezRos2013}.

\begin{remark}
\label{rem:MelnikovConstants}
Similar results hold for any degree~$n\ge 5$, although it is more
cumbersome to compute the \emph{Melnikov constants} $M_n$
(and $K_n$ if $n$ is even) and the \emph{Melnikov powers} $m_n$ such that
\[
\Delta_1^{(1,q)} \asymp
\begin{cases}
M_n q^{m_n} \rme^{-cq},  & \mbox{for odd $n$ and odd $q$}, \\
K_n q^{m_n} \rme^{-2cq}, & \mbox{for even $n$ and odd $q$}, \\
M_n q^{m_n} \rme^{-cq},  & \mbox{for even $n$ and even $q$},
\end{cases}
\]
as $q \to +\infty$.
The \emph{Melnikov exponent} $c$ does not depend on $n$.
\end{remark}

From the first order formula~\eqref{eq:DeltaExpansion}, we deduce that
\[
\lim_{\epsilon\to 0} \left[{\Delta^{(1,q)}}/{\epsilon\Delta_1^{(1,q)}}\right]=1,
\]
for any fixed $q\ge 3$.
Next, we wonder whether the roles of $\epsilon$ and $q$ are interchangeable;
that is, if
\begin{equation}
\label{eq:SingularLimit}
\lim_{q\to+\infty}\left[{\Delta^{(1,q)}}/{\epsilon\Delta_1^{(1,q)}}\right]=1,
\end{equation}
for any fixed but small enough $\epsilon>0$.

We should compute ${\Delta^{(1,q)}}/{\epsilon\Delta_1^{(1,q)}}$
for big periods~$q$ in order to answer this question,
but instead we compute $|D_q|/{\epsilon\Delta_1^{(1,q)}}$.
Both quotients coincide if $\epsilon$ is small enough,
see Proposition~\ref{prop:melnikov_boundary}.
We do not compute~$|D_q|/{\epsilon\Delta_1^{(1,q)}}$ for the cubic perturbation
and even periods,
because $\Delta_1^{(1,q)}=0$ for $n=3$ and even $q$.

We show the results obtained for the cubic and quartic perturbations
in Figure~\ref{fig:MelnikovComparison}.
These figures are obtained by taking the semi-minor axis~$b=4/5$.
Other values for the semi-minor axis give rise to similar figures.

\begin{figure}
\centering
\subfigure[$n=3$ and odd periods.
           Red: $\epsilon=10^{-10}$.
           Blue: $\epsilon=10^{-30}$.]{
\iffigures
\includegraphics[height=2in]{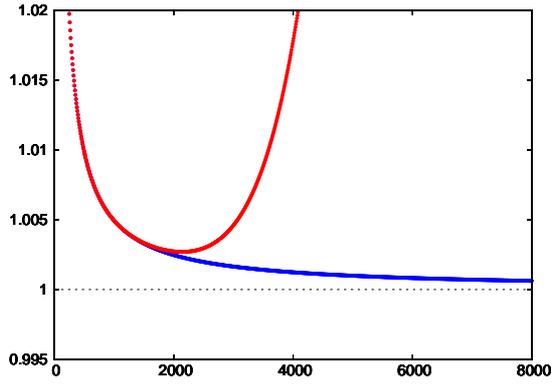}
\else
\fbox{\rule{0in}{1.9in}\hspace{8cm}}
\fi
}
\subfigure[$n=4$ and $\epsilon=10^{-10}$.
           Red: odd periods. Blue: even periods.]{
\iffigures
\includegraphics[height=2in]{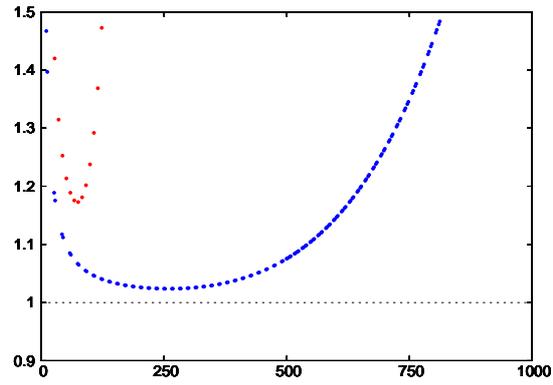}
\else
\fbox{\rule{0in}{1.9in}\hspace{8cm}}
\fi
}
\subfigure[$n=4$ and $\epsilon=10^{-20}$.
           Red: odd periods. Blue: even periods.]{
\iffigures
\includegraphics[height=2in]{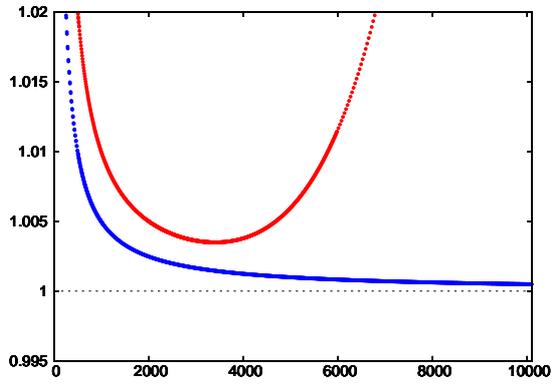}
\else
\fbox{\rule{0in}{1.9in}\hspace{8cm}}
\fi
}
\subfigure[$n=4$ and $\epsilon=10^{-30}$.
           Red: odd periods. Blue: even periods.]{
\iffigures
\includegraphics[height=2in]{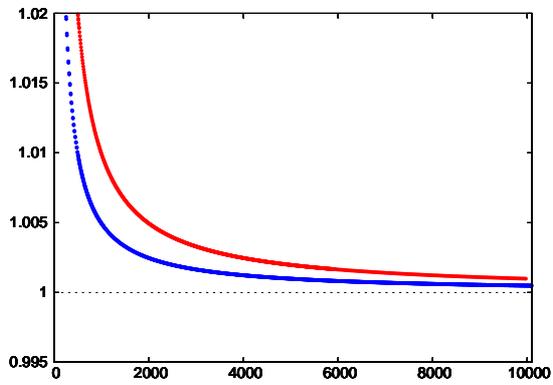}
\else
\fbox{\rule{0in}{1.9in}\hspace{8cm}}
\fi
}
\caption{The quotient $|D_q|/\epsilon\Delta_1^{(1,q)}$ versus the
period $q$ for $b=4/5$.}
\label{fig:MelnikovComparison}
\end{figure}

The Melnikov method does not predict the asymptotic behavior of
$\Delta^{(1,q)}$ in the singular case.
That is, limit~\eqref{eq:SingularLimit} does not hold.
Indeed, if we fix any $\epsilon>0$,
then the quotient~$|D_q|/\epsilon\Delta_1^{(1,q)}$
drifts away from one as $q$ grows.
The drift appears earlier for odd periods
in the case of the quartic perturbation.
As $\epsilon$ gets smaller, the drift appears at larger periods~$q$.
Since the computing time grows quickly when $q$ grows,
the computations to see that drift when $\epsilon$ is very small
are unfeasible with our resources.
This happens, for instance, when $n=4$ and
$\epsilon=10^{-30}$.
See Figure~\ref{fig:MelnikovComparison}.

Based on these numerical experiments, we guess that
there exist some critical exponents~$\nu_n>0$ such that
\begin{equation*}
\Delta^{(1,q)}=\Delta^{(1,q)}(\epsilon) \asymp
\begin{cases}
M_n \epsilon q^{m_n} \rme^{-cq},  & \mbox{for odd $n$ and odd $q$},  \\
K_n \epsilon q^{m_n} \rme^{-2cq}, & \mbox{for even $n$ and odd $q$}, \\
M_n \epsilon q^{m_n} \rme^{-cq},  & \mbox{for even $n$ and even $q$},
\end{cases}
\end{equation*}
when $\epsilon=\Order(q^{-\nu})$, $q \to +\infty$, and $\nu > \nu_n$.
Here, $M_n$, $K_n$, $m_n$, and $c$ are the Melnikov quantities
introduced in Proposition~\ref{prop:melnikov_boundary} and
Remark~\ref{rem:MelnikovConstants}.
We do not give an asymptotic behavior when $n$ is odd and $q$ is even
because we do not have any Melnikov prediction for that case.
Results about exponentially small asymptotic behaviors
based on Melnikov predictions are common in the literature.
For instance, the rapidly forced pendulum is studied
in~\cite{DelshamsSeara1992, DelshamsSeara1997,
GelfreichLazutkin2001,Guardiaetal2010,GuardiaSeara2012}
and some perturbed McMillan maps are studied
in~\cite{DelshamsRRR1996,DelshamsRRR1998,Martinetal2011a,Martinetal2011b}.

Nevertheless, we are interested in a more natural problem.
Namely, the asymptotic behavior of $\Delta^{(1,q)}$ when
$q\to+\infty$ and $\epsilon$ is fixed.
As we have said before, we compute the signed difference~$D_q$
instead of $\Delta^{(1,q)}$.
We have numerically checked that, if $\epsilon$ is small enough,
then there exist a constant~$A\neq 0$, a power~$m\in\Zset$,
and an exponent~$r>0$ such that
\begin{equation}
\label{eq:asymptoticDD}
D_q \asymp A q^m \rme^{-rq},
\end{equation}
as $q\to+\infty$.
In fact, the real behavior is slightly more complicated,
since these three quantities depend on the parity of $q$.
We summarize our results as follows.

\begin{figure}
\centering
\subfigure[$n=3$, $b=9/10$, and $\epsilon=1/10$.]{
\label{fig:HatDq_n3_b9o10_eps1o10_eps1_qodd}
\iffigures
\includegraphics[height=2.39in]{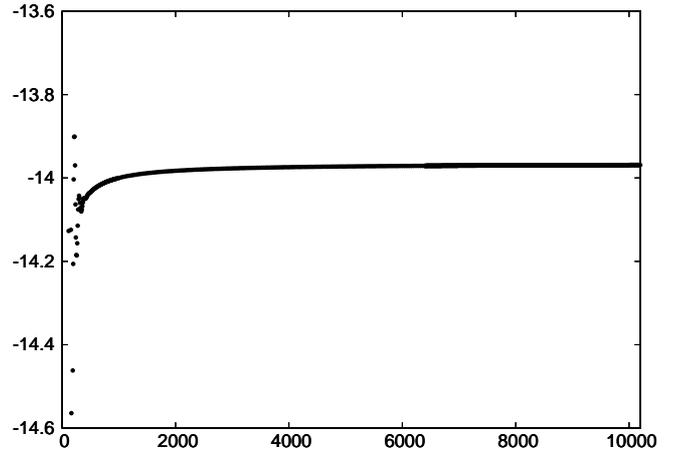}
\else
\fbox{\rule{0in}{2.39in}\hspace{8cm}}
\fi
}
\subfigure[$n=4$, $b=9/10$, $\epsilon=1/10$, and odd periods.]{
\label{fig:HatDq_n4_b9o10_eps1o10_eps1_qodd}
\iffigures
\includegraphics[height=2.39in]{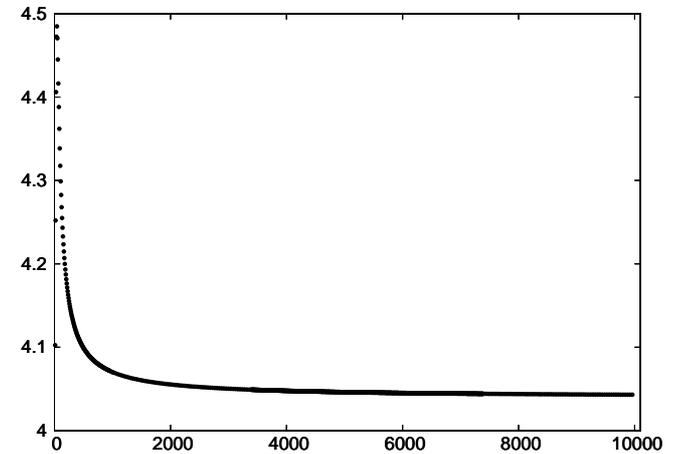}
\else
\fbox{\rule{0in}{2.39in}\hspace{8cm}}
\fi
}
\subfigure[$n=4$, $b=9/10$, $\epsilon=1/10$, and even periods.]{
\label{fig:HatDq_n4_b9o10_eps1o10_eps1_qeven}
\iffigures
\includegraphics[height=2.39in]{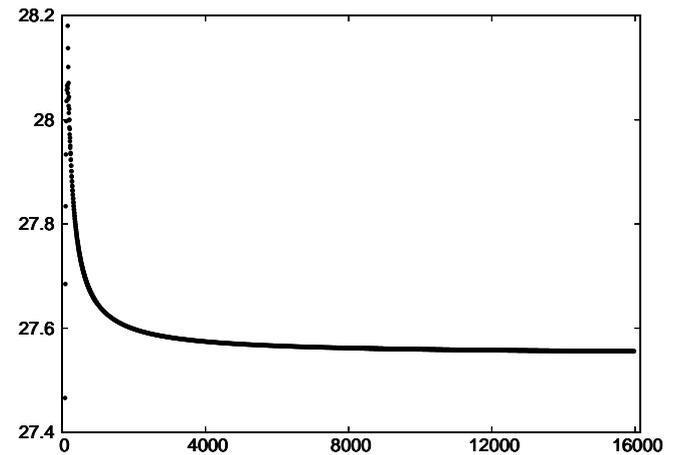}
\else
\fbox{\rule{0in}{2.39in}\hspace{8cm}}
\fi
}
\caption{The normalized differences $\hat{D}_q$ tend to a constant
when $q \to +\infty$ in the ranges $1/2 \le b \le 9/10$
and $0 < \epsilon \le 1/10$ for the cubic and quartic perturbations.
If $n$ is even, then we have to study the even and odd periods
separately.}
\label{fig:HatDqPerturbedEllipse}
\end{figure}

\begin{numres}\label{numres:r}
Fix $b\in(0,1)$ and $n\ge 3$. Let $I_n$ be the maximal interval defined
in~(\ref{eq:MaximalInterval}).
There exists $\hat{\epsilon}_n = \hat{\epsilon}_n(b) \in I_n$
such that the billiard inside~\eqref{eq:ModelTables}
verifies the following properties for all $\epsilon \in (0,\hat\epsilon_n)$.
The Borel transform~(\ref{eq:BorelTransform}) has a
radius of convergence $\rho \in (0,+\infty)$.
Set $r = \rho/2$.
There exist two constants $A,B \neq 0$ such that
\begin{equation}\label{eq:AsymptoticBehaviorDq}
D_q \asymp
\begin{cases}
B q^{-2} \rme^{-2r q}, & \mbox{for even $n$ and odd $q$}, \\
A q^{-3} \rme^{-r q} , & \mbox{otherwise},
\end{cases}
\end{equation}
as $q\to+\infty$.
The quantities $\rho$, $r$, $A$, and $B$ depend on $b$, $\epsilon$, and $n$.
The constant $B$ is defined only when $n$ is even.
Besides, $\lim_{\epsilon\to 0}r = c$,
where $c$ is the Melnikov exponent defined in~\eqref{eq:c_Melnikov}.
\end{numres}

We stated in Conjecture~\ref{conj:General} that the function $A(q)$
that appears in the exponentially small asymptotic formula~(\ref{eq:Aq_Billiard})
is constant when the billiard table belongs to a
certain open set of the space of axisymmetric algebraic curves.
Thus, the previous numerical result fits perfectly into the conjecture.

It is interesting to compare the Melnikov
formulas~\eqref{eq:Delta1_asymp_boundary} with the asymptotic
formulas~\eqref{eq:AsymptoticBehaviorDq}. The asymptotic behavior of
$D_q$ does not depend on the parity of $q$ when $n$ is odd. The
exponents~$c$ and $r$ play the same role. Finally, the factors
$q^{-2}$ and $q^{-3}$ in~\eqref{eq:AsymptoticBehaviorDq} can not be
directly guessed from the Melnikov formulas.

Let us describe our numerical experiments.
First, once the exponent~$r$ is determined (see next paragraph),
we compute the normalized differences
\begin{equation}\label{eq:NormalizedDqPerturbedEllipses}
\hat{D}_q =
\begin{cases}
q^{2} \rme^{2rq}D_q, & \mbox{for even $n$ and odd $q$}, \\
q^{3} \rme^{ rq}D_q, & \mbox{otherwise}.
\end{cases}
\end{equation}
We have checked that these normalized differences $\hat{D}_q$
tend to some constant as $q\to+\infty$
in the ranges $1/2 \le b\le 9/10$ and $0<\epsilon\le 1/10$.
Figure~\ref{fig:HatDqPerturbedEllipse} shows that behavior
on three different scenarios for $b=9/10$ and $\epsilon=1/10$.

Let us explain how to compute the exponent~$r=r(b,\epsilon,n)$.
First, we assume that the exponentially small
asymptotic formula~\eqref{eq:asymptoticDD}
can be refined as
\[
D_q \asymp q^m \rme^{-rq}\sum_{j\ge 0} d_j q^{-2j},
\]
for some asymptotic coefficients~$d_j\in\Rset$ with $d_0 = A \neq 0$.
This assumption is based on similar refined asymptotic formulas for the
splitting of separatrices of analytic maps~\cite{Gelfreich1999,Martinetal2011a}.
By taking logarithms, we find the asymptotic expansion
\[
\frac{1}{q}\log\left(q^{-m}D_q\right)
\asymp -r +\frac{1}{q}\log\left(\sum_{j\ge 0}\frac{d_j}{q^{2j}}\right)
\asymp -r +\sum_{j\ge 0}\frac{\alpha_j}{q^{2j+1}},
\]
for some coefficients~$\alpha_j\in\Rset$.
Therefore, we can compute $r$ by using a Neville extrapolation method
from a sequence of differences~$D_q$.
The longer the sequence, the more correct digits in $r$.
We obtain 15 correct digits with the following choices.
We fix the perturbed ellipse~$Q$, that is,
we fix $b\in (0,1)$, $\epsilon\in\Rset$, and $n\ge 3$.
Second, we fix the class of periods $q$, so that we are on one of the cases of
Table~\ref{table:ClassificationAPTs}.
That is, $q=q(k)=2k+1$, $q=q(k)=4k+2$, $q=q(k)=4k$, or $q=q(k)=2k$.
Then, we compute $D_q$ with at least 400 correct digits
on an increasing sequence of 500 periods~$q_i = q(k_i)$,
with $k_i = k_0 + 10 i$.
The initial period~$q_0$ is chosen to be
big enough so that $|D_{q_0}|\le 10^{-3000}$.
In fact, we perform the Neville extrapolation with two
different sequences of 500 periods each which allows us to
determine the number of correct digits in the final result.
The power~$m\in\{-2,-3\}$ is found by trial-and-error.

In Figure~\ref{fig:Exponent_r}, we display the exponent~$r=r(\epsilon)$ for several
values of $b$ for the cubic and quartic perturbations.
We also depict the Melnikov exponent~$c$ at $\epsilon=0$ in full circles.
Note that $\lim_{\epsilon\to 0} r=c$ and
$r$ is decreasing in $\epsilon$.

\begin{figure}[!t]
\centering
\subfigure[$n=3$.]{
\iffigures
\includegraphics[height=2.3in]{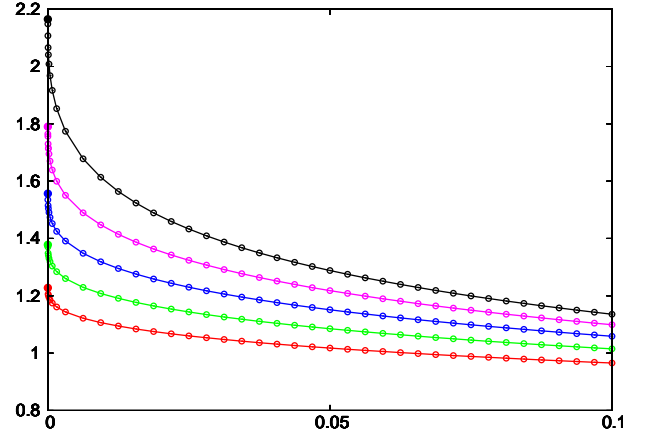}
\else
\fbox{\rule{0in}{2.3in}\hspace{8cm}}
\fi
}
\subfigure[$n=4$.]{
\iffigures
\includegraphics[height=2.3in]{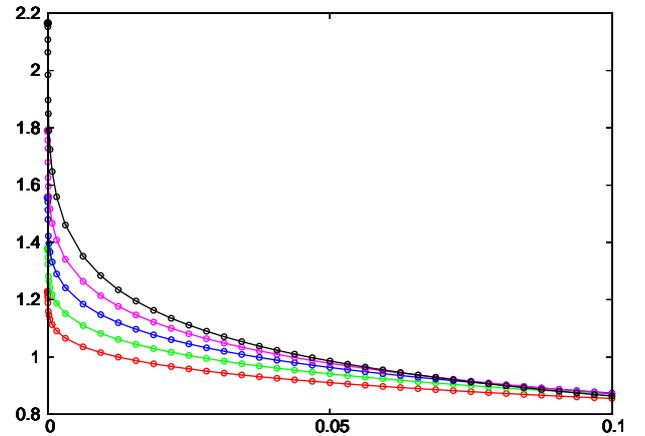}
\else
\fbox{\rule{0in}{2.3in}\hspace{8cm}}
\fi
}
\caption{The exponent $r$ versus the perturbative parameter $\epsilon$.
We also display the points $(0,c)$ in solid circles,
where $c$ is the Melnikov exponent.
We note that $\lim_{\epsilon \to 0^+} r = c$.
Red: $b=1/2$. Green: $b=3/5$. Blue: $b=7/10$.
Magenta: $b=4/5$. Black: $b=9/10$.}
\label{fig:Exponent_r}
\end{figure}

Next, let us relate the exponent $r$ with the radius of convergence $\rho$
of the Borel transform~\eqref{eq:BorelTransform}.
Once fixed $b\in(0,1)$, $\epsilon\in\Rset$, and $n\ge 3$,
we compute $\rho=\rho(b,\epsilon,n)$ as follows.

First, we compute the length~$L^{(1,q)}$
of one of the $(1,q)$-APTs inside $Q$
for the same sequences of periods~$(q_i)$ used for computing $D_q$.
We use a precision of 3000 correct digits in these computations.
The choice of the APT does not matter,
since $|D_{q_i}|\le 10^{-3000} $ for any period~$q_i\ge q_0$.
Second, we obtain the first asymptotic coefficients~$l_j$
in the expansion~\eqref{eq:asymptoticLength}
by using the Neville extrapolation method again.
Third, we determine the number of correct digits in each coefficient~$l_j$
by comparing the results obtained with two different sequences of periods.
The number of correct digits in $l_j$ decreases as $j$ grows.
We always get at least 1500 correct digits in $l_0$
and at least 40 correct digits in $l_{450}$.

It turns out that the coefficients $l_j$ increase at a factorial rate,
so the asymptotic series~\eqref{eq:asymptoticLength} is Gevrey-1
and diverges for any $q$.
Indeed, we have found that there exist
a radius of convergence $\rho = \rho(b,\epsilon,n) > 0$
and a constant $\gamma = \gamma(b,\epsilon,n) > 0$ such that
\[
\hat{l}_j \asymp \gamma j^{-2} \rho^{-2j}, \qquad j \to +\infty,
\]
provided $\epsilon$ is small enough.
That is, the Borel transform~(\ref{eq:BorelTransform})
has a singularity at $z = \rho$.
In particular,
\[
\rho=\lim_{j\to \infty}
\left|{\hat l_j}/{\hat l_{j+1}}\right|^{1/2}.
\]
We see this asymptotic behavior in Figure~\ref{fig:quoBorel}.

\begin{figure}
\iffigures
\centering
\includegraphics[height=2.4in]{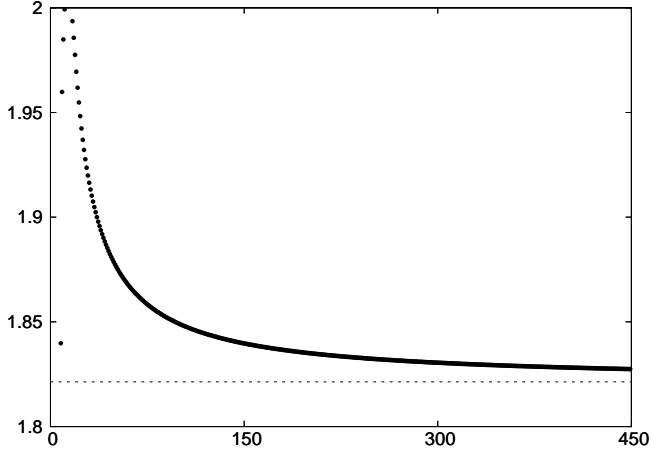}
\else
\fbox{\rule{0in}{2.4in}\hspace{8cm}}
\fi
\caption{${\left|{\hat l_j}/{\hat l_{j+1}}\right|}^{1/2}$
versus $j$ for $b=9/10$, $\epsilon=1/20$, and $n=4$.
The dashed line corresponds to the limit value $\rho$
obtained by extrapolation.}
\label{fig:quoBorel}
\end{figure}

The rough approximation
\[
\rho \approx \left|{\hat l_{449}/\hat l_{450}}\right|^{1/2}
\]
only gives about 3 correct digits.
If we use an extrapolation method based on the asymptotic expansion
\[
{\left|{\hat l_j}/{\hat l_{j+1}}\right|}^{1/2}
\asymp \rho + \sum_{i>0} \beta_i j^{-i},
\]
the radius of convergence is improved up to 8 correct digits.
This is the limit value plotted in Figure~\ref{fig:quoBorel}.
We stress that this asymptotic expansion in powers of $j^{-1}$
is probably wrong since the extrapolation becomes unstable after
a few steps.

The radius of convergence~$\rho$
does not depend on the parity of the periods of the sequence~$(q_i)$.
Thus, the value of $\rho$ obtained by sequences
of different parities must coincide.
This provides another validation to the number of correct digits of $\rho$.

\begin{remark}\label{remark:CostComputingR}
Taking into account relation $r = \rho/2$, we have two different
ways of computing the exponent~$r$, the direct method and the Borel one.
The Borel method is computationally much cheaper.
Indeed, the precision required to compute the differences~$D_{q_i}$ increases
along the periods~$q_i$ whereas
it is fixed when computing the lengths~$L^{(1,q_i)}$.
\end{remark}

At this point, we have established the relations among
the Melnikov exponent~$c$, the exponent~$r$,
and the radius of convergence~$\rho$.
Next, we relate $c$ with the distance~$\delta$ provided
by our candidate for limit problem,
since we are only able to analytically compute $\delta$
for unperturbed ellipses.

\begin{prop}\label{prop:deltaEllipse}
Let $b \in (0,1)$.
Let $\kappa(s)$ be the curvature of the unperturbed ellipse
$E = \{ (x,y) \in \Rset^2 : x^2 + y^2/b^2 = 1 \}$
in some arc-length parameter~$s$.
Let $\xi \in \Rset/\Zset$ be the angular variable defined
by~(\ref{eq:xiDefinition}).
Let $\delta$ be the distance of the set of singularities and zeros of the
curvature~$\kappa(\xi)$ to the real axis.
Then $2\pi\delta=c$, where $c$ is the
Melnikov exponent defined in~\eqref{eq:c_Melnikov}.
\end{prop}

This proposition is proved in~\ref{app:proofdeltaEllipse}.

We have numerically checked that the inequality $r < 2\pi\delta$
holds in the ranges $1/2 \le b \le 9/10$ and $0 < \epsilon \le 1/10$
for the cubic and quartic perturbations. The case $b = 4/5$ is
displayed in Figure~\ref{fig:Comparison_r_2pidelta_ellipses}.

\begin{remark}\label{remark:Computation_delta}
The distance $\delta$ is numerically computed as follows.
First, we write the curvature $\kappa$ and the length element $\rmd s$
of the perturbed ellipse~(\ref{eq:ModelTables}) in terms of
the vertical coordinate $y$.
It turns out that there exist three polynomials $r(y)$, $p(y)$ and $q(y)$
such that
\[
\kappa^{2/3} \rmd s = g(y) \rmd y := \frac{p^{2/3}(y)}{\sqrt{r(y)q(y)}} \rmd y.
\]
For instance, $r(y) = 1 - y^2/b^2 - \epsilon y^n$
and $\deg[p] = \deg[q] = 2n-2$.
Let $y_\pm$ be the roots of $r(y)$ that tend to $\pm b$ when $\epsilon \to 0$.
The points $(0,y_\pm)$ are the vertices on the vertical axis of the perturbed
ellipse~(\ref{eq:ModelTables}).
Then $\delta = |\Im \xi_\star|/C$, where
\[
C = \int_Q \kappa^{2/3} \rmd s = 2 \int_{y_-}^{y_+} g(y) \rmd y, \qquad
\xi_\star = \int_0^{y_\star} g(y) \rmd y,
\]
and $y_\star \neq y_\pm$ is the root of $p(y)$, $q(y)$, or $r(y)$ that gives the
closest singularity $\xi_\star \in \Cset/\Zset$ to the real axis.
That is, $y_\star$ minimizes $\delta$.
The path from $y = 0$ to $y = y_\star$ in the second integral should be
contained in an open simply connected subset of the complex plane
where the function $g(y)$ is analytic.
See~\ref{app:proofCircularBehaviour} for more details about
the function $g(y)$ and their domain of analyticity,
although that appendix deals with perturbed circles only.
\end{remark}

\begin{figure}
\iffigures
\centering
\includegraphics[height=2.4in]{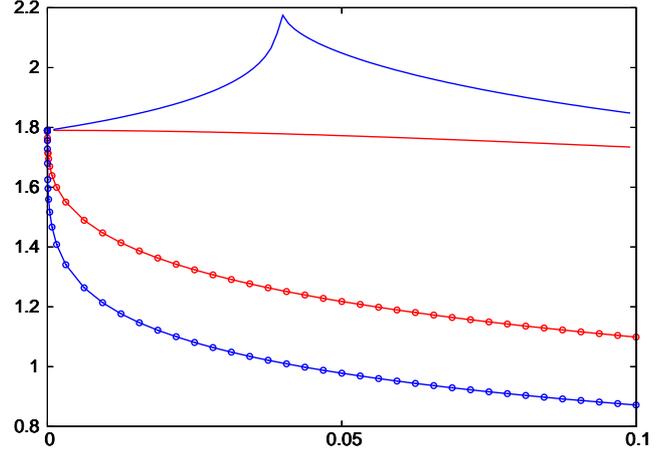}
\else
\fbox{\rule{0in}{2.4in}\hspace{8cm}}
\fi
\caption{The exponent $r$ (continuous lines with points) and
the quantity $2\pi\delta$ (continuous lines)
versus $\epsilon$ for $b= 4/5$.
Red: $n=3$. Blue: $n=4$.}
\label{fig:Comparison_r_2pidelta_ellipses}
\end{figure}

We note that the cusp that appears in the graph of $2\pi\delta$
for the quartic perturbation correspond to a perturbative parameter $\epsilon$
for which two different roots of $p(y)$, $q(y)$, or $r(y)$ give rise to
the same $\delta = |\Im \xi_\star|/C$.

\section{Perturbed circles}\label{sec:PerturbCircles}

In this section, we take $b=1$ in the model tables~\eqref{eq:ModelTables}.
This setting is harder than the one of the perturbed ellipses
both in the regular and singular cases.
Let us explain it.

We begin with the regular case, so we fix
the degree $n\ge 3$ and the period $q\ge 3$
whereas $\epsilon$ tends to zero.
First, we note that the Melnikov exponent $c$ in~\eqref{eq:c_Melnikov}
tends to infinity as $b$ tends to one,
since $K(0) = \pi/2$ and $\lim_{m \to 1^-} K(m) = +\infty$.
This suggests that the Melnikov method gives little information
for perturbed circles.
In fact, in~\cite{RamirezRos2006}, it is proved that
the first order coefficient $\Delta_1^{(1,q)}$
in~\eqref{eq:DeltaExpansion}
vanishes for every period~$q\notin\mathcal{Q}_n$, where
\[
\mathcal{Q}_n = \begin{cases}
 \{3,5,\ldots,n-2,n\}, & \text{for odd $n$,}\\
 \{2,4,\ldots,n-2, n\}\cup \{2,3,\ldots, n/2\}, & \text{for even $n$.}
\end{cases}
\]
We might use a higher order Melnikov method to look for an
order $k = k(n,q) \in \Nset$ such that
\[
\Delta^{(1,q)}=\epsilon^k\Delta_k^{(1,q)}+\Order(\epsilon^{k+1}),
\]
with $\Delta_k^{(1,q)}\neq 0$.
This Melnikov computation is not easy
so we have performed a numerical study instead.
As before, we do not study $\Delta^{(1,q)}$ but the difference $D_q$.

\begin{numres}\label{numres:DqOrderk}
Set
\[
k = k(n,q) =
\begin{cases}
1 + 2 \left\lceil \frac{q-n}{2n} \right\rceil, &
\text{for odd $n$ and odd $q$,} \\
2\lceil  q/2n \rceil, & \mbox{for odd $n$ and even $q$}, \\
 \lceil 2q/n  \rceil, & \mbox{for even $n$ and odd $q$}, \\
 \lceil  q/n  \rceil, & \mbox{for even $n$ and even $q$}.
\end{cases}
\]
If $n\ge 3$ and $q\ge 2$, then there exists
$d_k = d_k(n,q) \neq 0$ such that
\begin{equation}\label{eq:DqOrderk}
D_q(\epsilon) = d_k \epsilon^k+\Order(\epsilon^{k+1}).
\end{equation}
\end{numres}

This numerical result has two nice consequences on the breakup of
the resonant caustics of the circular billiard under
the perturbation $x^2+y^2+\epsilon y^n=1$
with any fixed degree $n \ge 3$.
First, \emph{all $(1,q)$-resonant caustics break up},
because, once fixed the period $q\ge 2$,
$\Delta^{(1,q)}\neq 0$ for $\epsilon$ small enough.
Second, \emph{there are breakups of any order},
because the map $q \mapsto k(n,q) \in \Nset$ is exhaustive.

We numerically compute the order $k$ in~\eqref{eq:DqOrderk}
by noting that
\[
k \simeq \log\left(\frac{D_q(\epsilon)}{D_q(\epsilon/\rme)}\right).
\]
For instance, if $n=7$, $q=36$, and $\epsilon=10^{-10}$,
then we obtain the approximation
\[
k\simeq  5.99999999999999999401\ldots,
\]
so $k=6$.
We have tested the formulas listed in Numerical
Result~\ref{numres:DqOrderk}
for all degrees $3\le n\le 8$ and all periods $3\le q \le 100$.
Note that, once fixed $n$,
\begin{equation} \label{eq:k_Asymp_qovern}
k = k(n,q) \asymp
\begin{cases}
2q/n, & \text{for even $n$ and odd $q$,} \\
q/n,  & \text{otherwise,}
\end{cases}
\end{equation}
as $q\to+\infty$.
Next, we focus on the singular case.

\begin{numres}\label{numres:DqCirclesSingular}
Fix $n\ge 3$. Let $I_n$ be the maximal interval defined
in~(\ref{eq:MaximalInterval}).
If $\epsilon\in I_n$, then the Borel transform~\eqref{eq:BorelTransform}
has a radius of convergence $\rho \in (0,+\infty)$ .
Set $r=\rho/2$.
There exist two non-zero quasiperiodic functions
$A(q)$ and $B(q)$ such that
\[
D_q \asymp
\begin{cases}
B(q) q^{-2} \rme^{-2r q}, & \mbox{for even $n$ and odd $q$}, \\
A(q) q^{-3} \rme^{-r q} , & \mbox{otherwise},
\end{cases}
\]
as $q \to +\infty$.
Besides, there exists $\chi_n \in \Rset$ such that
\begin{equation} \label{eq:rLogarithm}
r= \frac{|\log \epsilon|}{n} + \chi_n + \order(1)
\end{equation}
as $\epsilon\to 0$.
Finally, there exist a partition
$I_n = C_n \cup P_n \cup R_n$
satisfying the following properties.
\begin{enumerate}
\item
$C_n$ and $P_n$ are open subsets of $I_n$,
whereas $R_n$ is a set of isolated perturbative parameters.
\item
If $\epsilon\in C_n$, both functions $A(q)$ and $B(q)$ are constant.
\item
If $\epsilon\in P_n$, both functions $A(q)$ and $B(q)$ are periodic.
Namely, they have the form
\[
A(q) =  a\cos(2\pi\beta q), \quad
B(q) = \bar{b} + b\cos(4\pi\beta q),
\]
for some average $\bar{b} \neq 0$, some amplitudes $a,b > 0$,
and some ``shared'' frequency $\beta > 0$.
We note that $\bar{b} \neq b/2$.
\end{enumerate}
\end{numres}

All these numerical results strongly support Conjecture~\ref{conj:General}.
For instance, we conjectured that the function $A(q)$ is
either constant: $A(q) \equiv a/2$, or periodic:
$A(q) = a\cos(2\pi\beta q)$ in open sets of the space of
axisymmetric algebraic curves,
whereas all other cases are phenomena of co-dimension one.
This claim agrees with the fact that $C_n$ and $P_n$ are
open subsets of $I_n$, whereas $R_n$ only contains the perturbative
parameters where a transition between constant and periodic
cases takes place.

The functions $A(q)$ and $B(q)$ and the exponent $r$
depend on the degree $n$ and the perturbative parameter $\epsilon$,
although $B(q)$ is defined only for even $n$.
Both functions $A(q)$ and $B(q)$ ``share'' the
frequency in the periodic case.
To be precise, the frequency of $B(q)$ is twice the frequency
of $A(q)$.
It makes sense because the exponent in the
asymptotic formula containing the function $B(q)$ is also twice
the exponent in the one containing $A(q)$.

The logarithmic behavior of the exponent $r$ stated
in~\eqref{eq:rLogarithm} is closely related to the asymptotic
formula~\eqref{eq:k_Asymp_qovern}.
Indeed, if we roughly try to fit the regular behavior~\eqref{eq:DqOrderk}
when $\epsilon\to 0$ with the singular behavior
$D_q = \Order(q^m\rme^{-rq})$ when $q\to+\infty$,
then we get
\[
\Order(q^m\rme^{-rq})=
D_q=
\Order(\epsilon^k)\simeq
\Order( \epsilon^{q/n}) =
\Order(\rme^{-q|\log\epsilon|/n}),
\]
so we guess that $r\simeq |\log\epsilon|/n$.
This reasoning is informal but it is confirmed by our experiments.
Let us describe them.

We have set $\epsilon \in I_n \cap \Qset$ in all the experiments.
First, we do so because our multiple-precision computations
become a bit faster for rational perturbative parameters.
There is a second reason for that choice.
Namely, we change the precision very often along our computations,
and rational values of $\epsilon$ are not affected by such changes,
because they are stored as exact numbers.
We have also tried to deal with ``big'' perturbations
in order to stress that our results are not perturbative,
but we recall that $\epsilon$ should be smaller than the
singular value~(\ref{eq:SingularEpsilon}) when $n$ is odd.

\begin{figure*}
\centering
\subfigure[$\hat{D}_q$ versus $q$ for $n=4$, $\epsilon=1$, and odd $q$.]{
\label{fig:HatDq_n4_eps1_qodd}
\iffigures
\includegraphics[height=1.6in]{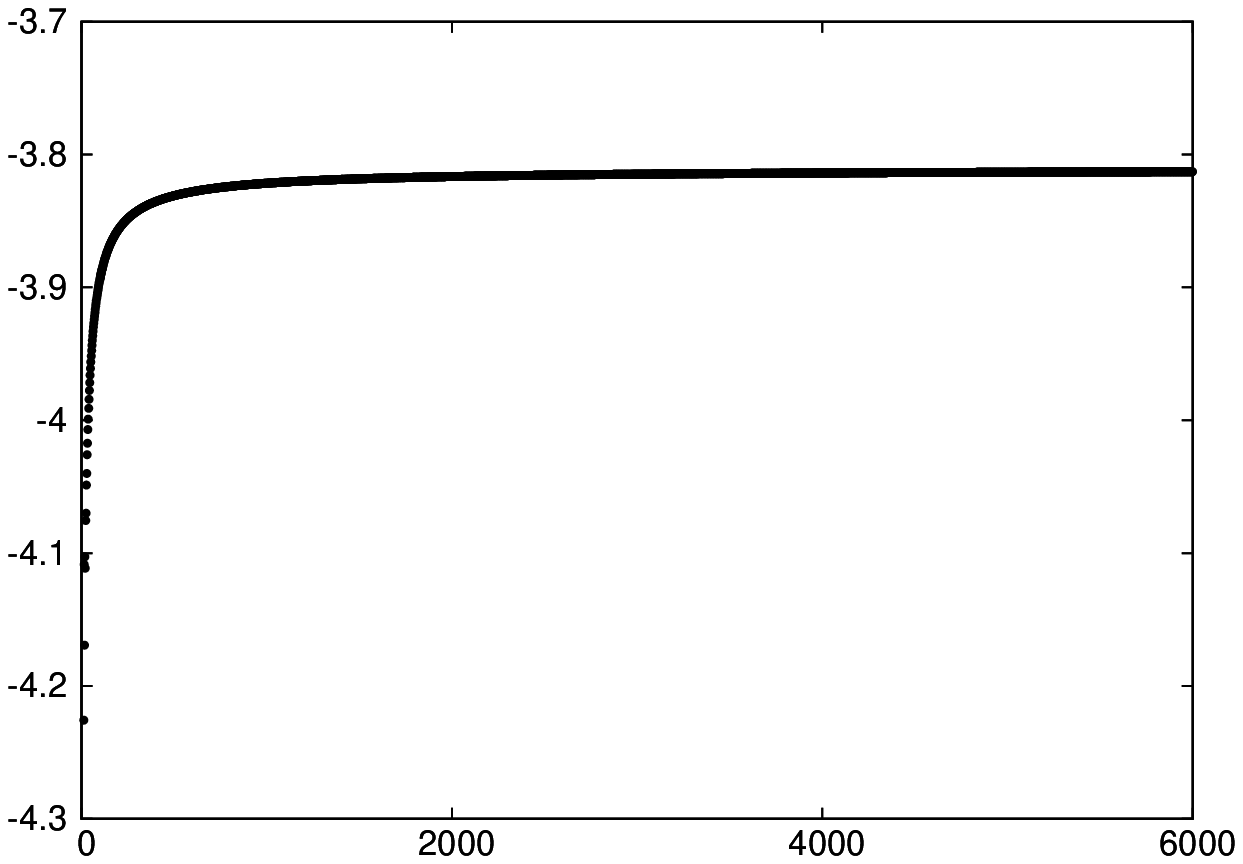}
\else
\fbox{\rule{0in}{1.6in}\hspace{5cm}}
\fi
}
\hfill
\subfigure[$\hat{D}_q$ versus $q$ for $n=4$, $\epsilon=1$, and even $q$.]{
\label{fig:HatDq_n4_eps1_qeven}
\iffigures
\includegraphics[height=1.6in]{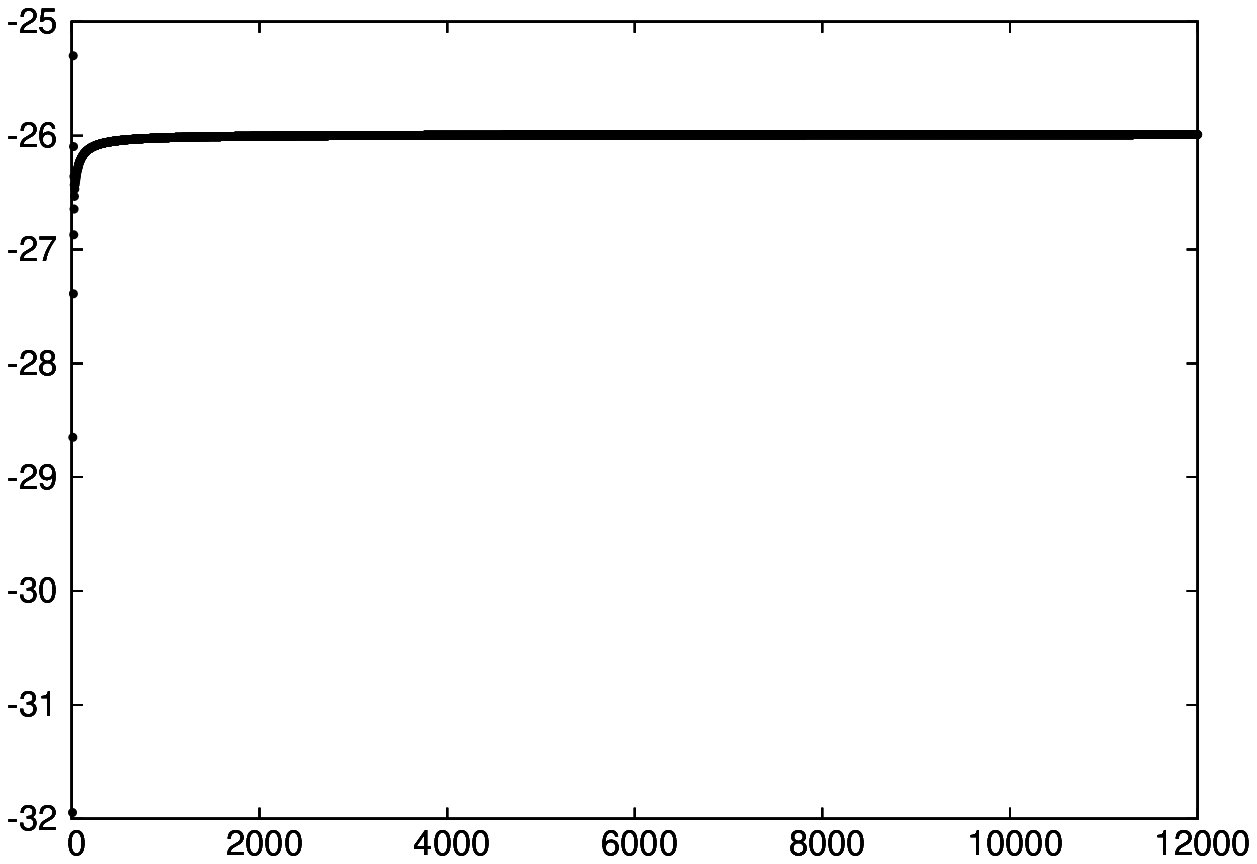}
\else
\fbox{\rule{0in}{1.6in}\hspace{5cm}}
\fi
}
\hfill
\subfigure[$\hat{D}_q$ versus $q$ for $n=7$ and $\epsilon=1/1280$.]{
\label{fig:HatDq_n7_eps1o1280}
\iffigures
\includegraphics[height=1.6in]{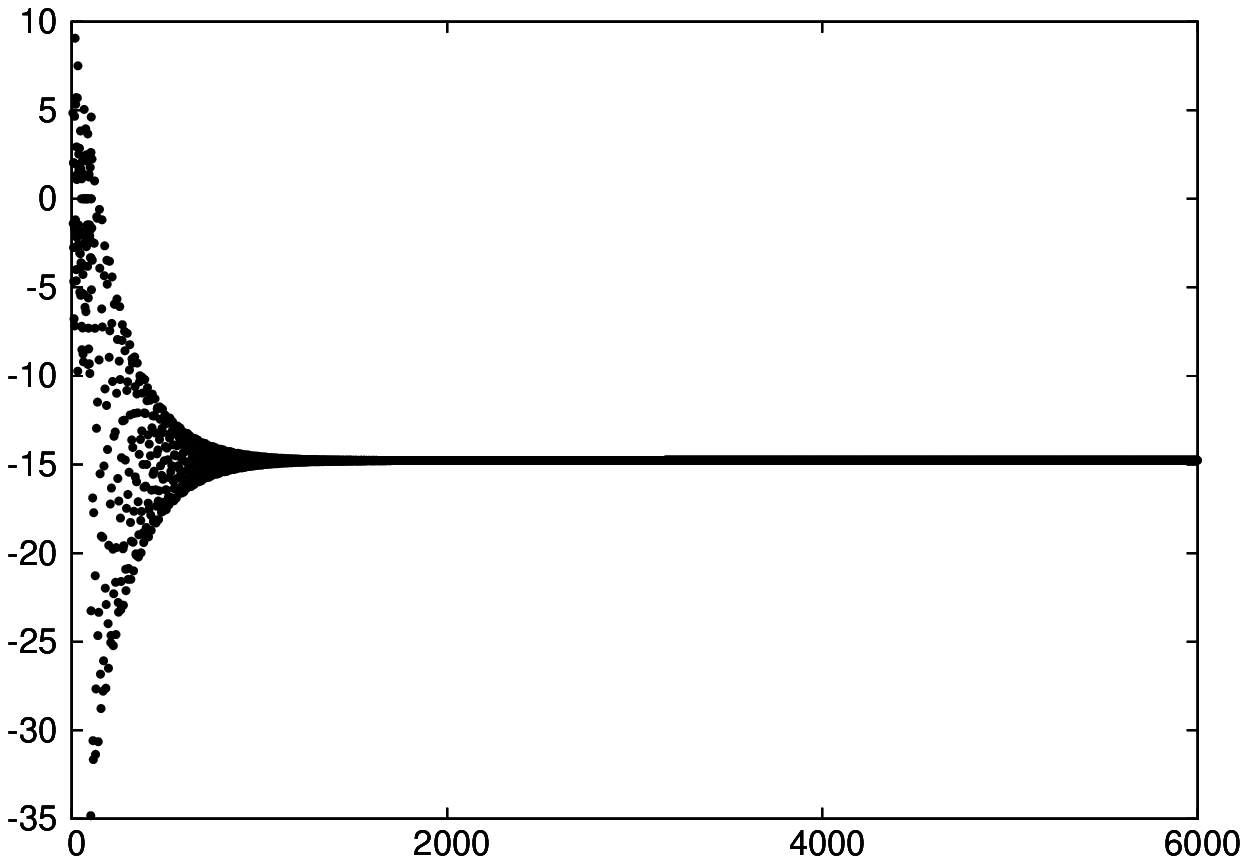}
\else
\fbox{\rule{0in}{1.6in}\hspace{5cm}}
\fi
}
\caption{Examples with a constant asymptotic behavior of the
normalized differences $\hat{D}_q$.}
\label{fig:HatDq_Constant}
\end{figure*}

\begin{figure*}
\centering
\subfigure[$\hat{D}_q$ versus $q$ for $n=3$ and $\epsilon=1/3$.]{
\label{fig:HatDq_n3_eps1o3}
\iffigures
\includegraphics[height=1.6in]{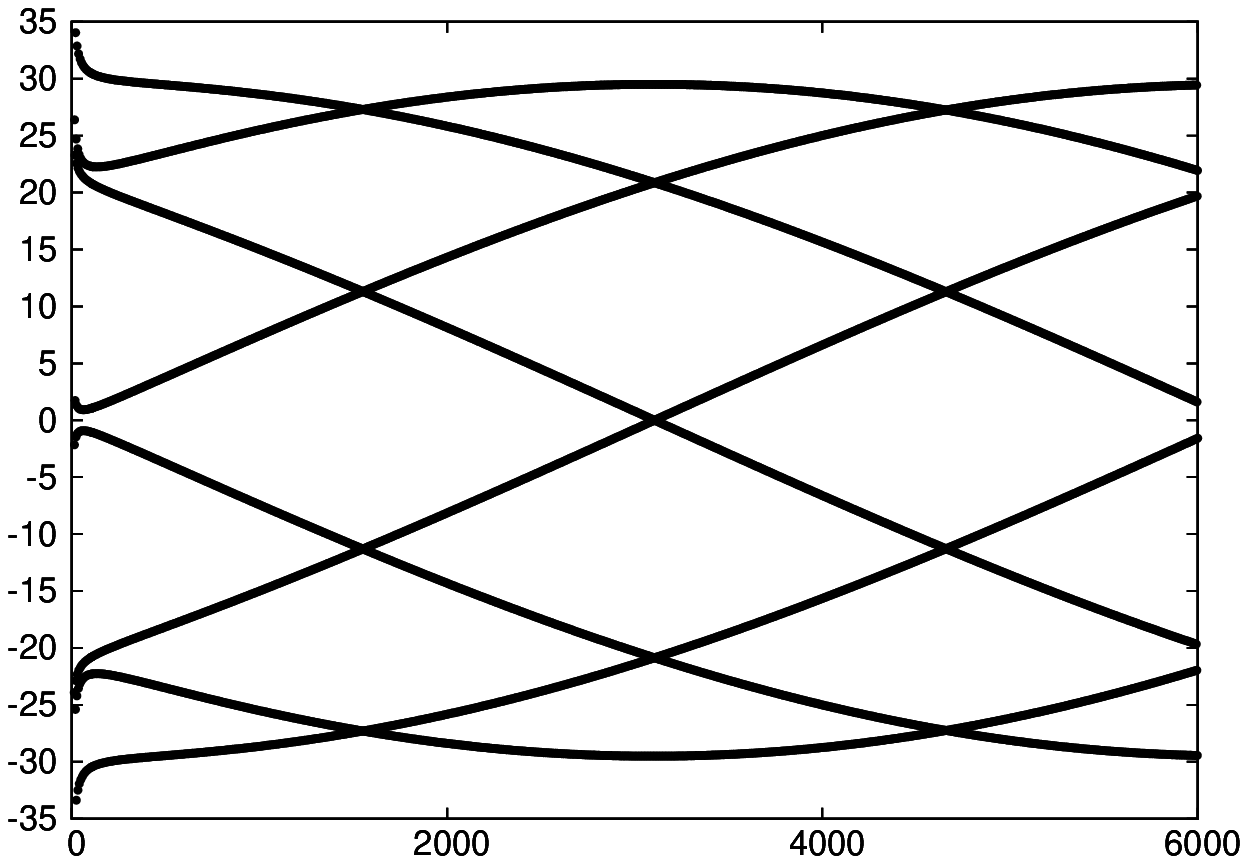}
\else
\fbox{\rule{0in}{1.6in}\hspace{5cm}}
\fi
}
\hfill
\subfigure[DFT of $\hat{D}_q$ for $n=3$ and $\epsilon=1/3$.]{
\label{fig:DFTHatDq_n3_eps1o3}
\iffigures
\includegraphics[height=1.6in]{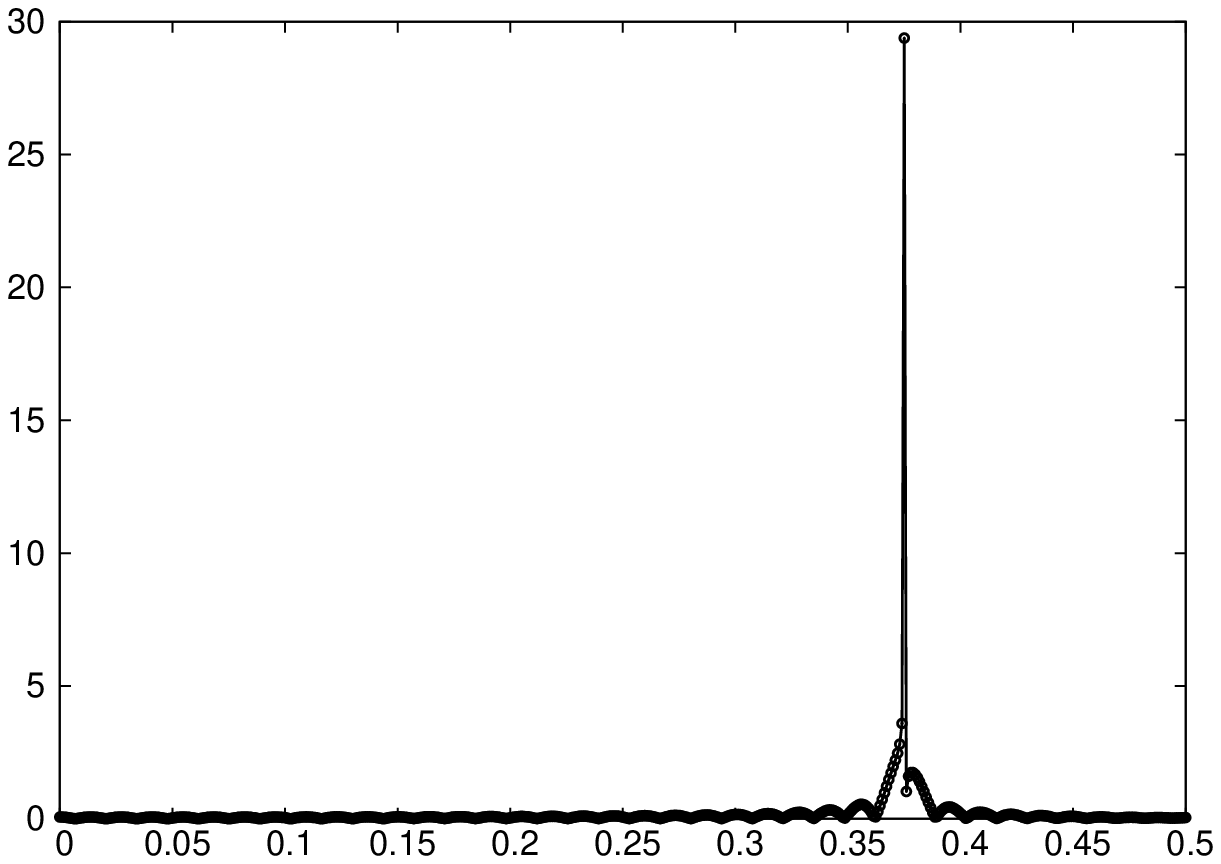}
\else
\fbox{\rule{0in}{1.6in}\hspace{5cm}}
\fi
}
\hfill
\subfigure[$\hat{D}_q - A(q)$ versus $q$ for $n=3$ and $\epsilon=1/3$.]{
\label{fig:HatDqminusAq_n3_eps1o3}
\iffigures
\includegraphics[height=1.6in]{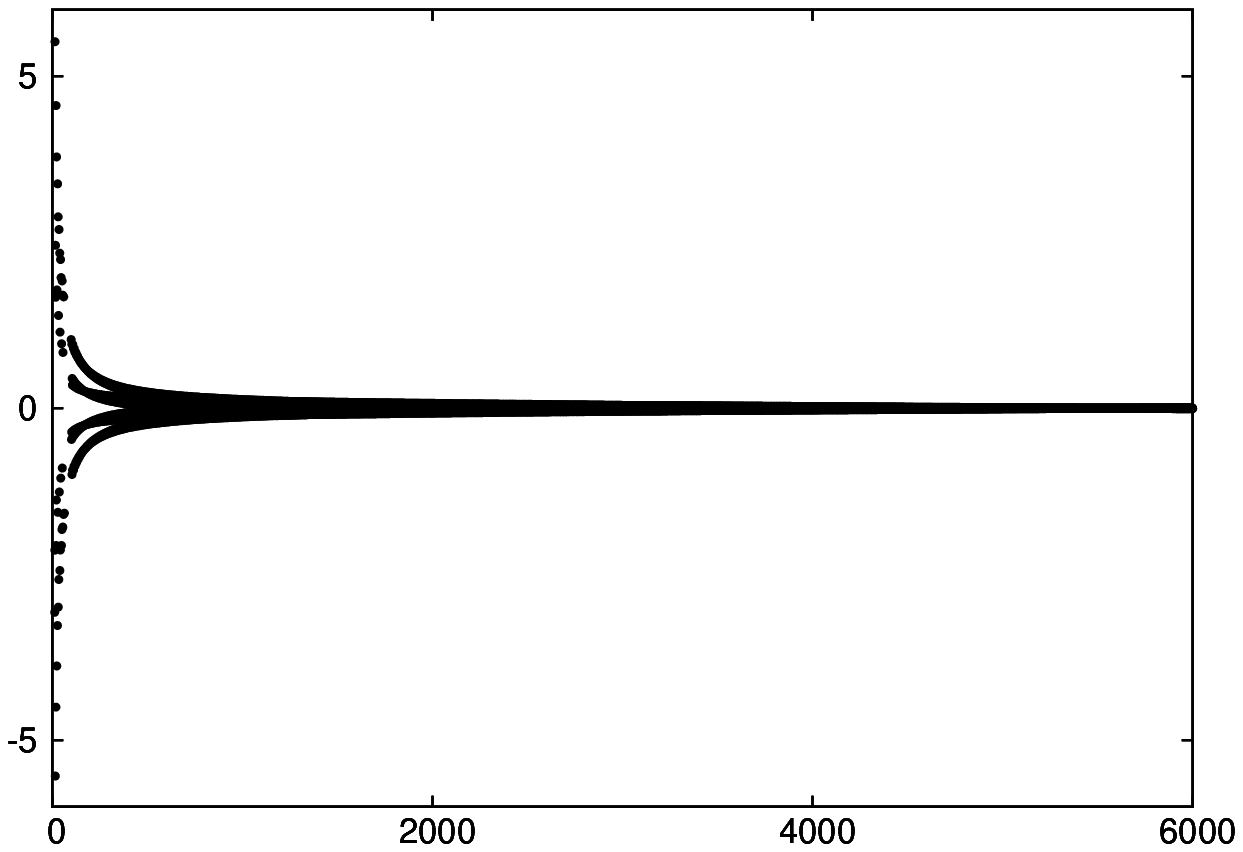}
\else
\fbox{\rule{0in}{1.6in}\hspace{5cm}}
\fi
}
\subfigure[$\hat{D}_q$ versus $q$ for $n=6$, $\epsilon=1$, and even $q$.]{
\label{fig:HatDq_n6_eps1_qeven}
\iffigures
\includegraphics[height=1.6in]{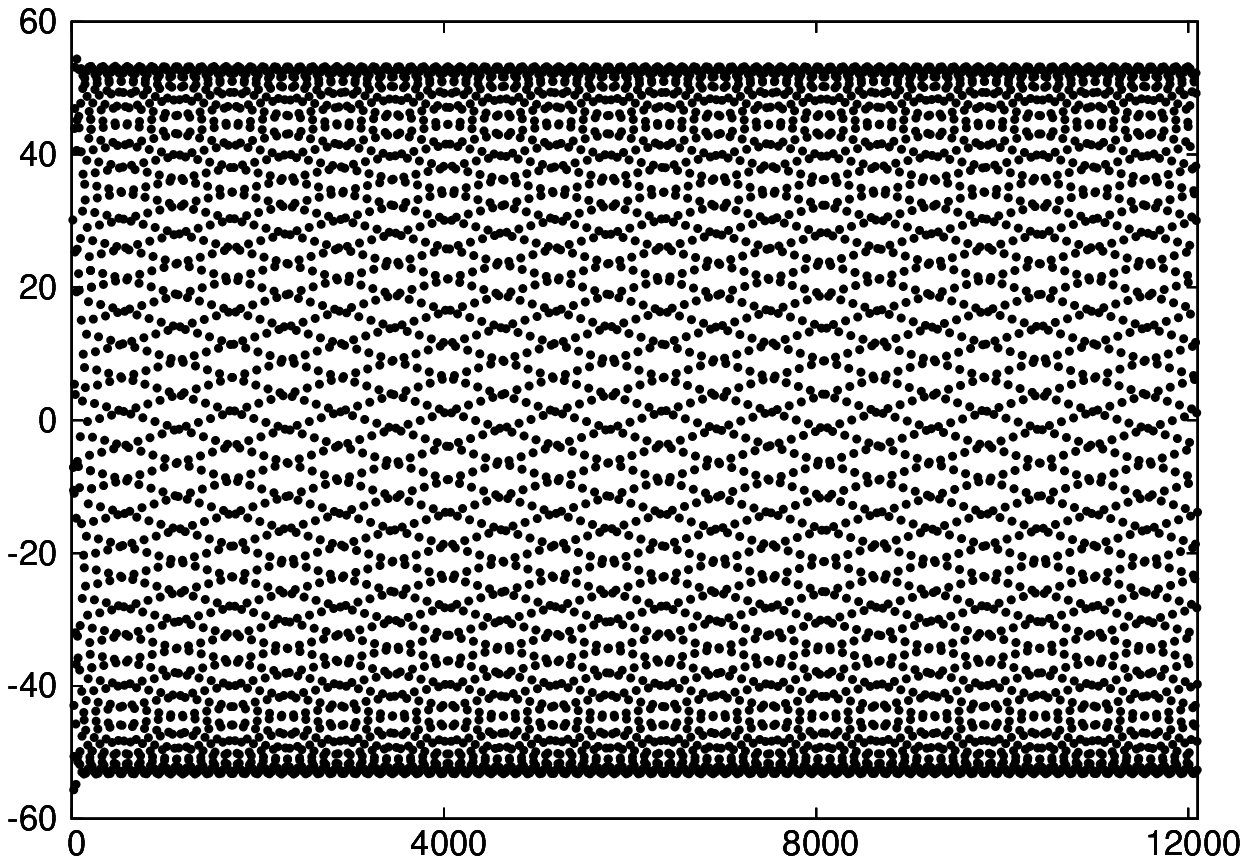}
\else
\fbox{\rule{0in}{1.6in}\hspace{5cm}}
\fi
}
\hfill
\subfigure[DFT of $\hat{D}_q$ for $n=6$, $\epsilon=1$, and even $q$.]{
\label{fig:DFTHatDq_n6_eps1_qeven}
\iffigures
\includegraphics[height=1.6in]{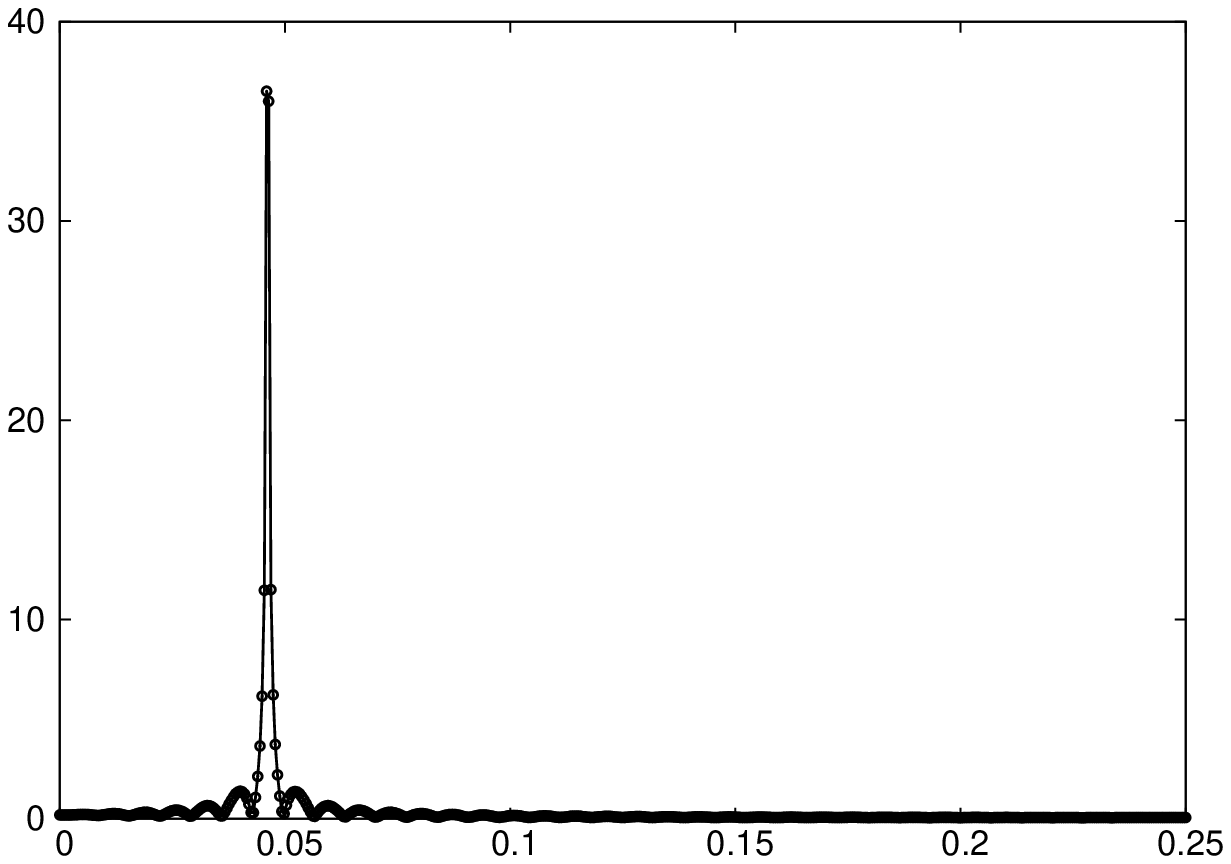}
\else
\fbox{\rule{0in}{1.6in}\hspace{5cm}}
\fi
}
\hfill
\subfigure[$\hat{D}_q - A(q)$ versus $q$ for $n=6$, $\epsilon=1$, and even $q$.]{
\label{fig:HatDqminusAq_n6_eps1_qeven}
\iffigures
\includegraphics[height=1.6in]{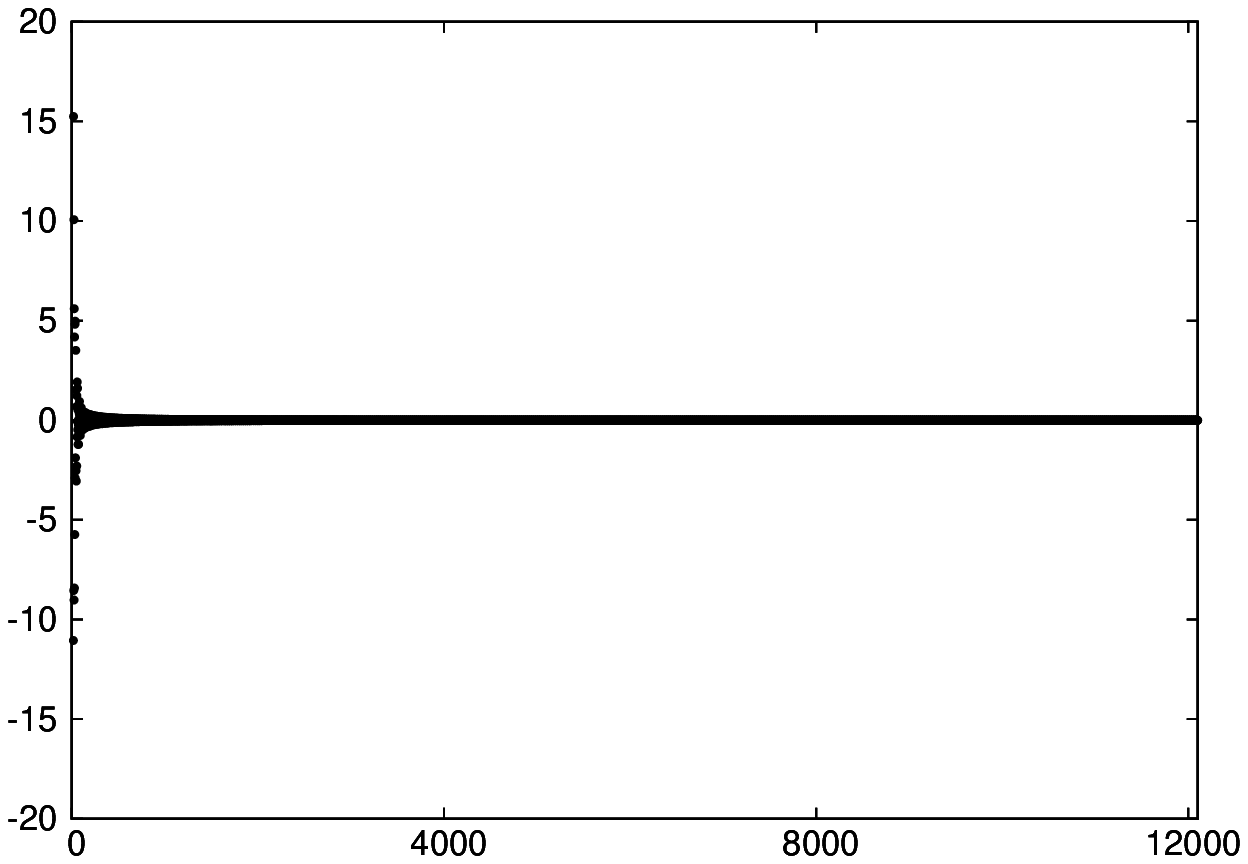}
\else
\fbox{\rule{0in}{1.6in}\hspace{5cm}}
\fi
}
\subfigure[$\hat{D}_q$ versus $q$ for $n=6$, $\epsilon=1$, and odd $q$.]{
\label{fig:HatDq_n6_eps1_qodd}
\iffigures
\includegraphics[height=1.6in]{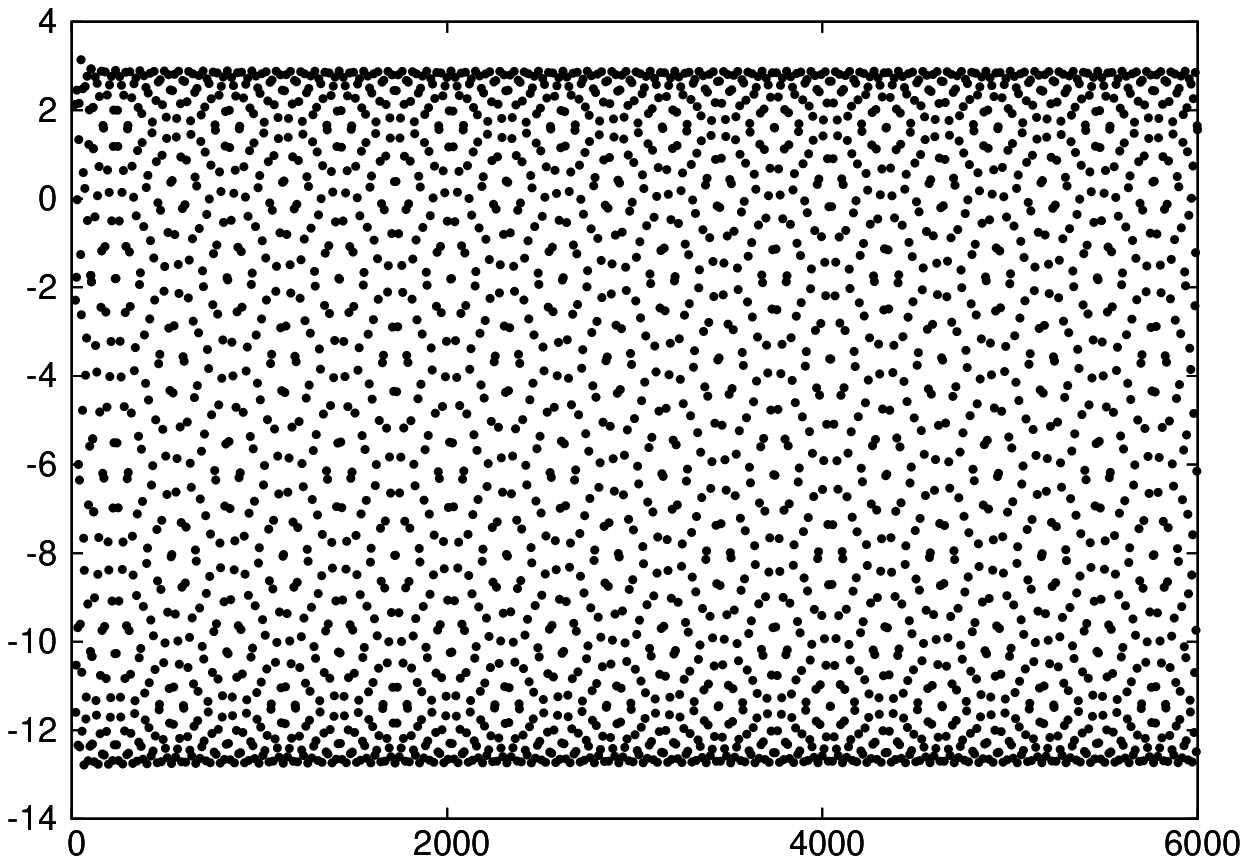}
\else
\fbox{\rule{0in}{1.6in}\hspace{5cm}}
\fi
}
\hfill
\subfigure[DFT of $\hat{D}_q$ for $n=6$, $\epsilon=1$, and odd $q$.]{
\label{fig:DFTHatDq_n6_eps1_qodd}
\iffigures
\includegraphics[height=1.6in]{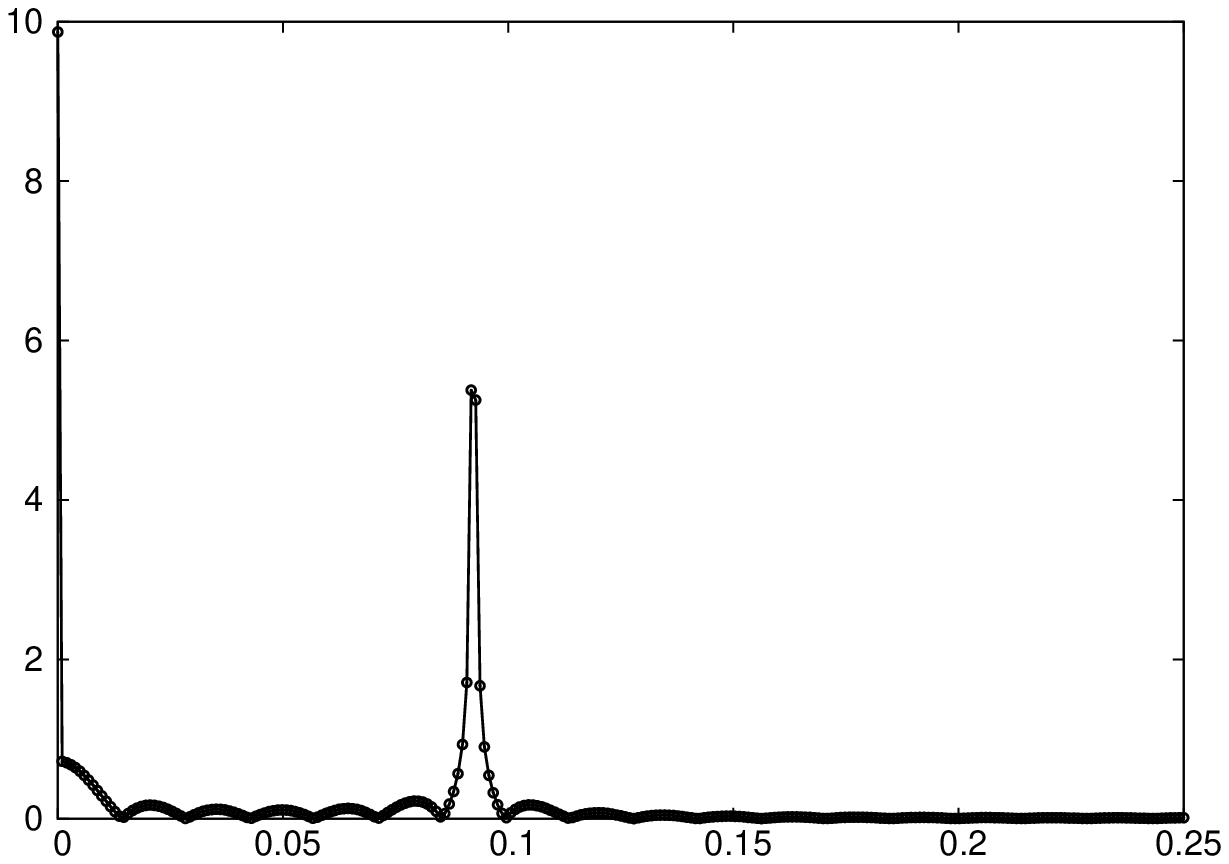}
\else
\fbox{\rule{0in}{1.6in}\hspace{5cm}}
\fi
}
\hfill
\subfigure[$\hat{D}_q - B(q)$ versus $q$ for $n=6$, $\epsilon=1$, and odd $q$.]{
\label{fig:HatDqminusBq_n6_eps1_qodd}
\iffigures
\includegraphics[height=1.6in]{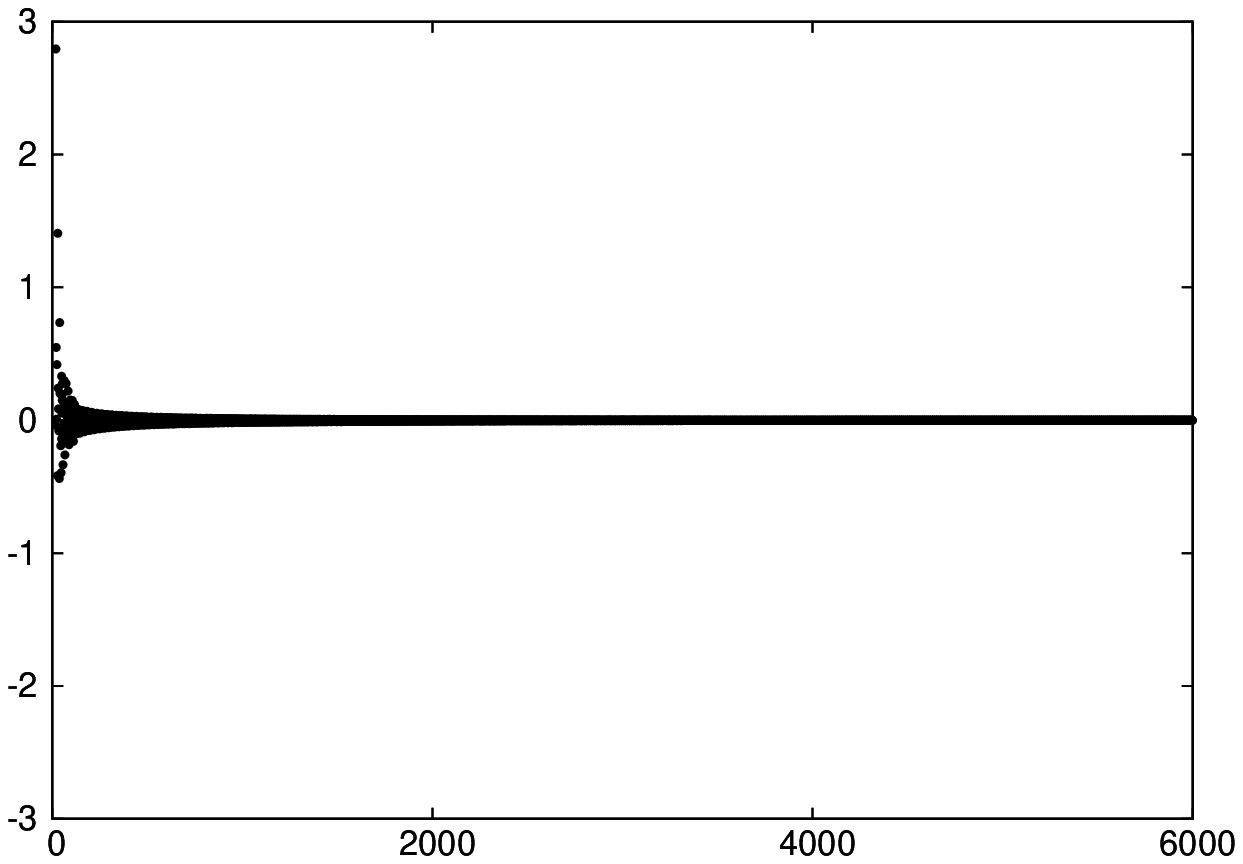}
\else
\fbox{\rule{0in}{1.6in}\hspace{5cm}}
\fi
}
\caption{Examples with a periodic asymptotic behavior of the
normalized differences $\hat{D}_q$.
We recall that $A(q) = a\cos(2\pi\beta q)$ and $B(q) = \bar{b} + b \cos(4\pi\beta q)$.
Besides, $a \approx 29.4849$ and $\beta \approx 1/8$ in
Figure~\ref{fig:HatDqminusAq_n3_eps1o3};
$a \approx 53.2369$ and $\beta \approx 0.04614$ in
Figure~\ref{fig:HatDqminusAq_n6_eps1_qeven};
and $\bar{b} \approx -4.9257$, $b \approx 7.80853$, and
$\beta \approx 0.04614$ in Figure~\ref{fig:HatDqminusBq_n6_eps1_qodd}.}
\label{fig:HatDq_Periodic}
\end{figure*}

First, we compute the exponent~$r = \rho/2$ by using the Borel method,
since it is computationally cheaper than the direct one.
See Remark~\ref{remark:CostComputingR}.
Besides, it is not clear how to adapt the direct method
when the functions $A(q)$ and $B(q)$ oscillate.
We follow the same steps as in the case of perturbed ellipses.
However, the Neville extrapolation is more unstable for perturbed circles.
In order to overcome this instability, now
we take sequences $(q_i)$ of 1000 periods such that $|D_{q_0}|\le 10^{-5000}$.

Once we find $r$, we compute the normalized differences~$\hat{D}_q$
already introduced in~\eqref{eq:NormalizedDqPerturbedEllipses}.
We have checked that there exist two non-zero quasiperiodic functions
$A(q)$ and $B(q)$ such that
\[
\hat{D}_q \asymp
\begin{cases}
B(q), & \mbox{for even $n$ and odd $q$}, \\
A(q), & \mbox{otherwise},
\end{cases}
\]
as $q \to +\infty$.

Some paradigmatic examples of the asymptotic behavior of the
normalized differences $\hat{D}_q$ are displayed in
Figures~\ref{fig:HatDq_Constant} and~\ref{fig:HatDq_Periodic}.
All these examples are generic in the sense
that a small change of the perturbative parameter $\epsilon$ does
not produce any qualitative change in the pictures.

For instance, we see three examples where $\hat{D}_q$ tends to
some constant as $q \to +\infty$ in Figure~\ref{fig:HatDq_Constant}.
The constant is $A$ in the second and third subfigures,
and $B$ in the first one.

We display a first example of periodic asymptotic behavior
in Figure~\ref{fig:HatDq_n3_eps1o3} for the cubic perturbation
and $\epsilon = 1/3$.
This value $\epsilon = 1/3$ is relatively close to the value
$\bar{\epsilon}_3(1) \approx 0.3849$ where the algebraic curve
$x^2 + y^2 + \epsilon y^3 =1$ becomes singular.
Next, we compute the discrete Fourier transform (DFT) of the
last terms of the sequence $\hat{D}_q$.
To be precise, the terms in the range $10000 < q \le 12000$
for $n=6$ and even $q$, and in the range  $5000 < q \le 6000$ otherwise.
We discard the first terms because $\hat{D}_q \asymp A(q)$
and $\hat{D}_q \asymp B(q)$,
so the last normalized differences are closer
to the periodic functions we want to determine.

The DFT of the normalized differences $\hat{D}_q$ suggests that the
periodic function $A(q)$ has a dominant harmonic with amplitude $a
\approx 29.4849$ and frequency $\beta \approx 0.375 = 1/8$ when
$\epsilon = 1/3$ and $n=3$. See Figure~\ref{fig:DFTHatDq_n3_eps1o3}.
This explains why we see eight waves in
Figure~\ref{fig:HatDq_n3_eps1o3}, each one with frequency
$|\beta-1/8|$. This situation is a source of problems for the
following reason. Let us assume that, due to time or computational
restrictions, we only compute the normalized differences for periods
of the form $q_i = q_0 + 8i$. In that case, we would only see one
wave and we would get a wrong frequency. The moral of this story is
that we have to compute the normalized differences for \emph{all}
periods. Then we compare the normalized differences $\hat{D}_q$ with
the cosine wave $A(q) = a \cos(2\pi\beta q)$ as $q \to +\infty$. The
amplitude $a$ and the frequency $\beta$ are determined by mixing
several tools: the DFT, some direct algebraic computations,
etcetera. The plot in Figure~\ref{fig:HatDqminusAq_n3_eps1o3} shows
that
\[
\lim_{q \to +\infty} \left( \hat{D}_q - A(q) \right) = 0.
\]

We study the case $n = 6$ and $\epsilon = 1$ in
Figures~\ref{fig:HatDq_n6_eps1_qeven}--\ref{fig:HatDqminusBq_n6_eps1_qodd}.
The most interesting phenomena shown up by those pictures are
the following ones.
First, we confirm that the frequency of the periodic function $B(q)$
is twice the frequency of the cosine wave $A(q)$.
See Figures~\ref{fig:DFTHatDq_n6_eps1_qeven}
and~\ref{fig:DFTHatDq_n6_eps1_qodd}.
Second, the average of $B(q)$ is not zero.
This is a surprise, because both the periodic functions obtained in similar
splitting problems and the periodic function $A(q)$ obtained in this
billiard problem have generically zero average.
Third, $B(q) = \bar{b} + b \cos(4\pi \beta q)$, but $\bar{b} \neq b/2$,
which sets another difference with the known asymptotic behaviors
for splitting problems.

Next, we present some results about the transition between the
two generic ---``constant'' and ``periodic''--- asymptotic behaviors
of the normalized differences $\hat{D}_q$.
That is, we intend to visualize what happens at some
$\epsilon_* \in \partial C_n \cap \partial P_n \subset R_n$.

We focus our attention on the sixtic perturbation: $n = 6$.
Then the normalized differences have ``constant'' and ``periodic''
asymptotic behaviors for $\epsilon = 1/10$ and $\epsilon = 1$, respectively.
We study the quantities $\hat{D}_q$ in a fine grid of perturbative parameters
in the interval $[1/10,1]$.
Both functions $A(q)$ and $B(q)$ change at the same transition value $\epsilon_*$.
Indeed,
\[
[1/10,23/200] \subset C_6,\qquad [3/25,1] \subset P_6,
\]
so the transition takes place at some $\epsilon_* \in (23/200,3/25)$.

Unfortunately, a more precise computation of $\epsilon_*$ is beyond our
current abilities,
because we do not have a limit problem whose complex singularities
allow us to determine analytically the transition values.
An example of such analytical computations for splitting problems
can be found in~\cite{GelfreichLazutkin2001,GelfreichSimo2008}.

Therefore, we only display the normalized differences $\hat{D}_q$ for
$\epsilon = 23/200$ and $\epsilon = 3/25$ in
Figures~\ref{fig:Constant2Oscillating_A} and~\ref{fig:Constant2Oscillating_B}
to see the transition of the functions $A(q)$ and $B(q)$, respectively.

\begin{figure*}
\centering
\subfigure[$\epsilon = 23/200$.]{
\iffigures
\includegraphics[height=2.4in]{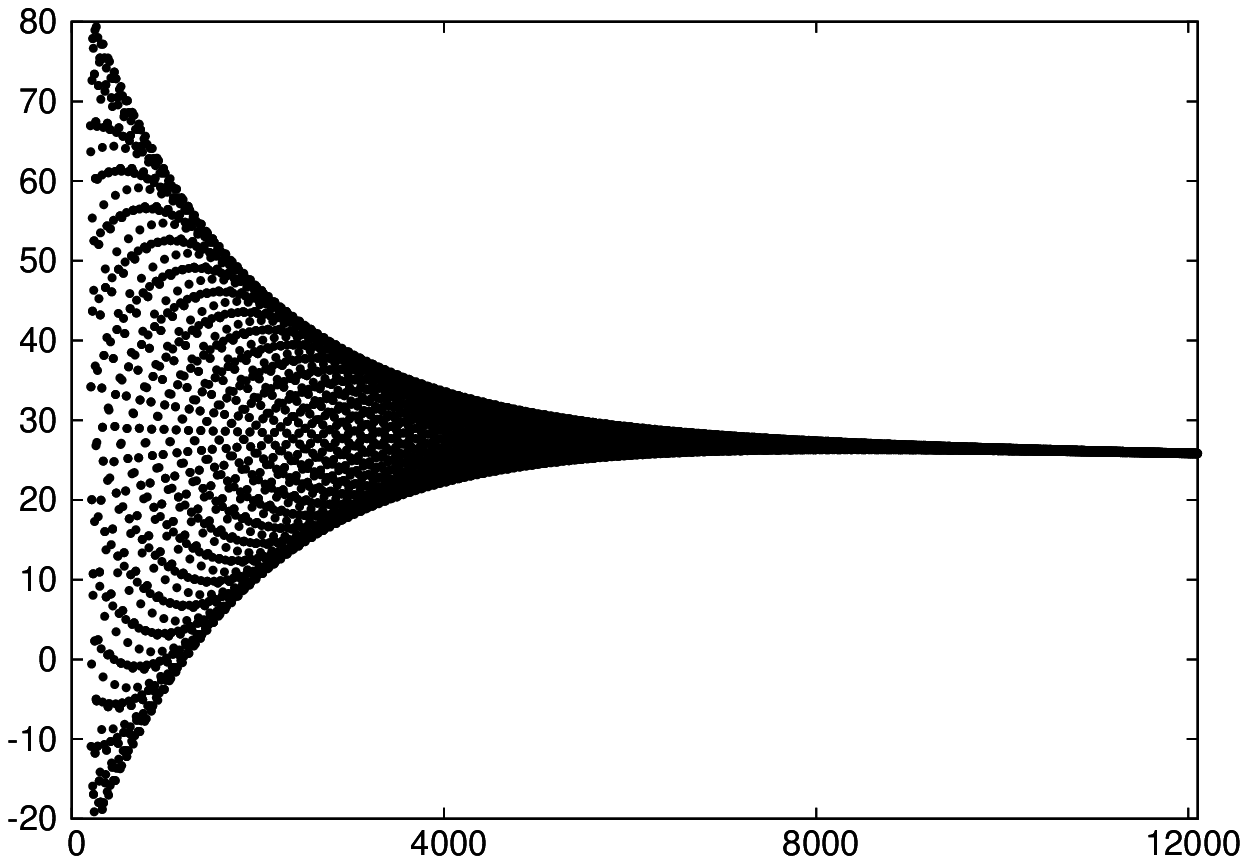}
\else
\fbox{\rule{0in}{2.4in}\hspace{8cm}}
\fi
}
\subfigure[$\epsilon = 3/25$.]{
\iffigures
\includegraphics[height=2.4in]{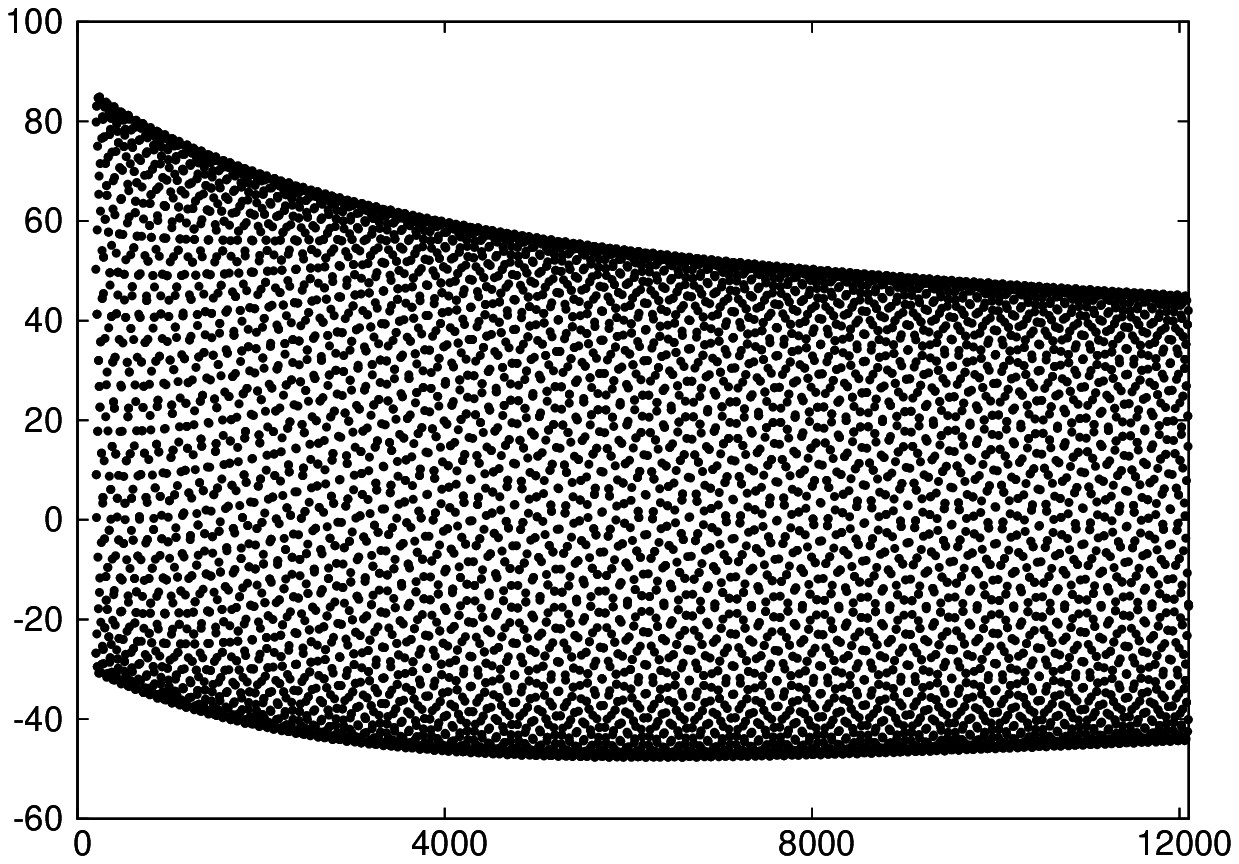}
\else
\fbox{\rule{0in}{2.4in}\hspace{8cm}}
\fi
}
\caption{Transition of the function $A(q)$ from constant to periodic.
We plot the normalized differences $\hat{D}_q$ versus $q$ for
$n=6$ and even periods.}
\label{fig:Constant2Oscillating_A}
\end{figure*}

\begin{figure*}
\centering
\subfigure[$\epsilon = 23/200$.]{
\iffigures
\includegraphics[height=2.4in]{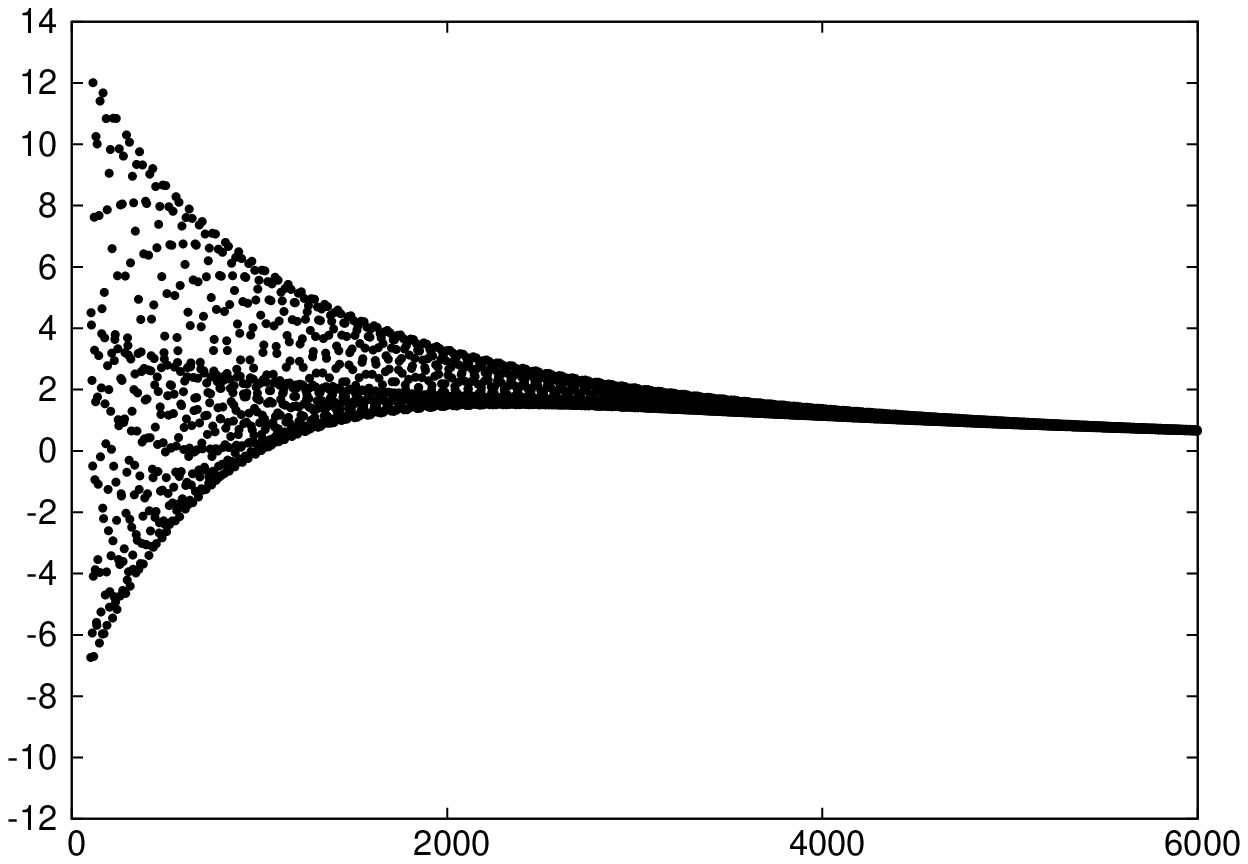}
\else
\fbox{\rule{0in}{2.4in}\hspace{8cm}}
\fi
}
\subfigure[$\epsilon = 3/25$.]{
\iffigures
\includegraphics[height=2.4in]{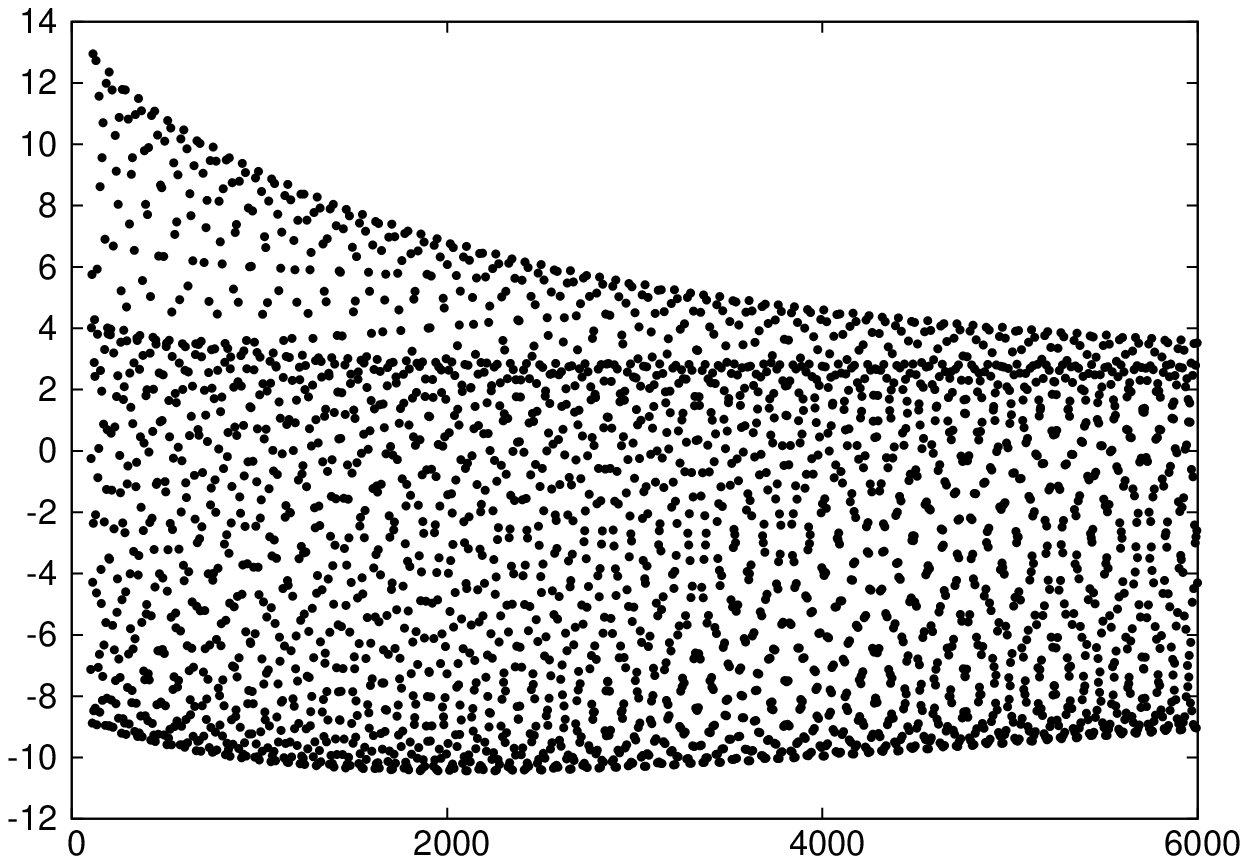}
\else
\fbox{\rule{0in}{2.4in}\hspace{8cm}}
\fi
}
\caption{Transition of the function $B(q)$ from constant to periodic.
We plot the normalized differences $\hat{D}_q$ versus $q$ for
$n=6$ and odd periods.}
\label{fig:Constant2Oscillating_B}
\end{figure*}

Let us present some numerical results about the logarithmic
growth~(\ref{eq:rLogarithm}) of the exponent $r$.
We have computed the exponent $r = \rho/2$ by using the Borel method
for $3 \le n \le 8$ in a sequence of perturbative parameters of the form
$\epsilon_j = 2^{-j}/10$ with $j \ge 0$.
We have plotted the results in Figure~\ref{fig:LogarithmicGrowth_r}.
On the one hand,
the curves in Figure~\ref{fig:r_vs_logeps} look like straight lines
with slopes $1/n$, as expected.
On the other hand,
the curves in Figure~\ref{fig:rminuslogepsn_vs_logeps} tend
to some constant values $\chi_n > 0$.
This ends the numerical study of such phenomenon.

\begin{figure*}
\centering
\subfigure[$r$ versus $|\log \epsilon|$.]{
\label{fig:r_vs_logeps}
\iffigures
\includegraphics[height=2.4in]{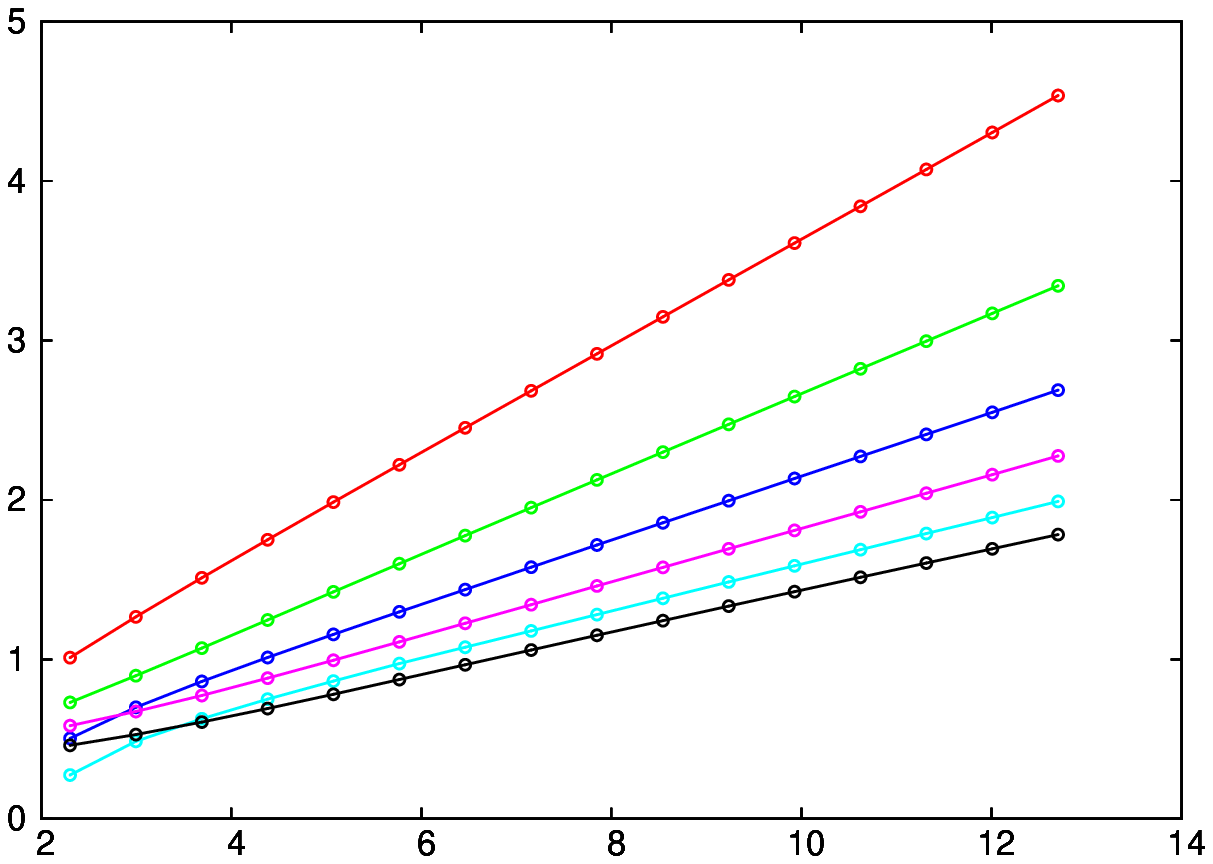}
\else
\fbox{\rule{0in}{2.4in}\hspace{8cm}}
\fi
}
\subfigure[$r-|\log\epsilon|/n$ versus $|\log \epsilon|$.]{
\label{fig:rminuslogepsn_vs_logeps}
\iffigures
\includegraphics[height=2.4in]{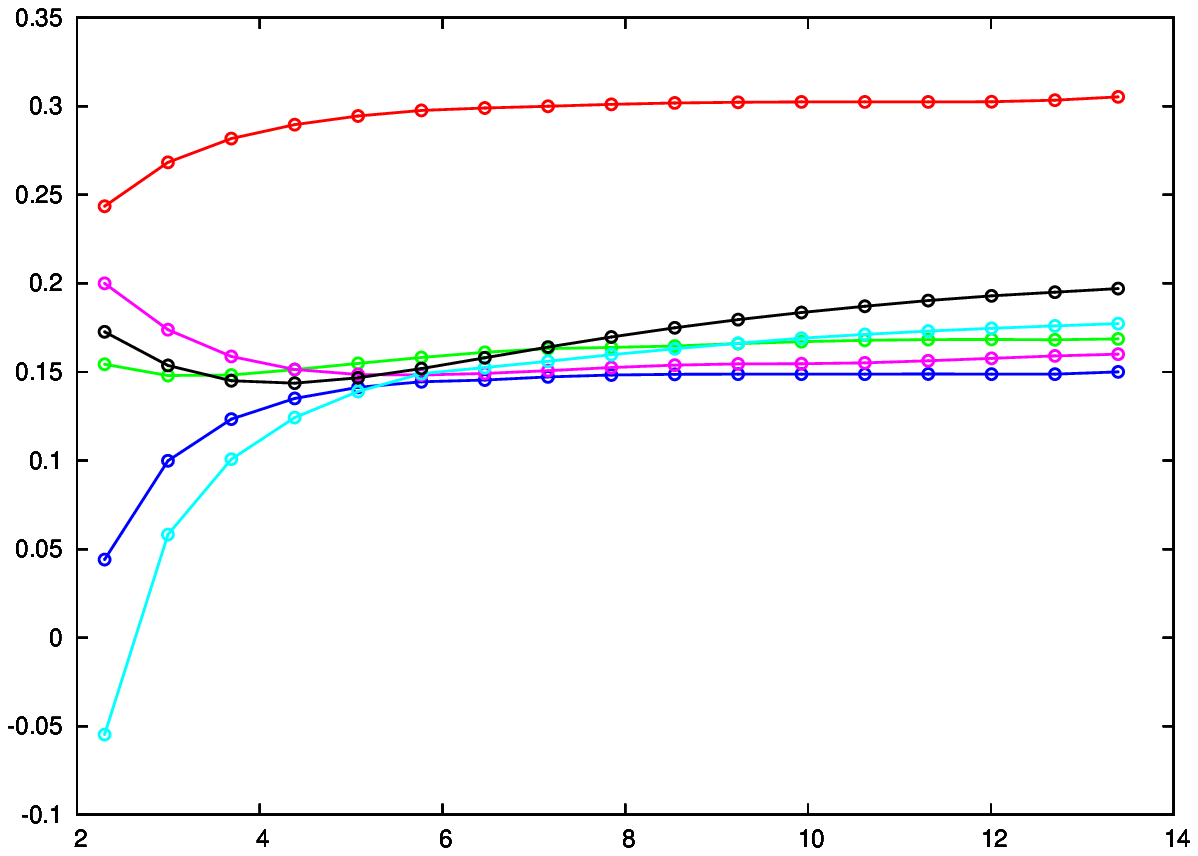}
\else
\fbox{\rule{0in}{2.4in}\hspace{8cm}}
\fi
}
\caption{
Logarithmic growth of the exponent $r$ as $\epsilon \to 0^+$.
Red: $n = 3$. Green: $n=4$. Blue: $n=5$. Magenta: $n=6$. Cyan: $n=7$.
Black: $n=8$.}
\label{fig:LogarithmicGrowth_r}
\end{figure*}

Finally, we see that our candidate for limit problem captures
this logarithmic behavior, although it may not give the
exact value of the exponent $r$.

\begin{prop}\label{prop:deltaPerturbedCircle}
Let $n \ge 3$ and $\epsilon \in I_n$.
Let $\kappa(s)$ be the curvature of the strictly convex
curve $Q = \{ (x,y) \in \Rset^2 : x^2 + y^2 +\epsilon y^n = 1\}$
in some arc-length parameter $s$.
Let $\xi \in \Rset/\Zset$ be the angular variable defined
by~(\ref{eq:xiDefinition}).
Let $\delta$ be the distance of the set of singularities and zeros of
the curvature $\kappa(\xi)$ to the real axis.
There exists $\eta_n \in\Rset$ such that
\begin{equation}\label{eq:deltaLogarithm}
2\pi\delta =
\frac{|\log\epsilon|}{n} + \eta_n + \Order(\epsilon^{2/n}\log \epsilon),
\end{equation}
as $\epsilon \to 0^+$.
\end{prop}

The proof of this proposition is placed in~\ref{app:proofCircularBehaviour}.

\begin{table}
\begin{center}
\begin{tabular}{c c c}
\hline
\hline
$n$ &  $\chi_n$ & $\eta_n$ \\
\hline
$3$ & $0.30\ldots$ & $1.1358418243\ldots$ \\ %1.13584182430739674350
$4$ & $0.17\ldots$ & $1.0703321545\ldots$ \\ %1.0703321545057934111811962751
$5$ & $0.15\ldots$ & $0.1488295936\ldots$ \\ %0.14882959361123870308
$6$ & $0.15\ldots$ & $1.0385641059\ldots$ \\ %1.038564105862838
$7$ & $0.18\ldots$ & $0.1823551667\ldots$ \\ %0.1823551666713
$8$ & $0.19\ldots$ & $1.0332248276\ldots$ \\ %1.03322482762
\hline
\hline
\end{tabular}
\end{center}
\caption{The constants $\chi_n$ and $\eta_n$, with $\chi_n \le \eta_n$,
that appear in formulas~(\ref{eq:rLogarithm}) and~(\ref{eq:deltaLogarithm}),
respectively.}
\label{table:LogarithmicGrowth_r}
\end{table}

The constant $\chi_n$ in~(\ref{eq:rLogarithm}) is always smaller than
(or equal to) the constant $\eta_n$ in~(\ref{eq:deltaLogarithm}).
We compare both constants in Table~\ref{table:LogarithmicGrowth_r}.

Constants $\chi_n$ are computed from the numerical data used in
Figure~\ref{fig:rminuslogepsn_vs_logeps}. Constants $\eta_n$ are
computed by using the techniques explained in
Remark~\ref{remark:Computation_delta}. On the one hand, we obtain
just two significant digits for the constants $\chi_n$. On the other
hand, we can compute $\eta_n$ with a much higher precision; here we
have just written their first ten decimal digits. We see that
$\chi_n < \eta_n$ for $n \in \{3,4,6,8\}$. We do not discard the
equalities $\chi_5 = \eta_5$ and $\chi_7 = \eta_7$. In order to
elucidate them, we compare the exponent $r$ with the quantity
$2\pi\delta$, as we have done before for perturbed ellipses at the
end of Section~\ref{sec:PerturbEllipses}. The results are displayed
in Figure~\ref{fig:Comparison_r_2pidelta_circles}, where we see that
our candidate for limit problem gives the exact exponent $r$ in two
cases.

To be precise, our numerical results suggest that:
\begin{itemize}
\item
If $n \in \{3,4,6,8\}$, then $r < 2\pi\delta$ for all
$\epsilon \in (0,1/10)$; and
\item
If $n \in \{5,7\}$, then $r = 2\pi\delta$ for all $\epsilon \in (0,1/10)$.
\end{itemize}

\begin{figure}
\iffigures
\centering
\includegraphics[height=2.4in]{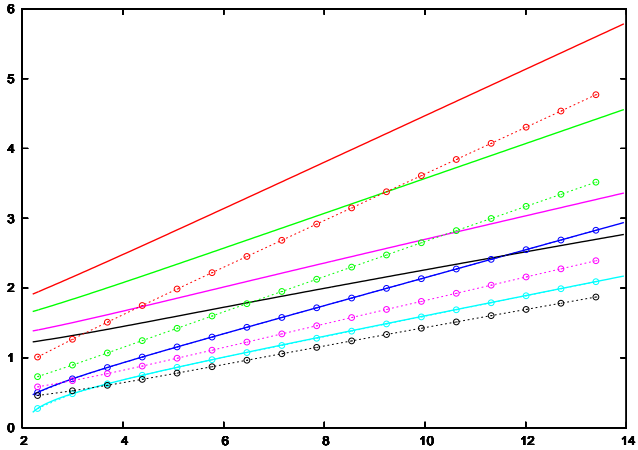}
\else
\fbox{\rule{0in}{2.4in}\hspace{8cm}}
\fi
\caption{The exponent $r$ (dashed lines with points) and
the quantity $2\pi\delta$ (continuous lines) versus $|\log \epsilon|$.
Red: $n = 3$. Green: $n=4$. Blue: $n=5$. Magenta: $n=6$. Cyan: $n=7$.
Black: $n=8$.}
\label{fig:Comparison_r_2pidelta_circles}
\end{figure}

\section*{Acknowledgements}
We thank T.~M.~Seara and C.~Sim\'o for very useful remarks and
comments. We also appreciate the assistance of A.~Granados and
P.~Rold\'{a}n in the use of the UPC Applied Math cluster for our
experiments.

\appendix

\section{Proof of Proposition~\ref{prop:melnikov_boundary}}
\label{app:proofMelnikov}

We will use many properties of elliptic functions listed
in the books~\cite{AbramowitzStegun1964,WhittakerWatson1996},
a couple of technical results about elliptic billiards contained
in~\cite{ChangFriedberg1988,CasasRamirez2011},
and the subharmonic Melnikov potential of billiards inside
perturbed ellipses introduced in~\cite{PintodeCarvalhoRamirezRos2013}.

We consider the unperturbed ellipse
\begin{equation}\label{eq:Ellipse}
E =
\left\{ (x,y) \in \Rset^2 : x^2/a^2 + y^2/b^2 = 1 \right\},\qquad
0 < b < a.
\end{equation}
It is known that the convex caustics of the billiard inside $E$ are the
confocal ellipses
\[
C_\lambda =
\left\{
(x,y) \in \Rset^2 : \frac{x^2}{a^2-\lambda^2} + \frac{y^2}{b^2-\lambda^2} = 1
\right\},\qquad 0 < \lambda < b.
\]
There is a unique $(p,q)$-resonant elliptic caustic $C_\lambda$ for
any relatively prime integers $p$ and $q$ such that $1 \le p < q/2$.
The caustic parameter of the $(p,q)$-resonant elliptic caustic
is implicitly determined by means of equation~(\ref{eq:ResonantCondition}).

The \emph{complete elliptic integral of the first kind} is
\[
K = K(m) = \int_{0}^{\pi/2}(1-m \sin^2 \phi)^{-1/2} \rmd\phi.
\]
Its argument $m \in (0,1)$ is called the \emph{parameter}.
We also write $K' = K'(m) = K(1-m)$.
The \emph{amplitude} function $\varphi = \am t$ is defined through
the inversion of the integral
\[
t = \int_{0}^{\varphi}(1-m \sin^2 \phi)^{-1/2}\rmd \phi.
\]
The \emph{elliptic sine} and \emph{elliptic cosine}
associated to the parameter $m\in (0,1)$
are defined by the trigonometric relations
\[
\sn t = \sn(t,m) = \sin \varphi,\qquad
\cn t = \cn(t,m) = \cos \varphi.
\]
If the angular variable $\varphi$ changes by $2\pi$,
the angular variable $t$ changes by $4K$.
Thus, any $2\pi$-periodic function in $\varphi$,
becomes $4K$-periodic in $t$.
By abuse of notation,
we will also denote the $4K$-periodic functions with the
name of the corresponding $2\pi$-periodic ones.
For example, if $q(\varphi) = (a\cos \varphi, b\sin\varphi)$
is the natural $2\pi$-periodic parameterization of the ellipse $E$,
then $q(t) = (a\cn t, b\sn t)$ is the corresponding
$4K$-periodic parameterization.
The billiard dynamics associated to an elliptic caustic $C_\lambda$
becomes a rigid rotation $t \mapsto t + \delta$ in the variable $t$.
It suffices to find the shift $\delta$ and the parameter $m$
associated to each elliptic caustic $C_\lambda$.
The parameter $m$ is given in~\cite[Eq.~(3.28)]{ChangFriedberg1988}
and the constant shift $\delta$ is given in~\cite[p. 1543]{ChangFriedberg1988}.
We list the formulas in the following lemma.

\begin{lem}
\label{lem:ChangFriedberg}
Once fixed an elliptic caustic $C_\lambda$ with $\lambda \in (0,b)$,
the parameter $m \in (0,1)$ and the shift $\delta \in (0,2K)$ are
\begin{equation}\label{eq:ParameterShift}
m = \frac{a^2-b^2}{a^2 - \lambda^2},\qquad
\delta/2 = \int_0^{\vartheta/2} (1 - m\sin^2 \phi)^{-1/2}\rmd \phi,
\end{equation}
where $\vartheta \in (0,\pi)$ is the angle such that
$\sin (\vartheta/2) = \lambda/b$.
The segment joining the points $q(t)$ and $q(t+\delta)$
is tangent to $C_\lambda$ for all $t \in \Rset$.
\end{lem}

From now on, $m$ and $\delta$ will denote the parameter and the constant shift
defined in~(\ref{eq:ParameterShift}).
Observe that the elliptic caustic $C_\lambda$ is $(p,q)$-resonant
if and only if
\begin{equation}\label{eq:ResonantCondition}
q \delta = 4 K p.
\end{equation}
This identity has the following geometric interpretation.
When a billiard trajectory makes one turn around $C_\lambda$,
the old angular variable $\varphi$ changes by $2\pi$,
so the new angular variable $t$ changes by $4K$.
Besides, we have seen that the variable $t$ changes
by $\delta$ when a billiard trajectory bounces once.
Hence, a billiard trajectory inscribed in $E$ and circumscribed around
$C_\lambda$ makes exactly $p$ turns after $q$ bounces
if and only if~(\ref{eq:ResonantCondition}) holds.

We consider the elliptic coordinates $(\mu,\varphi)$ associated to
the semi-lengths $0 < b < a$.
That is, $(\mu,\varphi)$ are defined by relations
\begin{equation}\label{eq:EllipticCoordinates}
x = \sigma \cosh \mu \cos \varphi,\qquad
y = \sigma \sinh \mu \sin \varphi,
\end{equation}
where $\sigma = \sqrt{a^2-b^2}$ is the semi-focal distance of $E$.
The ellipse $E$ in these coordinates reads as $\mu \equiv \mu_0$,
where $\cosh \mu_0 = a/\sigma$ and $\sinh \mu_0 = b/\sigma$.
Hence, any smooth perturbation of $E$
can be written in elliptic coordinates as
\begin{equation}\label{eq:EllipticPerturbation}
\mu = \mu_0 + \epsilon \mu_1(\varphi) + \Order(\epsilon^2),
\end{equation}
for some $2\pi$-periodic function $\mu_1: \Rset \to \Rset$.

\begin{lem}\label{lem:PotentialEllipses}
Let $p$ and $q$ be two relatively prime integers such that
$1 \le p < q/2$.
Let $C_\lambda$ be the $(p,q)$-resonant elliptic caustic
of the ellipse~(\ref{eq:Ellipse}).
Let
\[
\Delta^{(p,q)} = \epsilon \Delta_1^{(p,q)} + \Order(\epsilon^2)
\]
be the maximal difference among lengths of $(p,q)$-periodic trajectories
inside the perturbed ellipse~(\ref{eq:EllipticPerturbation}).
Let $\mu_1(t)$ be the $4K$-periodic function
associated to the $2\pi$-periodic one $\mu_1(\varphi)$.
Let
\begin{equation}\label{eq:SubharmonicMelnikovPotential}
L^{(p,q)}_1(t) = 2 \lambda \sum_{j=0}^{q-1} \mu_1(t+j \delta).
\end{equation}
be the subharmonic Melnikov potential of the
caustic $C_\lambda$ for the perturbed ellipse~(\ref{eq:EllipticPerturbation}).
If $L^{(p,q)}_1(t)$ does not have degenerate critical points and
$\epsilon > 0$ is small enough,
then there is a one-to-one correspondence between the critical points
of $L^{(p,q)}_1(t)$ and the $(p,q)$-periodic billiard trajectories
inside~(\ref{eq:EllipticPerturbation}).
Besides,
\[
\Delta_1^{(p,q)} = \max L_1^{(p,q)}- \min L_1^{(p,q)}.
\]
\end{lem}

\begin{proof}
It follows directly from results contained
in~\cite{PintodeCarvalhoRamirezRos2013}.
\end{proof}

We will determine the asymptotic behavior of
$\Delta_1^{(p,q)}$.
First, we study the asymptotic behavior of the $(p,q)$-resonant caustic
$C_\lambda$ as $p/q \to 0^+$.

\begin{lem}\label{lem:Xi}
If $C_\lambda$ is the $(p,q)$-resonant elliptic caustic of the
ellipse~(\ref{eq:Ellipse}),
then $\lambda \asymp \Xi p/q$ as $p/q \to 0^+$, where
\begin{equation}\label{eq:Xi}
\Xi = \Xi(a,b) := ab \int^{a^2}_{b^2}\left(s(s-b^2)(a^2-s)\right)^{-1/2}\rmd s.
\end{equation}
\end{lem}

\begin{proof}
It follows directly from~\cite[Proposition~10]{CasasRamirez2011}.
\end{proof}

\begin{lem}\label{lem:Cosine2}
The following properties hold for $\mu_1(\varphi) =  \cos^2 \varphi$.
\begin{enumerate}
\item
The Melnikov potential $L_1^{(p,q)}(t)$ has just two real critical points
(modulo its periodicity), none of them degenerate.
\item
There exist an exponent $\zeta=\zeta(\omega_*,a,b) > 0$ and
a quantity $\Omega_4 = \Omega_4(\omega_*,a,b,p,q) > 0$ such that
\[
\Delta_1^{(p,q)} \asymp
\begin{cases}
2 \Omega_4 \rme^{-2 \zeta q}, & \mbox{for odd  $q$,}\\
  \Omega_4 \rme^{-\zeta q},  & \mbox{for even $q$,}
\end{cases}
\]
as $p/q \to \omega_* \in \{0\} \cup \big((0,1)\setminus \Qset\big)$.
\item
There exist $\Gamma_4 = \Gamma_4(\omega_*,a,b) > 0$ and
$\Theta_4 = \Theta_4(a,b) > 0$ such that
\[
\Omega_4(\omega_*,a,b,p,q)=
\begin{cases}
\Gamma_4 q^2, & \mbox{if $\omega_* \in(0,1) \setminus \Qset$,}\\
\Theta_4 p q, & \mbox{if $\omega_* = 0$.} \\
\end{cases}
\]
\item
$\zeta(0,a,b) = \pi K'(1-(b/a)^2)/2K(1-(b/a)^2)$.
\end{enumerate}
\end{lem}

\begin{proof}
By definition, if $\mu_1(\varphi) =  \cos^2 \varphi$, then
\[
L^{(p,q)}_1(t) = 2  \lambda \sum_{j=0}^{q-1} \cn^2(t+j \delta).
\]
The square of the elliptic cosine is an elliptic function
of order two, periods $2K$ and $2K'\rmi$, and double poles in the set
\[
P = K'\rmi + 2K\Zset + 2K'\rmi\Zset.
\]
Besides, the principal part of any pole $\tau \in P$ is
$-m^{-1}(t-\tau)^{-2}$.
In particular, $L_1^{(p,q)}(t)$ is also an elliptic function of order two,
and so, it can be determined (modulo an additive constant) by its periods,
poles, and principal parts.

We study the cases odd $q$ and even $q$ separately.

If $q$ is odd, then $L_1^{(p,q)}(t)$ has periods $2K/q$ and $2K'\rmi$
and double poles with principal parts
$-2  \lambda m^{-1}(t-\tau)^{-2}$ in the set
\[
P_q = K'\rmi + \frac{2K}{q}\Zset + 2K'\rmi.
\]
It is known that $K'(m)/K(m)$ is a decreasing function such that
\[
\lim_{m\to 0^+} \frac{K'(m)}{K(m)} = +\infty,\qquad
\lim_{m\to 1^-} \frac{K'(m)}{K(m)} = 0.
\]
Therefore, there exists a unique $m_q \in (0,1)$ such that
\[
\frac{K'_q}{K_q} := \frac{K'(m_q)}{K(m_q)} = q \frac{K'(m)}{K(m)} =:
q \frac{K'}{K}.
\]
Henceforth, we write that $K =K(m)$, $K' =K'(m)$, $K_q =K(m_q)$,
and $K'_q =K'(m_q)$ for short.
Thus,
\[
L_1^{(p,q)}(t) =
\mbox{const.} + 2  \lambda (q K_q/K)^2 (m_q/m) \cn^2(qK_q t/K, m_q),
\]
which has just two real critical points (modulo its periodicity),
none of them degenerate.
Besides
\[
\Delta_1^{(p,q)} =
\max L_1^{(p,q)}- \min L_1^{(p,q)} =
2  \lambda (q K_q/K)^2 (m_q/m).
\]
If $p/q \to \omega_*\in(0,1)\setminus \Qset$,
then $q \to +\infty$ and $\lambda \to \lambda_* \in (0,b)$,
where $C_{\lambda_*}$ is the elliptic caustic with rotation number
$\omega_*$, so
\begin{align*}
m \to m_* :=  \frac{a^2-b^2}{a^2 - \lambda_*^2} \in (0,1), & \qquad
m_q \to 0^+, \\
\frac{K'}{K} \to \frac{K'_*}{K_*} := \frac{K'(m_*)}{K(m_*)} \in (0,+\infty), & \qquad
K_q \to \frac{\pi}{2}.
\end{align*}
Using~\cite[17.3.14 \& 17.3.16]{AbramowitzStegun1964},
we get the asymptotic formula $m_q \asymp 16\rme^{-2 \zeta q}$,
where $\zeta := \pi K'_*/2K_*$.
Finally, we obtain that
\begin{equation}\label{eq:Delta1_odd_asymp}
\Delta^{(p,q)}_1 \asymp \frac{8\pi^2  \lambda_*}{m_*K_*^2} q^2
\rme^{-2 \zeta q}, \mbox{ as $p/q \to \omega_*$ and $q$ is odd}.
\end{equation}

If $q$ is even, then $\cn^2(t+q\delta/2,m)=\cn^2(t,m)$ and
\[
L_1^{(p,q)}(t) = 4  \lambda \sum_{j=0}^{q/2-1} \cn^2(t+j\delta,m),
\]
so $L_1^{(p,q)}(t)$ has periods $4K/q$ and $2K'\rmi$.
In this case,
\begin{equation}\label{eq:Delta1_even_asymp}
\Delta^{(p,q)}_1 \asymp
\frac{4\pi^2  \lambda_*}{m_*K_*^2} q^2 \rme^{-\zeta q},
\mbox{ as $p/q \to \omega_*$ and $q$ is even}.
\end{equation}

Next, we study the case $\omega_* = 0$,
when the $(p,q)$-periodic orbits approach the boundary.
In this case,
\[
\lambda_* = 0, \quad
m_* = 1-(b/a)^ 2, \quad
\zeta = \zeta(0,a,b) = \frac{\pi K'(1-(b/a)^2)}{2K(1-(b/a)^2)}.
\]
Since $\lambda_* = 0$,
we need the asymptotic behavior of the caustic parameter $\lambda$
as $p/q \to 0^+$.
We recall that $\lambda \asymp \Xi p/q$ in that case,
where $\Xi=\Xi(a,b)$ is the integral defined in~(\ref{eq:Xi}).
Hence,
\[
\Gamma_4 = \frac{4\pi^2\lambda_*}{m_*K_*^2},\qquad
\Theta_4 = \frac{4\pi^2\Xi(a,b)}{(1-(b/a)^2) K(1-(b/a)^2)^2},
\]
and this ends the proof of the lemma.
\end{proof}

\begin{lem}
The following properties hold for $\mu_1(\varphi) = -\sin \varphi$.
\begin{enumerate}
\item
If $q$ is even, then $L_1^{(p,q)}(t) \equiv 0$ and $\Delta_1^{(p,q)} = 0$.
\item
If $q$ is odd, then $L_1^{(p,q)}(t)$ has just two real critical points
(modulo its periodicity), none of them degenerate.
\item
Let $\zeta(\omega_*,a,b)$ be the exponent introduced in Lemma~\ref{lem:Cosine2}.
If $q$ is even, then there exists $\Omega_3 = \Omega_3(\omega_*,a,b,p,q) > 0$
such that
\[
\Delta_1^{(p,q)} \asymp \Omega_3 \rme^{-\zeta q},\qquad
p/q \to \omega_* \in \{0\} \cup \big((0,1)\setminus \Qset\big).
\]
\item
There exist $\Gamma_3 = \Gamma_3(\omega_*,a,b) > 0$ and
$\Theta_3 = \Theta_3(a,b) > 0$ such that
\[
\Omega_3(\omega_*,b,a,p,q)=
\begin{cases}
\Gamma_3 q, & \mbox{if $\omega_* \in(0,1) \setminus \Qset$,}\\
\Theta_3 p, & \mbox{if $\omega_* = 0$.} \\
\end{cases}
\]
\end{enumerate}
\end{lem}

\begin{proof}
If $q$ is even, then $p$ is odd, $\sn(t+\delta/2) = -\sn t$, and
$L_1^{(p,q)}(t) = -2\lambda \sum_{j=0}^{q-1} \sn(t+j \delta) \equiv 0$.

The case odd $q$ follows the lines of the proof of Lemma~\ref{lem:Cosine2}.
The constants are
\[
\Gamma_3 = \frac{8\pi \lambda_*}{\sqrt{m_*} K_*},\qquad
\Theta_3 = \frac{8\pi \Xi(a,b)}{(1-(b/a)^2)^{1/2}K(1-(b/a)^2)},
\]
where $C_{\lambda_*}$ is the elliptic caustic with rotation number $\omega_*$,
$m_* = (a^2-b^2)/(a^2-\lambda_*^2)$, and $K_* = K(m_*)$.
We omit the details.
\end{proof}

Next, we relate the original perturbed ellipses~(\ref{eq:ModelTables})
written in Cartesian coordinates,
to the perturbed ellipses~(\ref{eq:EllipticPerturbation})
written in elliptic coordinates.

\begin{lem}
Set $0 < b < a$.
\begin{enumerate}
\item
The perturbed ellipse~(\ref{eq:EllipticPerturbation}) with
$\mu_1(\varphi) = -\sin \varphi$ has, up to terms of second order in $\epsilon$,
the implicit equation
\[
\frac{x^2}{a^2} + \frac{(y-\epsilon b^2/a)^2}{b^2} +
2\frac{a^2-b^2}{b^4} \epsilon y^3  = 1.
\]
\item
The perturbed ellipse~(\ref{eq:EllipticPerturbation}) with
$\mu_1(\varphi) = \cos^2 \varphi$ has,
up to terms of second order in $\epsilon$, the implicit equation
\[
\frac{x^2}{\alpha^2} + \frac{y^2}{\beta^2} + 2\frac{a^2-b^2}{b^5} \epsilon y^4 = 1,
\]
for some semi-lengths $\alpha = a + \Order(\epsilon)$ and
$\beta  = b + \Order(\epsilon)$.
\end{enumerate}
\end{lem}

\begin{proof}
It is a tedious, but straightforward, computation.
\end{proof}

Finally, we get the claims stated in Proposition~\ref{prop:melnikov_boundary}
from the previous results by using that
$\alpha = a + \Order(\epsilon)$ and $\beta = b + \Order(\epsilon)$ and
by taking $a=1$.
To be precise, then
\begin{align}\label{eq:M3_M4}
\nonumber
c   &= c(b) = \zeta(0,1,b) = \frac{\pi K'(1-b^2)}{2 K(1-b^2)},\\
M_3 &= M_3(b) = \frac{b^4 \Theta_3(1,b)}{2(1-b^2)} =
\frac{4 \pi b^4 \Xi(1,b)}{(1-b^2)^{3/2} K(1-b^2)},\\
\nonumber
M_4 &= M_4(b) = \frac{b^5 \Theta_4(1,b)}{2(1-b^2)} =
\frac{2 \pi^2 b^5 \Xi(1,b)}{(1-b^2)^2 K(1-b^2)^2},
\end{align}
where the elliptic integral $\Xi=\Xi(a,b)$ is defined in~(\ref{eq:Xi}).

\section{Proof of Proposition~\ref{prop:deltaEllipse}} \label{app:proofdeltaEllipse}

We parameterize the ellipse by using the angular variable $\varphi$.
That is, we use the parametrization
$\sigma(\varphi)=(\cos\varphi,b\sin\varphi)$. The curvature of the
ellipse $E$ at the point $\sigma(\varphi)$ is
\[
\kappa(\varphi) = \frac{b}{ (\sin^2\varphi + b^2\cos^2\varphi)^{3/2}}
=\frac{1}{ b^2 (1+\nu\sin^2\varphi)^{3/2}},
\]
where $\nu=(1-b^2)/b^2>0$.
The arc-length parameter~$s$ and the angular parameter~$\varphi$ are related
by
\[
\frac{\rmd s}{\rmd \varphi}(\varphi) = \|\sigma'(\varphi)\| = \sqrt{
\sin^2\varphi + b^2\cos^2\varphi} = b\sqrt{1+\nu\sin^2\varphi}.
\]
First, we compute the constant
\begin{align*}
C &=
\int_E \kappa^{2/3}\rmd s =
4 b^{-1/3} \int_0^{\pi/2} (1+\nu\sin^2\varphi)^{-1/2}\rmd \varphi \\ &=
4 b^{-1/3}K(-\nu) =
4 b^{2/3} K(1-b^2).
\end{align*}
We have used~\cite[17.4.17]{AbramowitzStegun1964} in the last equality.

The \emph{incomplete elliptic integral of the first kind} with
amplitude~$\varphi\in (0,\pi/2)$ and parameter~$m\in(0,1)$ is
\[
F(\varphi|m) =\int_0^\varphi (1-m\sin^2\theta)^{-1/2}\rmd \theta.
\]
This definition can be extended to
complex amplitudes and any real parameter~\cite{AbramowitzStegun1964}.
Note that $F(\pi/2|m)=K(m)$.

The curvature~$\kappa(\varphi)$ has no complex zeros but has
complex singularities at the points such that $\sin^2\varphi= -1/\nu$.
This equation becomes $\sinh^2\psi = 1/\nu$
under the change~$\varphi=\rmi\psi$.
Let $\psi_*$ be the only positive solution of the previous equation.
Any singularity of $\kappa(\varphi)$ has the form
\[
\varphi= \varphi^\pm_n :=\pm\rmi\psi_*+n\pi,\qquad n\in\Zset.
\]

Let $\xi^\pm_n$ be the complex singularity of $\kappa(\xi)$
associated to $\varphi^\pm_n$ through the change of variables
\[
\xi =
C^{-1}\int_0^{s} \kappa^{2/3}(t) \rmd t =
C^{-1}\int_0^{\varphi}
\kappa^{2/3}(\theta)\frac{\rmd s}{\rmd\varphi}(\theta) \rmd \theta.
\]
The complex path in this integral is the segment from 0 to $\varphi$.

Next, we compute the complex singularities~$\xi^+_n$:
\begin{equation*}
\begin{aligned}
\xi^+_n &=
C^{-1}\int_0^{\varphi^+_n} \kappa^{2/3}(\theta)
\frac{\rmd s}{\rmd \varphi}(\theta) \rmd\theta \\
&=C^{-1}b^{-1/3} F(\rmi\psi_*+n\pi|-\nu) \\
&=2n C^{-1}b^{-1/3}K(-\nu) + \rmi C^{-1}b^{-1/3}F(\pi/2 | b^2)\\
&=2n C^{-1}b^{2/3}K(1-b^2) + \rmi C^{-1}b^{2/3} K(b^2) \\
&=n/2 + \rmi C^{-1}b^{2/3} K'(1-b^2).
\end{aligned}
\end{equation*}
By symmetry, $\xi^-_n = -\xi^+_{-n}$.
We have used formula~\cite[17.4.3]{AbramowitzStegun1964} to compute
$F(\rmi\psi_*+n\pi|-\nu)$, formula~\cite[17.4.8]{AbramowitzStegun1964} to compute
$F(\rmi\psi_*|-\nu)$, and formula~\cite[17.4.15]{AbramowitzStegun1964} to compute
$F(\pi/2 | b^2)$.

Therefore, the distance~$\delta$ of the set of
singularities and zeros of the curvature
$\kappa(\xi)$ to the real axis is
\[
\delta= C^{-1}b^{2/3} K'(1-b^2) =\frac{K'(1-b^2)}{4K(1-b^2)} = c/2\pi.
\]

\section{Proof of Proposition~\ref{prop:deltaPerturbedCircle}}
\label{app:proofCircularBehaviour}

Fix the integer $n \ge 3$.
We consider the perturbed circles
\begin{equation}\label{eq:PerturbedCircle}
Q =
\left\{ (x,y) \in \Rset^2 : x^2 + y^2 + \epsilon y^n = 1 \right\}
\end{equation}
where $0 < \epsilon \ll 1$ is a small perturbative parameter.

Let $C = C(\epsilon)$ be the constant defined in~(\ref{eq:xiDefinition}).
If $\epsilon = 0$,
then $Q$ is a circle of radius one with curvature $\kappa \equiv 1$,
so
\[
C(0) = \int_Q \kappa^{2/3} \rmd s = \int_Q \rmd s = \Length(Q) = 2\pi.
\]
We note that~(\ref{eq:PerturbedCircle}) is a smooth perturbation of
a circle of radius one, so $C(\epsilon)$ is smooth at $\epsilon = 0$ and
\begin{equation}\label{eq:Cepsilon}
C = C(\epsilon) = C(0) + \Order(\epsilon) = 2\pi + \Order(\epsilon).
\end{equation}

We introduce the polynomial $r(y) = 1 - y^2 - \epsilon y^n$. Note
that $(x,y) \in Q$ if and only if $x^2 = r(y)$. By taking derivatives
twice with respect to $y$ the implicit relation $x^2 = r(y)$, we get
the auxiliary polynomials
\begin{align*}
p(y)
&= -x^3 \frac{\rmd^2 x}{\rmd y^2}
 = \left( \frac{r'(y)}{2} \right)^2 - \frac{r(y) r''(y)}{2} \\
&= 1 + \epsilon p_{n-2} y^{n-2} + \epsilon p_n y^n + \epsilon^2 p_{2n-2} y^{2n-2}, \\
q(y)
&= x^2 + \left(x \frac{\rmd x}{\rmd y} \right)^2
 = r(y) + \left( \frac{r'(y)}{2} \right)^2 \\
&= 1 + \epsilon q_n y^n + \epsilon^2 q_{2n-2} y^{2n-2},
\end{align*}
whose coefficients are $p_{n-2} = n(n-1)/2$, $p_n = -(n-1)(n-2)/2$,
$p_{2n-2} = -n(n-2)/4$, $q_n = n-1$, and $q_{2n-2} = n^2/4$.
The length element and the curvature at the point $(x,y) \in Q$ are
\begin{align*}
\rmd s
&=
\sqrt{1+ \left( \frac{\rmd x}{\rmd y}\right)^2} \rmd y =
\sqrt{\frac{q(y)}{r(y)}} \rmd y,\\
\kappa
&=
- \frac{\rmd^2 x}{\rmd y^2}
   \left(1+ \left(\frac{\rmd x}{\rmd y}\right)^2\right)^{-3/2} =
\frac{p(y)}{q^{3/2}(y)}.
\end{align*}
The curvature should be positive,
which explains the minus sign in the formula for $\kappa(y)$.
Thus, we can relate any singularity (or any zero) $y_\star \in \Cset$ of
the curvature $\kappa(y)$, with the corresponding singularities
(or zeros) $s_\star \in \Cset/l\Zset$ and $\xi_\star \in \Cset/\Zset$
by means of the formula
\[
\xi_\star = \int_0^{s_\star} \kappa^{2/3}(s) \rmd s =
\int_0^{y_\star} g(y) \rmd y,
\]
where
\[
g(y) :=  \kappa^{2/3}(y) \frac{\rmd s}{\rmd y}(y) =
\frac{p^{2/3}(y)}{\sqrt{r(y) q(y)}}.
\]
Let $\mathcal{R} \subset \Cset$ be the union of the complex rays $\{
\alpha y_0: \alpha \ge 0 \}$, where $y_0$ is a root of $p(y)$,
$q(y)$ or $r(y)$. The function $g(y)$ is analytic in $\Cset
\setminus \mathcal{R}$, so we will avoid the set $\mathcal{R}$ when
computing the  integral $\int_0^{y_\star} g(y) \rmd y$ along complex
paths.

\begin{lem}
Let $0 < \epsilon \ll 1$ and $n \in \Nset$ with $n \ge 3$.
\begin{enumerate}
\item
The polynomial $p(y)$ has $n$ roots of the form
\[
z \epsilon^{-1/n} + \Order(\epsilon^{1/n}), \qquad
z^n = 2/((n-1)(n-2));
\]
and $n-2$ roots of the form
\[
z \epsilon^{-1/(n-2)} + \Order(\epsilon^{1/(n-2)}),\qquad
z^{n-2} = -(n-1)(n-2)/n.
\]

\item
The polynomial $q(y)$ has $n$ roots of the form
\[
z \epsilon^{-1/n} + \Order(\epsilon^{1/n}),\qquad z^n = -1/(n-1);
\]
and $n-2$ roots of the form
\[
z \epsilon^{-1/(n-2)} +  \Order (\epsilon^{1/(n-2)}),\qquad
z^{n-2} = -4(n-1)/n^2.
\]

\item
The polynomial $r(y)$ has $n-2$ roots of the form
\[
z \epsilon^{-1/(n-2)} + \Order(\epsilon^{1/(n-2)}),\qquad
z^{n-2} = -1;
\]
and two real roots of the form $y_\pm = \pm 1 + \Order(\epsilon)$.
\end{enumerate}
Besides, each one of these roots depends on some positive fractional
power of $\epsilon$ in an analytic way.
\end{lem}

\proof
If $w_0(z)$ is a polynomial with a simple root $z_0$
and $w_1(z)$ is another polynomial,
then $w(z) = w_0(z) + \mu w_1(z)$ has some root of the form
$z = z_0 + \Order(\mu)$ which depends analytically on $\mu$.
The roots $y_\pm = \pm 1 + \Order(\epsilon)$ of the polynomial
$r(y) = 1 - y^2 - \epsilon y^n$ are obtained directly
with $w_0(z) = 1-z^2$, $w_1(z) = -z^n$, and $\mu = \epsilon$.

If we take $\mu = \epsilon^{2/n}$, then
\begin{align*}
p(\epsilon^{-1/n} z) &=
1 + p_n z^n + \mu (p_{n-2} z^{n-2} + p_{2n-2} z^{2n-2}), \\
\nonumber
q(\epsilon^{-1/n} z) &=
1 + q_n z^n + \mu q_{2n-2} z^{2n-2},
\end{align*}
and we find the $n$ roots with an $\Order(\epsilon^{1/n})$-modulus
of $p(y)$ and the $n$ roots with an $\Order(\epsilon^{1/n})$-modulus of $q(y)$.

If we take $\mu = \epsilon^{2/(n-2)}$, then
\begin{align*}
\mu p(\epsilon^{-1/(n-2)} z) &=
z^n (p_n + p_{2n-2} z^{n-2}) + \mu (1 + p_{n-2} z^{n-2}), \\
\mu q(\epsilon^{-1/(n-2)} z) &=
z^n (q_n + q_{2n-2} z^{n-2}) + \mu, \\
\mu r(\epsilon^{-1/(n-2)} z) &=
-z^2(1+z^{n-2}) + \mu,
\end{align*}
and we find the $n-2$ roots with an $\Order(\epsilon^{1/(n-2)})$-modulus
of $p(y)$, the $n-2$ roots with an $\Order(\epsilon^{1/(n-2)})$-modulus
of $q(y)$, and the $n-2$ roots with an $\Order(\epsilon^{1/(n-2)})$-modulus
of $r(y)$.
\qed

\begin{lem}\label{lem:eta}
If $0 < \epsilon \ll 1$, $n \in \Nset$ with $n \ge 3$,
and $y_\star \in \Cset$ is a root of $p(y)$ or $q(y)$ with an
$\Order(\epsilon^{-1/n})$-modulus,
then there exists a constant $\eta_\star \in \Rset$
such that
\[
|\Im \xi_\star| =
\frac{|\log \epsilon|}{n} + \eta_\star + \Order(\epsilon^{2/n} \log \epsilon),
\]
as $\epsilon \to 0^+$.
\end{lem}

\proof
For simplicity,
we assume that $y_\star$ is a root of $q(y)$ such that
$\Re y_\star \le 0$ and $\Im y_\star \ge 0$.
Other cases require minor changes.

If $r_0 = (n-1)^{-1/n}/2$, $r_\star = \epsilon^{1/n} |y_\star|$,
and $\theta_\star  = \arg y_\star$,
then
$\pi/2 \le \theta_\star < n \pi/(n+1)$ and
$r_\star = 2r_0 + \Order(\epsilon^{2/n})$,
because $y_\star = \epsilon^{-1/n}z + \Order(\epsilon^{1/n})$
for some $z \in \Cset$ such that $z^n = -1/(n-1) < 0$.

We compute $\xi_\star = \int_0^{y_\star} g(y) \rmd y$ by integrating over
the path $\sigma_\star = \sigma_1 \cup \sigma_2 \cup \sigma_3$,
where
\begin{align*}
\sigma_1 &=
\{ \epsilon^{-1/n} \rmi t : 0 \le t \le r_0 \}, \\
\sigma_2 &=
\{ \epsilon^{-1/n} r_0 \rme^{\theta \rmi} : \pi/2 \le \theta \le \theta_\star \}, \\
\sigma_3 &=
\{ \epsilon^{-1/n} \rme^{\theta_\star \rmi} r : r_0 \le r \le r_\star \}.
\end{align*}
This path only intersects the set of rays $\mathcal{R}$ at
its endpoint $y_\star$, since the $2n$ roots of $p(y)$ and $q(y)$ with an
$\Order(\epsilon^{-1/n})$-modulus have pairwise different arguments
when $\epsilon \to 0^+$.

We write
$\xi_\star =
 \int_0^{y_\star} g(y) \rmd y =
 \int_{\sigma_\star} g(y) \rmd y =
 \xi_1 + \xi_2 + \xi_3$, where
\begin{align*}
\xi_1 &=
\int_{\sigma_1} g(y) \rmd y =
\int_0^{r_0} \epsilon^{-1/n} \rmi g\big(\epsilon^{-1/n} \rmi t \big) \rmd t, \\
\xi_2 &=
\int_{\sigma_2} g(y) \rmd y =
\int_{\pi/2}^{\theta_\star} \epsilon^{-1/n} r_0 \rme^{\theta \rmi} \rmi
                            g\big(\epsilon^{-1/n} r_0 \rme^{\theta \rmi}\big) \rmd \theta, \\
\xi_3 &=
\int_{\sigma_3} g(y) \rmd y =
\int_{r_0}^{r_\star} \epsilon^{-1/n} \rme^{\theta_\star \rmi}
                     g\big(\epsilon^{-1/n} \rme^{\theta_\star \rmi} r\big) \rmd r. \\
\end{align*}

In order to study $\xi_1$, we consider the function
\begin{equation}\label{eq:h_h0}
h(t) :=
\epsilon^{-1/n} {\sqrt{t^2 + \epsilon^{2/n}}} g\big (\epsilon^{-1/n} \rmi t \big) =
h_0(t) + \Order(\epsilon^{2/n} ),
\end{equation}
where
$h_0(t) = (1+p_n \rmi^n t^n)^{2/3}(1+ q_n \rmi^n t^n)^{-1/2}$.
The function $h_0(t)$ is smooth in the interval $[0,r_0]$
and $h_0(t) = 1$.
Besides,
\[
\xi_1 = \rmi \int_0^{r_0} (t^2 + \epsilon^{2/n})^{-1/2} h(t) \rmd t =
 \hat{\xi}_1 + \check{\xi}_1 + \tilde{\xi}_1 + \breve{\xi}_1,
\]
where
\begin{align*}
\hat{\xi}_1  &=
\rmi \int_0^{r_0} \frac{\rmd t}{\sqrt{t^2 + \epsilon^{2/n}}} =
\rmi \Argsinh(\epsilon^{-1/n} r_0) \\
&=\rmi \frac{|\log \epsilon|}{n} + \rmi \log(2r_0) + \Order(\epsilon^{2/n}) , \\
\check{\xi}_1 &= \rmi \int_0^{r_0} \frac{h_0(t)-1}{t} \rmd t, \\
\tilde{\xi}_1 &=
\rmi \int_0^{r_0}
\frac{h_0(t)-1}{t} \left( \frac{t}{\sqrt{t^2 + \epsilon^{2/n}}} - 1\right) \rmd t =
\Order(\epsilon^{2/n} \log \epsilon),\\
\breve{\xi}_1 &=
\rmi \int_0^{r_0} \frac{h(t) - h_0(t)}{\sqrt{t^2 + \epsilon^{2/n}}} \rmd t =
\Order(\epsilon^{2/n} \log \epsilon).
\end{align*}
The integral $\hat{\xi}_1$ is immediate.
The integral $\check{\xi}_1$ does not depend on $\epsilon$.
The integral $\tilde{\xi}_1$ is bounded using ideas from the proof of
Lemma~23 in~\cite{CasasRamirez2011}.
The integral $\breve{\xi}_1$ is bounded using~(\ref{eq:h_h0}).
Hence, we have already seen that there exists $\eta_1 \in \Rset$ such that
\[
| \Im \xi_1 | =
\frac{|\log \epsilon|}{n} + \eta_1 + \Order(\epsilon^{2/n} \log \epsilon ).
\]

The study of $\xi_2$ and $\xi_3$ is easier, because
\[
\xi_2 = \check{\xi}_2 + \Order(\epsilon^{2/n}),\qquad
\xi_3 = \check{\xi}_3 + \Order(\epsilon^{2/n}),
\]
for some constants $\check{\xi}_2$ and $\check{\xi}_3$
that do not depend on $\epsilon$.

For instance, $\xi_2$ depends on $\epsilon^{2/n}$ in an analytic way,
because the integrand $\epsilon^{-1/n} r_0 \rme^{\theta \rmi} \rmi
 g\big(\epsilon^{-1/n} r_0 \rme^{\theta \rmi}\big)$ and the argument
$\theta_\star$ are both analytic in $\epsilon^{2/n}$,
and all the singularities of the integrand are far from the integration path.
The study of $\xi_3$ is similar.
\qed

Finally, if $\delta$ is the distance of the set of singularities and
zeros of the curvature $\kappa(\xi)$ to the real axis, then
\begin{align*}
2\pi \delta &=
\frac{2\pi}{C} \min \left \{ |\Im \xi_\star| :
\mbox{$y_\star$ is a root with an $\Order(\epsilon^{1/n})$-modulus}
\right\} \\
&= \frac{|\log \epsilon|}{n} + \eta + \Order(\epsilon^{2/n} \log \epsilon),
\end{align*}
where the constant $\eta = \eta_n \in \Rset$ is equal to the smallest
constant $\eta_\star$ provided by Lemma~\ref{lem:eta}
among all the roots of $p(y)$ and $q(y)$ with an $\Order(\epsilon^{1/n})$-modulus.
We have also used relation~(\ref{eq:Cepsilon}) in the last equality.

We do not care about the roots $y_\pm = \pm 1 + \Order(\epsilon)$
of $r(y)$, since they correspond to points where $y$ is not a true coordinate
over the perturbed circle $Q$.
To be precise, the points $(0,y_\pm)$ are the two vertices of $Q$
over the symmetry line $\{ x= 0 \}$, and the curvature
has a finite positive value at them.
Nor do we care about the roots whose modulus is
$\Order(\epsilon^{-1/(n-2)})$, because
\[
\epsilon^{-1/(n-2)} g(\epsilon^{-1/(n-2)} z) =
 \epsilon^{-1/(3n-6)} \big( l_0(z) + \order(1) \big),
\]
where
\[
l_0(z) =
\frac{z^{n/6-1}(p_n + p_{2n-2} z^{n-2})^{2/3}}
     {(1+z^{n-2})^{1/2} (q_n + q_{2n-2}z^{n-2})^{1/2}}.
\]
This implies that,
if $y_\star \in \Cset$ is one of those farther roots,
then
\[
| \Im \xi_\star | = \epsilon^{-1/(3n-6)} \big( \nu_\star + \order(1) \big)
\]
for some constant $\nu_\star \in \Rset$.
That is, the farther roots give rise to much bigger imaginary parts.

\addcontentsline{toc}{section}{References}
\bibliographystyle{ieeetr}
\bibliography{mybibliography}

\end{document}